\documentclass[11pt]{article}
\usepackage{amsthm}
\usepackage{amssymb}
\usepackage{amsmath}
\usepackage{hyperref}
\usepackage{graphicx}
\usepackage{caption}
\usepackage{array}
\usepackage{enumitem}
\usepackage{mathtools}
\usepackage{etoolbox}
\usepackage{bbm}
\usepackage{mathdots}
\usepackage[backend=biber, style=alphabetic, maxbibnames=9, maxcitenames=9, maxalphanames=9]{biblatex}
\bibliography{Residual}
\usepackage[all,cmtip]{xy}
\usepackage{tensor}

\pretocmd{\section}{%
  }{}{}

\numberwithin{table}{section}

\newtheorem{theorem}{Theorem}[subsection]
\newtheorem*{theorem*}{Theorem}
\newtheorem{proposition}[theorem]{Proposition}
\newtheorem*{proposition*}{Proposition}
\newtheorem{corollary}[theorem]{Corollary}
\newtheorem{lemma}[theorem]{Lemma}

\newtheorem*{lemma*}{Lemma}

\newtheorem{introthm}{Theorem}[section]

\theoremstyle{definition}

\newtheorem*{exercise*}{Exercise}
\newtheorem{remark}[theorem]{Remark}

\newtheorem{example}[theorem]{Example}
\newtheorem{setup}[theorem]{Setup}

\numberwithin{equation}{subsection}
\numberwithin{figure}{section}

\let\hom\relax
\let\det\relax

\newcommand{\mb}{\mathbb}
\newcommand{\mbf}{\mathbf}
\newcommand{\mc}{\mathcal}
\newcommand{\mf}{\mathfrak}
\newcommand{\mr}{\mathrm}
\newcommand\sm[1]{{\tiny\arraycolsep=0.3\arraycolsep\ensuremath{\begin{pmatrix}#1\end{pmatrix}}}}
\newcommand\pmat[1]{\begin{pmatrix}#1\end{pmatrix}}

\newcommand\tp[1]{\prescript{t}{}{#1}}

\newcommand{\Z}{\mathbb{Z}}
\newcommand{\Q}{\mathbb{Q}}
\newcommand{\R}{\mathbb{R}}
\newcommand{\C}{\mathbb{C}}

\newcommand{\A}{\mathbb{A}}

\renewcommand{\sl}{\mathfrak{sl}}

\newcommand{\Sset}[2]{\left\lbrace{#1}\,\,\middle|\,\,{#2}\right\rbrace}
\newcommand{\sset}[2]{\lbrace{#1}\,\,|\,\,{#2}\rbrace}
\newcommand{\Set}[1]{\left\lbrace{#1}\right\rbrace}

\DeclareMathOperator{\Ad}{Ad}

\DeclareMathOperator{\det}{det}
\DeclareMathOperator{\diag}{diag}

\DeclareMathOperator{\disc}{disc}

\DeclareMathOperator{\hom}{Hom}
\newcommand{\id}{\mathrm{id}}

\DeclareMathOperator{\Ind}{Ind}

\DeclareMathOperator{\Lie}{Lie}

\newcommand{\modulo}[1]{\,\,(\mathrm{mod}\,\,{#1})}

\DeclareMathOperator{\pr}{pr}

\DeclareMathOperator{\re}{Re}

\DeclareMathOperator*{\res}{Res}

\DeclareMathOperator{\sign}{sign}

\DeclareMathOperator{\sph}{sph}

\DeclareMathOperator{\std}{Std}

\DeclareMathOperator{\Sym}{Sym}

\begin{document}
\title{A vanishing theorem for residual Eisenstein cohomology}
\author{Sam Mundy}
\date{}
\maketitle

\begin{abstract}
We study the residual Eisenstein cohomology of semisimple groups in the context of maximal parabolic subgroups which remain maximal over $\R$. Under certain general hypotheses, we show that these residual representations are cohomological one degree below middle, and one above; however, the classes above middle vanish in the full automorphic cohomology.

We offer two proofs of this vanishing, one quick and one more explicit. The latter of the two proofs finds an explicit cochain which provides a primitive to the image of a nontrivial class from the cohomology of such a residual representation. This cochain is valued in regular Eisenstein series. Along the way, we study in detail the archimedean component of the relevant induced representation. In particular, we prove that it has a subrepresentation which is the sum of two discrete series, whose Harish-Chandra parameters we describe, and that the intertwining operator vanishes to order exactly $1$ on that subrepresentation.
\end{abstract}

\section*{Introduction}

Let $\mc{G}$ be a connected, semisimple, linear algebraic group over $\Q$. This paper aims to use explicit cochain constructions in relative Lie algebra cohomology to study a part of the automorphic cohomology of $\mc{G}$ which has generally remained rather mysterious, namely that part which is spanned by residual Eisenstein series. Our main theorem amounts to giving sufficient conditions for certain nontrivial direct summands of the $L^2$-cohomology of $\mc{G}$ to vanish in the full automorphic cohomology of $\mc{G}$, and for others to persist. Let us begin by recalling the relevant notions.

Let $E$ be an irreducible, finite dimensional representation of $\mc{G}(\C)$. Let $\mf{g}$ be the complexified Lie algebra of $\mc{G}$, and $U(\mf{g})$ its universal enveloping algebra. Then the annihilator of $E$ in the center of $U(\mf{g})$ is an ideal, which we denote by $J_E$. Inside the space of all automorphic forms on $\mc{G}(\A)$, we can define $\mc{A}_E(\mc{G})$ to be the subspace of automorphic forms which are annihilated by a power of $J_E$. Fix a maximal compact subgroup $K_\infty$ in $\mc{G}(\R)$, and write $K_\infty^\circ$ for the connected component of the identity in $K_\infty$. We define the \textit{automorphic cohomology} of $\mc{G}$ with respect to the dual $E^\vee$ of $E$ to be
\[H^*(\mc{G},E^\vee)=H^*(\mf{g},K_\infty^\circ;\mc{A}_E(\mc{G})\otimes E^\vee).\]
The right hand side is the definition of the left and it denotes the usual $(\mf{g},K_\infty^\circ)$-cohomology of the $(\mf{g},K_\infty)$-module $\mc{A}_E(\mc{G})\otimes E^\vee$; we review briefly this notion in Section \ref{subsecgkcoh}. The automorphic cohomology $H^*(\mc{G},E^\vee)$ carries the natural structure of a $\mc{G}(\A_f)$-module (where $\A_f$ denotes the ring of finite adeles) with a natural commuting action of the component group $\pi_0(K_\infty)$ of $K_\infty$.

The space $H^*(\mc{G},E^\vee)$ is primarily of interest because it is isomorphic, by a deep theorem of Franke \cite[Theorem 18]{franke}, to the cohomology of a local system associated with $E^\vee$ on the locally symmetric spaces associated with $\mc{G}$ in the inductive limit over all level structures. We will not use this point view in this paper, however. But the techniques used to prove Franke's theorem have also led to a better understanding of the structure of the space $\mc{A}_E(\mc{G})$, particularly in ways that are helpful for computing the automorphic cohomology of $\mc{G}$; see, for example, \cite{FS}. The space $\mc{A}_E(\mc{G})$ decomposes into cuspidal and Eisenstein parts, and then the Eisenstein part decomposes much further into individual direct summands. If the Eisenstein series contributing to a particular summand are all regular, then this summand has a particularly simple parabolically induced structure. Otherwise it is an iterated extension of such. This paper is concerned with the cohomology of these summands in the simplest case when such a summand is not merely a parabolically induced representation; see Theorem \ref{thmgrbac} of Grbac below for a precise description of this structure in the case of interest for us.

Now let $\mc{P}$ be a maximal parabolic $\Q$-subgroup of $\mc{G}$ with Levi factor $\mc{M}$ and unipotent radical $\mc{N}$. Then given a unitary cuspidal automorphic representation $\pi$ of $\mc{M}(\A)$ and $s\in\C$, we can build the induced representation
\[\Ind_{\mc{P}(\A)}^{\mc{G}(\A)}(\pi,s);\]
here the transformation law is normalized so that for $\phi:\mc{G}(\A)\to\pi$ in this induced representation, we have
\[\phi(mng)=\delta_{\mc{P}}^{s+1/2}(m)\pi(m)\phi(g),\qquad m\in\mc{M}(\A),\, n\in\mc{N}(\A),\, g\in\mc{G}(\A),\]
where $\delta_{\mc{P}}$ denotes the usual modulus character for $\mc{P}(\A)$. Putting such a $\phi$ into a flat family of sections $\phi_s$ as $s$ varies, we can build an \textit{Eisenstein series} $E(\phi,s,g)$ for any $s\in\C$ and $g\in\mc{G}(\A)$; we recall the theory in this case in Section \ref{subseceisenstein} below. The Eisenstein series $E(\phi,s,g)$ is defined by meromorphic continuation, and so it might have a pole at some $s$ for given $\phi$ and $g$. But if $\re(s)\geq 0$, then at least the pole is simple and taking the residue gives a \textit{residual Eisenstein series}. In this case, the automorphic representation spanned by such residues at a fixed $s$ as $\phi\in \Ind_{\mc{P}(\A)}^{\mc{G}(\A)}(\pi,s)$ varies will be denoted $\mc{L}(\pi,s)$, and it is isomorphic to the unique irreducible quotient of $\Ind_{\mc{P}(\A)}^{\mc{G}(\A)}(\pi,s)$.

We assume throughout that
\begin{equation}
\label{eqnassumption1}
\mc{P}(\R)\textrm{ is maximal in }\mc{G}(\R),\tag{A.1}
\end{equation}
and further that
\begin{equation}
\label{eqnassumption2}
\textrm{Both }\mc{G}(\R)\textrm{ and }\mc{M}(\R)\textrm{ have discrete series.}\tag{A.2}
\end{equation}
In this case we can define a particular value $s_0>0$ which, in practice, seems to frequently give rise to residual representations $\mc{L}(\pi,s_0)$ like above which have $(\mf{g},K_\infty^\circ)$-cohomology. We will give a few different interpretations of this number in this paper; here is one. The condition that \eqref{eqnassumption1} holds implies that the $\R$-split center, call it $A$, of $\mc{M}(\R)$ is one dimensional. Let $\mf{a}$ be the complexified Lie algebra of $A$, and let $\mf{m}$ and $\mf{n}$ be those of $\mc{M}$ and $\mc{N}$, and complete $\mf{a}$ to a Cartan subalgebra $\mf{h}$ of $\mf{g}$ contained in $\mf{m}$. Then it follows from Lemma \ref{lemrootalpha} below that, under the assumption \eqref{eqnassumption2} there is a unique root $\alpha$ of $\mf{h}$ which is contained in $\mf{n}$ and which vanishes on the orthogonal complement of $\mf{a}$ in $\mf{h}$ under the Killing form. We will want to twist by $\frac{1}{2}\alpha$, so we let $s_0$ be defined by the equation
\[\frac{1}{2}\alpha|_{\mf{a}}=2s_0\rho_{\mc{P}}|_{\mf{a}},\]
where $\rho_{\mc{P}}$ denotes the half sum of roots of $\mf{h}$ in $\mf{n}$ with multiplicity.

Now fix a cuspidal automorphic representation $\pi$ of $\mc{M}(\A)$ such that
\begin{equation}
\label{eqnassumption3}
\textrm{The central character }\chi\textrm{ of }\pi\textrm{ is trivial on }A^\circ;\tag{A.3}
\end{equation}
This assumption is just an innocuous normalization. Let $\pi_\infty$ be the archimedean component of $\pi$. We assume in addition that
\begin{equation}
\label{eqnassumption4}
\pi_\infty\textrm{ is discrete series,}\tag{A.4}
\end{equation}
and that there is an irreducible finite dimensional representation $E$ of $\mc{G}(\C)$ such that
\begin{equation}
\label{eqnassumption5}
\textrm{The infinitesimal character of }\Ind_{\mc{P}(\R)}^{\mc{G}(\R)}(\pi_\infty,s_0)\textrm{ matches that of }E.\tag{A.5}
\end{equation}
This assumption guarantees that the Eisenstein series, or their possible residues, defined by sections $\phi$ in $\Ind_{\mc{P}(\A)}^{\mc{G}(\A)}(\pi,s_0)$ are in the space $\mc{A}_E(\mc{G})$. We then have the following theorem.

\begin{introthm}
\label{introthmmain}
Let the notation be as above, including assumptions \eqref{eqnassumption1}-\eqref{eqnassumption5}. Assume in addition that:
\begin{enumerate}[label=(\roman*)]
\item There is a flat section
\[\phi\in\Ind_{\mc{P}(\A)}^{\mc{G}(\A)}(\pi,s)\]
such that the Eisenstein series $E(\phi,s,g)$ has a pole at $s=s_0$.
\item The corresponding residual representation $\mc{L}(\pi,s_0)$ has that
\[H^*(\mf{g},K_\infty^\circ;\mc{L}(\pi,s_0)\otimes E^\vee)\]
is nonvanishing.
\end{enumerate}
Let us write $d=\frac{1}{2}\dim(\mc{G}(\R)/K_\infty^\circ)$ (which is middle degree). Then:
\begin{enumerate}[label=(\alph*)]
\item We have that there is an integer $m>0$ such that
\[\dim_{\C}H^q(\mf{g},K_\infty^\circ;\mc{L}(\pi,s_0)_\infty\otimes E^\vee)=\begin{cases}
m&\textrm{if }q=d-1\textrm{ or }d+1;\\
0&\textrm{otherwise}.
\end{cases}\]
\item The natural map 
\[H^{d-1}(\mf{g},K_\infty^\circ;\mc{L}(\pi,s_0)\otimes E^\vee)\to H^{d-1}(G,E^\vee),\]
induced by including $\mc{L}(\pi,s_0)$ into $\mc{A}_E(\mc{G})$, is injective.
\item However, the natural map 
\[H^{d+1}(\mf{g},K_\infty^\circ;\mc{L}(\pi,s_0)\otimes E^\vee)\to H^{d+1}(G,E^\vee),\]
again induced by including $\mc{L}(\pi,s_0)$ into $\mc{A}_E(\mc{G})$, is zero.
\end{enumerate}
\end{introthm}

Thus this says that the \textit{residual Eisenstein cohomology} coming from $\pi$ is concentrated in degree $d-1$, despite the residual representation $\mc{L}(\pi,s_0)$ having cohomology in degree $d+1$ as well. Part (c) above is the vanishing theorem alluded to in the title. We offer two proofs of it. The first such proof is relatively short once everything is set up, and it mimics a method employed in \cite{RSson1}. The second proof we give is much more explicit than the first, and it is the main novelty in this paper. Parts (a) and (b) have been observed in many cases and our main contribution to them here has been to put a general framework around their occurrence.

Let us give a few remarks about the hypotheses. First, we have assumed throughout that $\mc{G}$ is semisimple instead of merely reductive. This is simply to avoid technical, but surmountable, issues involving the possible presence of a nonsplit center in $\mc{G}$. More precisely, assume for a moment that $\mc{G}$ were instead reductive, and assume moreover that the center $Z_{\mc{G}}$ of $\mc{G}$ were to have a torus component $\mc{T}_{ns}$ which is nonsplit over $\Q$ but split over $\R$. This component $\mc{T}_{ns}$ would then contribute an extra factor the automorphic cohomology of $\mc{G}$ that spans several degrees. From the point of view of the spaces involved, the locally symmetric spaces for $\mc{G}$ would become nontrivial fibrations over those for the derived group $\mc{G}^{\mr{der}}$ with fibers given by compact tori. Therefore the statements in the theorem above about degrees $d\pm 1$ become more numerically complicated. However, it is not difficult to obtain precise statements in the reductive case using K\"unneth-type arguments to reduce to the semisimple case which we have covered.

Next we remark that the number $s_0$ is easy to calculate in explicit contexts and is often related to the calculation of the constant terms of the relevant Eisenstein series. Indeed, in general, this number seems to appear quite frequently in the following context. Often, the quotient of $L$-functions appearing via the Langlands--Shahidi method in the constant term of the Eisenstein series $E(\phi,s,g)$ will have a zeta function of the form $\zeta(n_0s)$ as a constituent of the numerator, for some positive integer $n_0$. In cases where this is the only zeta function appearing, one very often has that $s_0=n_0^{-1}$, which often allows $E(\phi,s,g)$ to attain a pole at $s=s_0$ (depending on the nature of the other $L$-values involved).

Finally, we remark on the assumption (ii) in the theorem, namely that $\mc{L}(\pi,s_0)$ is cohomological. This assumption may look potentially difficult to verify in practice, but actually part of this paper is devoted to giving a way to explicitly write down all representations $\pi_\infty$ of $\mc{M}(\R)$ satisfying \eqref{eqnassumption3}-\eqref{eqnassumption5} and (ii), at least when $\mc{G}(\R)$ is connected (the disconnected case can then be treated by restriction to the connected component of the identity in $\mc{G}(\R)$).

We give several examples of situations where the theorem applies, some in quite a bit of detail, in Section \ref{subsecexamples}. And indeed, instead of verifying (ii) directly in these examples, we use the content of Section \ref{secreal} to write down all possible representations which can occur as the archimedean part of $\pi$ and for which the conditions of the theorem can hold.

Theorem \ref{introthmmain} is proved by writing down a specific $d$-cochain in the $(\mf{g},K_\infty^\circ)$-cohomology complex valued in regular Eisenstein series whose image under the boundary map is the $(d+1)$-cocycle which represents an element in the image of the map
\[H^{d+1}(\mf{g},K_\infty^\circ;\mc{L}(\pi,s_0)\otimes E^\vee)\to H^{d+1}(G,E^\vee).\]
To explain how to write down this cochain, we need to describe the structure of the archimedean induced representation $\Ind_{\mc{P}(\R)}^{\mc{G}(\R)}(\pi_\infty,s_0)$. We do this here in the introduction under the simplifying assumption that $\mc{G}(\R)$ is connected; in the paper, Section \ref{secreal} is written for connected groups, and the content of Section \ref{secadele} reduces to this case.

So under the additional assumption that $\mc{G}(\R)$ is connected, instead of starting with a $\pi_\infty$ satisfying \eqref{eqnassumption3}-\eqref{eqnassumption5} along with (ii) of Theorem \ref{introthmmain}, we start with the following in Sections \ref{subsecsetup} and \ref{subsecds}: First we give ourselves
\begin{itemize}
\item A Cartan involution $\theta$ on $\mc{G}(\R)$ giving $K_\infty$;
\item A compact Cartan subgroup $T$ of $\mc{G}(\R)$ contained in $K_\infty$ with complexified Lie algebra $\mf{t}$;
\item A system of positive roots $\Delta^{+}$ in the root system $\Delta(\mf{g},\mf{t})$ of $\mf{t}$ in $\mf{g}$;
\item A simple noncompact root $\alpha_{0}$ in $\Delta^{+}$;
\item A Harish-Chandra parameter $\Lambda$ for $\mf{t}$ in $\mf{g}$ which is dominant for $\Delta^{+}$ and such that
\[\frac{2\langle\Lambda,\alpha_{0}\rangle}{\langle\alpha_{0},\alpha_{0}\rangle}=1.\]
\end{itemize}
Section \ref{subsecsetup} produces from these first four pieces of data a maximal parabolic subgroup of $\mc{G}(\R)$, which we will call $P$, with Levi decomposition $P=MN$, and Section \ref{subsecds} produces for us a discrete series representation of $M$, which we will call $\pi_\infty'$ in this introduction (it is simply called $\pi$ in Section \ref{subsecds}, but we have already used that notation here). We eventually prove in Proposition \ref{propaltsetup} that we can choose this data so that $P$ is conjugate to our $\mc{P}(\R)$ and that $\pi_\infty'$ is conjugate $\pi_\infty$ by the same element of $\mc{G}(\R)$. So let us state the main theorems we prove on the archimedean side for $\mc{P}(\R)$ and $\pi_\infty$ in this introduction.

We then prove the following; see Theorem \ref{thmstdmodforind} below.

\begin{introthm}
\label{introthmstdmodforind}
Continue to assume $\mc{G}(\R)$ is connected. Let $D_+$ be the discrete series representation of $\mc{G}(\R)$ with Harish-Chandra parameter $\Lambda$, and $D_-$ that with Harish-Chandra parameter $\Lambda-\alpha_0$. Let $J$ be the unique irreducible quotient of $\Ind_{\mc{P}(\R)}^{\mc{G}(\R)}(\pi_\infty,s_0)$. Then
\begin{enumerate}[label=(\alph*)]
\item The representation $J$ along with, of course, $D_+$ and $D_-$, are all unitary.
\item We have exact sequences
\begin{equation}
\label{eqnthmBses1}
0\to D_+\oplus D_-\to \Ind_{\mc{P}(\R)}^{\mc{G}(\R)}(\pi_\infty,s_0)\to J\to 0,\tag{SES.1}
\end{equation}
and
\begin{equation}
\label{eqnthmBses2}
0\to J \to \Ind_{\mc{P}(\R)}^{\mc{G}(\R)}(\pi_\infty,-s_0)\to D_+\oplus D_-\to 0.\tag{SES.2}
\end{equation}
\end{enumerate}
\end{introthm}

The proof of this theorem is, more or less, an assembly of facts from the theory of \textit{cohomological induction}; see \cite{knvo} and \cite{voganbook}. In fact, there is a natural $\theta$-stable parabolic subalgebra $\mf{q}$ in $\mf{g}$ singled out by the data we have fixed, namely the one whose Levi contains just the two roots $\pm\alpha_0$. Because $\alpha_0$ is noncompact, the corresponding $\theta$-stable Levi subgroup $L$ of $\mc{G}(\R)$ has derived group $L^{\mr{der}}$ given by $SL_2(\R)$ or $PSL_2(\R)$, and center a codimension $1$ subgroup of the torus $T$. Using $\Lambda$ we define a character, call it $\lambda$, of $L$ which is trivial on the derived group $L^{\mr{der}}$. Let $I$ denote the nonunitary induction to $L$ of the trivial character from a Borel subgroup of $L$. Then there is an exact sequence of representations of $L$:
\[0\to\lambda \to I\otimes \lambda \to (D_{2}\otimes\lambda)\oplus (D_{-2}\otimes\lambda)\to 0;\]
here, $D_2$ denotes the weight $2$ (holomorphic) discrete series of $L$ with trivial central character, and $D_{-2}$ the weight $-2$ (antiholomorphic) discrete series of $L$, again with trivial central character. Then it turns out that
\[\Ind_{\mc{P}(\R)}^{\mc{G}(\R)}(\pi_\infty,-s_0)\cong \mc{R}_\mf{q}^S(I\otimes\lambda),\]
where $\mc{R}_\mf{q}^S$ denotes a cohomological induction functor. We deduce all parts of Theorem \ref{introthmstdmodforind} from this identity. So the archimedean situation for $\mc{G}$ is really \textit{cohomologically induced} from that for $(P)SL_2(\R)$.

We then write down the long exact $(\mf{g},K_\infty)$-cohomology sequence for \eqref{eqnthmBses1} and \eqref{eqnthmBses2} in Section \ref{subsecgkcoh}, a key fact for the proof of Theorem \ref{introthmmain} being that the boundary map for \eqref{eqnthmBses2} between degrees $d$ and $d+1$ is nonzero. In fact, it is not difficult to deduce the following from the presentation of the representations $D_\pm$ and $J$ as cohomologically induced modules.

\begin{introthm}
\label{introthmcohles}
Continue to assume $\mc{G}(\R)$ is connected. For short, let us write
\[h^q(\cdot)=\dim_{\C}H^q(\mf{g},K_\infty;(\cdot)\otimes E^\vee),\qquad I_{\pm}=\Ind_{\mc{P}(\R)}^{\mc{G}(\R)}(\pi_\infty,\pm s_0),\qquad D=D_+\oplus D_-.\]
Then the following tabulates the dimensions of each term in the long exact cohomology sequences associated with \eqref{eqnthmBses1} and \eqref{eqnthmBses2} of Theorem \ref{introthmstdmodforind} in the degrees where they are nonvanishing:
\begin{center}
\begin{tabular}{| c || c | c | c |} 
\hline
$q$ & $h^q(D)$ & $h^q(I_{+})$ & $h^q(J)$ \\
\hline\hline
$d-1$ & $0$ & $0$ & $1$ \\ 
\hline
$d$ & $2$ & $1$ & $0$ \\
\hline
$d+1$ & $0$ & $1$ & $1$ \\
\hline
\end{tabular}
\qquad\qquad
\begin{tabular}{| c || c | c | c |}
\hline
$q$ & $h^q(J)$ & $h^q(I_{-})$ & $h^q(D)$ \\
\hline\hline
$d-1$ & $1$ & $1$ & $0$ \\ 
\hline
$d$ & $0$ & $1$ & $2$ \\
\hline
$d+1$ & $1$ & $0$ & $0$ \\
\hline
\end{tabular}
\end{center}
We also have $h^d(D_+)=h^d(D_-)=1$.
\end{introthm}

This theorem follows from Propositions \ref{propcomplexesconc} and \ref{propcohofinduced} below.

For the explicit cochain computations in our second proof of Theorem \ref{introthmmain} (c), the nature of intertwining operators and their orders of vanishing will also be of significant importance. There is an element $w_0$ preserving the Levi $\mc{M}$ and sending $\mc{P}(\R)$ to its opposite. Then $w_0$ defines an intertwining operator
\[M(s,\cdot):\Ind_{\mc{P}(\R)}^{\mc{G}(\R)}(\pi_\infty,s)\to\Ind_{\mc{P}(\R)}^{\mc{G}(\R)}(\pi_\infty,-s)\]
for any $s\in\C$. We specialize to $s=s_0$ and prove the following; see Theorem \ref{thmintertwining} below.

\begin{introthm}
\label{introthmintertwining}
Continue to assume $\mc{G}(\R)$ is connected. Let $\phi_s\in\Ind_{\mc{P}(\R)}^{\mc{G}(\R)}(\pi_\infty,s)$ be any flat section whose specialization $\phi_{s_0}$ at $s=s_0$ is in $D_+\oplus D_-$ (see Theorem \ref{introthmstdmodforind}). Then the limit
\[\lim_{s\to s_0}\frac{1}{s-s_0}M(s,\phi_s)\modulo{J}\]
defines an isomorphism
\[(D_+\oplus D_-)\overset{\sim}{\longrightarrow}\Ind_{\mc{P}(\R)}^{\mc{G}(\R)}(\pi_\infty,-s)/J.\]
\end{introthm}

We remark that $s_0$ has an alternate interpretation here in terms of the root $\alpha_0$. In fact, we have
\[s_0=(2\rho_{\mc{P}}(X_{\alpha_0}+X_{-\alpha_0}))^{-1},\]
where again $\rho_{\mc{P}}$ is the half sum of roots in the unipotent radical $\mc{N}$ of $\mc{P}$, and $X_{\pm\alpha_0}$ are appropriately normalized root vectors; then $X_{\alpha_0}+X_{-\alpha_0}$ spans the real Lie algebra of the $\R$-split center of $\mc{M}(\R)$.

The proof of Theorem \ref{introthmintertwining}, perhaps somewhat surprisingly, makes crucial use of the $(\mf{g},K_\infty)$-cohomology complex for the representations $\Ind_{\mc{P}(\R)}^{\mc{G}(\R)}(\pi_\infty,s)\otimes E^\vee$ as $s$ varies. The idea is as follows. We know for general reasons that $M(s_0,\cdot)$ vanishes on $D_+\oplus D_-$ in $\Ind_{\mc{P}(\R)}^{\mc{G}(\R)}(\pi_\infty,s_0)$, but not on elements mapping nontrivially to $J$. Because the boundary map in the long exact sequence associated with \eqref{eqnthmBses1} is nonzero, there is a cochain $c$ in degree $d-1$ for $\Ind_{\mc{P}(\R)}^{\mc{G}(\R)}(\pi_\infty,s_0)\otimes E^\vee$ whose image is nontrivial in $J\otimes E^\vee$, and whose image $dc$ under the boundary map is a nonzero cocycle valued in $D_+\oplus D_-$. So we want to show that $M(s_0,dc)$ vanishes to order $1$ at $s=s_0$.

Now this cochain $c$ can be extended flatly to a cochain $c_s$ valued in $\Ind_{\mc{P}(\R)}^{\mc{G}(\R)}(\pi_\infty,s)\otimes E^\vee$ for all $s\in\C$. Moreover, the evaluation of $dc_s$ at any element of $\bigwedge^d\mf{g}$ is by definition a linear combination of degree $1$ elements of $U(\mf{g})$ applied to flat sections tensored with elements of $E^\vee$; this, in turn, may be shown to be a element of the form
\[s\phi_s^{(1)}\otimes e_1+\phi_s^{(0)}\otimes e_0,\]
for some flat $\phi_s^{(0)}$ and $\phi_s^{(1)}$ in $\Ind_{\mc{P}(\R)}^{\mc{G}(\R)}(\pi_\infty,s)$ and some $e_0,e_1\in E^\vee$. The linearity in $s$ here is crucial and follows from the fact that the elements of $U(\mf{g})$ involved in the definition of the boundary are degree $1$. It follows that if $dc_s$ vanishes, it does so to order $1$ in $s$ somewhere. But it must actually do so at $s=-s_0$ since the boundary map between degrees $d-1$ and $d$ coming from \eqref{eqnthmBses2} is zero. However,
\[M(s,dc)=dM(s,c)=a(s)dc_{-s},\]
for some holomorphic function $a(s)$ defined and nonvanishing near $s=s_0$. Thus taking $s$ to $s_0$ gives order $1$ vanishing on the right hand side, hence on the left, as desired.

If $dc$ takes values genuinely in both components in the sum $(D_+\otimes E^\vee)\oplus (D_-\otimes E^\vee)$, then Theorem \ref{introthmintertwining} follows. In the case that $dc$ is valued in only one of these two summands, we obtain the theorem only for the discrete series representation in that corresponding summand. However, a similar but modified argument passing between degrees $d$ and $d+1$ will prove the theorem for the other summand in this case.

Finally, let us outline the second proof we give of Theorem \ref{introthmmain} (c) in the case that $\mc{G}(\R)$ is connected. Let
\[\phi_s=\phi_{f,s}\otimes\phi_{\infty,s}\in\Ind_{\mc{P}(\A_f)}^{\mc{G}(\A_f)}(\pi_f,s)\otimes\Ind_{\mc{P}(\R)}^{\mc{G}(\R)}(\pi_\infty,s)\cong\Ind_{\mc{P}(\A)}^{\mc{G}(\A)}(\pi,s)\]
be a flat section such that $E(\phi,s,g)$ has a pole at $s=s_0$. The constant term of this Eisenstein series along $\mc{P}$ is
\[E_{\mc{N}}(\phi,s,g)=\phi_s(g)(1_{\mc{M}})+(M(s,\phi_f)\otimes M(s,\phi_\infty))(1_{\mc{M}});\]
here, both sections above are valued in $\pi$, which is a space of cusp forms on $\mc{M}(\A)$, so it makes sense to evaluate them at the identity element $1_{\mc{M}}$ of $\mc{M}(\A)$ to obtain a complex number. Then it follows that the finite adelic intertwined section $M(s,\phi_f)$ must have a pole at $s=s_0$, which is necessarily simple, as maximal parabolic Eisenstein series have at worst simple poles. It also follows that $\phi_{\infty,s_0}$ has nontrivial image in $J$.

One key point then is that if we replace $\phi_{\infty,s}$ above with any flat section $\phi_{\infty,s}'$ such that $\phi_{s_0,\infty}\in D_+\oplus D_-$, then by Theorem \ref{introthmintertwining}, the simple zero of $M(s,\phi_\infty)$ \textit{exactly} cancels the pole of $M(s,\phi_f)$ at $s=s_0$. Therefore the Eisenstein series $E(\phi_f\otimes\phi_\infty',s,g)$ will be regular at $s=s_0$, and its constant term still contains a second summand.

So we can define two $\mc{A}_E(\mc{G})\otimes E^\vee$-valued cocycles as follows. First, let
\[c_{\infty}^{d,\pm}:\sideset{}{^d}\bigwedge(\mf{g}/\mf{k})\to D_{\pm}\otimes E^\vee\]
be the unique (up to scalar) nontrivial cocycles, viewed as taking values in $\Ind_{\mc{P}(\R)}^{\mc{G}(\R)}(\pi_\infty,s_0)\otimes E^\vee$. Then there is an obvious way to use $c_{\infty}^{d,\pm}$ to define two $\mc{A}_E(\mc{G})\otimes E^\vee$-valued cochains $c^{d,\pm}$ by tensoring with $\phi_{f,s_0}$ and forming Eisenstein series. We note the important fact that these Eisenstein series valued cochains will not necessarily be cocycles! In fact, the constant terms of these Eisenstein series are given by the sum of the cocycle
\[\phi_{f,s_0}\otimes c_{\infty,s_0}:\sideset{}{^d}\bigwedge(\mf{g}/\mf{k})\to\Ind_{\mc{P}(\R)}^{\mc{G}(\R)}(\pi,s_0)\otimes E^\vee\]
and some nonzero multiple of the cochain
\[M(s_0,\phi_f)\otimes M(s_0,c_\infty^{d,\pm}):\sideset{}{^d}\bigwedge(\mf{g}/\mf{k})\to\Ind_{\mc{P}(\R)}^{\mc{G}(\R)}(\pi,-s_0)\otimes E^\vee,\]
where the above is really interpreted as a limit as $s\to s_0$. By Theorem \ref{introthmcohles}, it follows that the boundary map $d$ kills the first of these, but that, for at least one choice of the sign $\pm$, the second has image a nontrivial $(d+1)$-cocycle valued in $\mc{L}(\pi,s_0)_f\otimes J\otimes E^\vee=\mc{L}(\pi,s_0)\otimes E^\vee$. Note that \textit{a priori} this cocycle is a coboundary valued in
\[\mc{L}(\pi,s_0)_f\otimes \Ind_{\mc{P}(\R)}^{\mc{G}(\R)}(\pi,s_0)\otimes E^\vee,\]
but it becomes a nontrivial cocycle when viewed as valued in the subrepresentation $\mc{L}(\pi,s_0)\otimes E^\vee$.

This is enough to imply Theorem \ref{introthmmain} (c) by comparing constant terms of $dc^{d,\pm}$ with that of certain residual Eisenstein series. Actually, in the main body of the paper, we take the point of view that the above computes the connecting homomorphism between degrees $d$ and $d+1$ in a long exact cohomology sequence associated with a simple short exact sequence coming from the \textit{Franke filtration} on a direct summand of $\mc{A}_E(\mc{G})$; see Theorem \ref{thmgrbac} of Grbac \cite{grbac} recalled below. This point of view seems to be necessary for proving part (b) of Theorem \ref{introthmmain} in any case. And for part (c), Grbac's theorem implies that the cocycles $dc^{d,\pm}$ are automatically valued in the residual representation $\mc{L}(\pi,s_0)\otimes E^\vee$, and so the nontriviality of one of them is enough to conclude.

\subsection*{Previous results}

The method used proof of part (b) of Theorem \ref{introthmmain} is not new and uses a technique involving the Franke filtration which is common in the literature on Eisenstein cohomology. This technique may be found in, for example, the paper of Rohlfs--Speh \cite{rohlfsspeh}, whose result actually implies Theorem \ref{introthmmain} (b) in the case that $\mc{G}(\R)$ is connected, though they do not explicitly use this filtration. The paper of Grobner \cite{grobqres} does use this filtration explicitly, and it takes place in a much more general setting than ours. However, it falls just short of implying Theorem \ref{introthmmain} (b) just because it is so general. But our technique for this part of our theorem is a specialization of his. We mention also the paper of Grbac and Schwermer \cite{grbsch} on the residual Eisenstein cohomology of symplectic groups, as well as that of Gotsbacher and Grobner \cite{gotsgrob} on that of split odd special orthogonal groups. A paper of Grobner \cite{grobsp11} also obtains results for inner forms of $Sp_4$, while also bounding the degree of residual Eisenstein cohomology above by the middle one, similar to part (c) of Theorem \ref{introthmmain}.

The consolidation of the archimedean setting used in this paper, made in sections \ref{subsecsetup}, \ref{subsecds}, and \ref{subsecaltsetup}, is a relatively novel feature of this paper which we have not seen elsewhere.

Questions of the type answered by part (c) of Theorem \ref{introthmmain} seem not to have been widely addressed thus far. Of course, there is the paper Rohlfs--Speh \cite{RSson1} which we mention above, and which we mimic for the first proof we give of Theorem \ref{introthmmain} (c). There is also the paper of Grobner \cite{grobsp11} mentioned above, as well as \cite{GrobnerQuat}; the latter in fact uses the cited method of Rohlfs--Speh from \cite{RSson1}.

The second method to prove Theorem \ref{introthmmain} (c) seems to be new. It shares some features with another new method of C. Skinner which roughly uses Eisenstein cohomology to construct nontrivial extensions of Hodge structures, the existence of which is predicted by conjectures of Beilinson--Bloch--Kato. See, for example, the recent Princeton thesis of Z. Shang \cite{shang}, written under the supervision of Skinner.

We believe that to go further with these methods, one really needs to be able to construct explicit primitives in Eisenstein cohomology and examine boundary maps in cohomology coming from the Franke filtration; our second proof of Theorem \ref{introthmmain} is a first step in this direction, and may be viewed as a proof of concept in this light.

\subsection*{Acknowledgements}

I wish to thank Harald Grobner for inviting me to Universit\"at Wien to speak on this topic, and for pointing me to various places in the literature. I also thank A. Raghuram for taking an interest in this project encouraging me to work on it.

\subsection*{Notation and conventions}

Throughout this paper, the symbol $\A$ denotes the adeles of $\Q$ and $\A_f$ the ring of finite adeles of $\Q$.

Given a real Lie group $G$, we denote by $G^\circ$ the connected component of the identity in $G$.

The center of a group will be denoted with $Z$; for instance if $G$ is a real Lie group, then $Z_G$ is its center.

Centralizers will be denoted with $C$ and normalizers with $N$; so for instance if $G$ is again a real Lie group, and $X$ is a subset of either $G$ or its (possibly complexified) Lie algebra, then $C_G(X)$ and $N_G(X)$ denote the centralizer and normalizer, respectively, of $X$ in $G$.

Certain fonts will be reserved for certain types of objects.

Linear algebraic groups over $\Q$ will be denoted with calligraphic letters, like $\mc{G}$. The only exception to this will be groups with standardized names, for example $SL_2$, $Sp_{2n}$, and so on.

Real Lie groups will be denoted with Roman letters, like $G$, unless they are the $\R$ points of a linear algebraic group, say $\mc{G}$, in which case we would simply write $\mc{G}(\R)$. Since Section \ref{secreal} is about real Lie groups and no algebraic groups over $\Q$ appear there, these conventions will then be useful in Section \ref{secadele} for differentiating between groups coming from Section \ref{secreal} and the algebraic groups appearing in Section \ref{secadele}.

Lowercase fraktur letters alone, like $\mf{g}$, will be reserved for complex Lie algebras. If there is a (real Lie or linear algebraic) group whose complexified Lie algebra is considered, the corresponding letters in the notation will almost always match. So $\mf{g}$ could denote the complexified Lie algebra of a real Lie group $G$ or a linear algebraic group $\mc{G}$. The one exception to this rule will appear in Section \ref{secreal} in certain places; there we will have fixed a real parabolic subgroup $P$ of a real Lie group $G$, but $\mf{p}$ will denote the $(-1)$-eigenspace with respect to a fixed Cartan involution in the complexified Lie algebra $\mf{g}$ of $G$, and it will \textit{not} denote the complexified Lie algebra of $P$. However, we warn the reader that in Section \ref{secadele}, we will have occasion to consider the complexified Lie algebra $\mf{p}$ of a parabolic subgroup $\mc{P}$ of a linear algebraic group $\mc{G}$, and we will not need to consider the $(-1)$-eigenspace of a Cartan involution there. So we expect this to cause no confusion.

When we have occasion to consider real Lie algebras, they will always come with a subscript $\R$. So for instance $\mf{g}_\R$ could denote the real Lie algebra of a real Lie group $G$.

Let $G$ be a real Lie group, and write $\mf{g}$ and $K$ for its complexified Lie algebra and a fixed maximal compact subgroup, respectively. Then a ``representation" of $G$ will always mean an ``admissible $(\mf{g},K)$-module."

If $\mc{G}$ is a reductive algebraic group over $\Q$, an ``automorphic representation" of $\mc{G}(\A)$ will mean, in addition to many other conditions, an irreducible object in the category of $\mc{G}(\A_f)\times (\mf{g},K_\infty)$-modules where, of course, $\mf{g}$ is the complexified Lie algebra of $\mf{g}$ and $K_\infty$ is a fixed maximal compact subgroup of $\mc{G}(\R)$. If $\pi$ is such, then we write $\pi_f$ for the finite adelic part and $\pi_\infty$ for the archimedean part. Then $\pi_f$ is smooth and admissible and $\pi_\infty$ is admissible.

We will often consider representations induced from Levi factors of maximal parabolic subgroups, and we will do so in the real, adelic, and finite adelic settings. Let $\mc{G}$ be a linear semisimple algebraic group over $\Q$, and let $\mc{P}$ be a maximal parabolic subgroup of $\mc{G}$ with Levi factor $\mc{M}$ and unipotent radical $\mc{N}$. Let $\mf{g}$ be the complexified Lie algebra of $\mc{G}$ and $K$ a fixed maximal compact subgroup of $\mc{G}(\A)$. Let $\pi$ be an automorphic representation of $\mc{M}(\A)$, and $s$ a complex number. Then the notation
\[\Ind_{\mc{P}(\A)}^{\mc{G}(\A)}(\pi,s)\]
will mean the space of functions $\mc{G}(\A)\to \pi$ such that:
\begin{itemize}
\item $\phi(mng)=\delta_{\mc{P}}^{s+1/2}(m)\pi(m)\phi(g)$ for all $m\in\mc{M}(\A)$, $n\in\mc{N}(\A)$, and $g\in\mc{G}(\A)$, where $\delta_{\mc{P}}$ is the modulus character of $\mc{P}(\A)$;
\item $\phi$ is smooth under right translation by $\mc{G}(\A)$;
\item $\phi$ is $K$-finite under right translation.
\end{itemize}
Thus implicit in this notation is a normalization by $\delta_{\mc{P}}^{1/2}$. Then $\Ind_{\mc{P}(\A)}^{\mc{G}(\A)}(\pi,s)$ is an admissible $\mc{G}(\A_f)\times (\mf{g},K_\infty)$-module by right translation.

Similar remarks, including this normalization by the square root of the modulus character, all apply to induced representations in the finite adelic and real contexts. We will also consider an alternative interpretation of these induced representations in the real context in Section \ref{subsecintertwining}.

A linear algebraic group, respectively real Lie group, will mean one which admits a faithful finite dimensional algebraic, respectively complex, representation. In the real context, this is important because this condition is assumed in Vogan's book \cite{voganbook}, whose content we use heavily in Section \ref{subsecstandardmods}.

Finally, we note a discrepancy between our choice of notation versus a piece of notation that exists in much of the literature on real Lie groups. We will often consider a Levi subgroup, which we will denote $M$, of a real semisimple Lie group $G$. We will then have subgroups of $M$ which we denote below by $M_0$ and $A$, and we sometimes consider the connected component $M_0^\circ$ of the identity in $M_0$. However, in much of the literature, other authors write $M_0$ where we have written $M_0^\circ$, they have written $M$ where we have written $M_0$, and they have written $MA$ where we have written $M$. We have chosen our notation in order to be more consistent with the notation we use in the global, adelic setting, and that notation seems to be more standard there.

\tableofcontents

\section{The $SL_2$ case}
\setcounter{subsection}{1}
In order to illustrate the second of the two methods used to prove the main theorem of this paper, Theorem \ref{thmmainthm} below, we run the entire proof in the simple case of $SL_2$ here in this section. Let us begin by setting up notation.

Throughout this section we denote by $SL_2$ the usual group with that name, viewed as simple algebraic group over $\Q$. Its complex Lie algebra will be denoted by $\mf{sl}_2$. The group $SL_2(\R)$ of real points of $SL_2$ contains the compact subgroup $SO_2$ in the standard way:
\[SO_2=\Sset{\pmat{a&b\\ -b&a}}{a,b\in\R,\,\, a^2+b^2=1}.\]
Then $SO_2$ is a maximal compact subgroup of $SL_2(\R)$ whose complexified Lie algebra we denote by $\mf{so}_2$. Note that $SO_2$ is also a Cartan subgroup of $SL_2(\R)$.

The main consideration of this section is the $(\mf{sl}_2,SO_2)$-cohomology of the space, which we denote by $\mc{A}(SL_2)$, of automorphic forms for $SL_2$. The general theory of $(\mf{g},K)$-cohomology is recalled in Section \ref{subsecgkcoh} below, but in this case we can be quite explicit.

So consider the matrices
\begin{equation}
\label{eqnHXplusXminus}
H=i\pmat{0&-1\\ 1&0},\qquad X_+=\frac{1}{2}\pmat{1&i\\ i&-1},\qquad X_-=\frac{1}{2}\pmat{1&-i\\ -i&-1}.
\end{equation}
Then $H$ spans $\mf{so}_2$, and $(H,X_+,X_-)$ forms an $\mf{sl}_2$-triple.

The $(\mf{sl}_2,SO_2)$-cohomology complex an $(\mf{sl}_2,SO_2)$-module $V$ can be explicitly written as the complex
\[\dotsb\to 0\to C^0(\mf{sl}_2,SO_2;V)\to C^1(\mf{sl}_2,SO_2;V)\to C^2(\mf{sl}_2,SO_2;V)\to 0 \to\dotsb,\]
whose terms are defined as follows. Note first that the subspace
\[\C X_+\oplus\C X_-\subset \sl_2\]
is preserved by the adjoint action of $SO_2$, and the determinant of this space, namely
\[\C \cdot(X_+\wedge X_-),\]
inherits the trivial action of $SO_2$. Thus we can define
\[C^i(\mf{sl}_2,SO_2;V)=\hom_{SO_2}(\sideset{}{^i}{\bigwedge}(\C X_+\oplus\C X_-),V),\]
where the subscript above means $SO_2$-equivariant linear maps. Thus there are obvious isomorphisms
\[C^0(\mf{sl}_2,SO_2;V)\cong V^{SO_2}\cong C^2(\mf{sl}_2,SO_2;V),\]
where the superscript above denotes $SO_2$-invariants; the left isomorphism is given by taking the image of $1\in\C=\bigwedge^0(\C X_+\oplus\C X_-)$, and the right by taking the image of $X_+\wedge X_-$.

The boundary maps
\[d_q:C^q(\mf{sl}_2,SO_2;V)\to C^{q+1}(\mf{sl}_2,SO_2;V),\qquad q=0,1\]
are given by
\begin{equation}
\label{eqndzero}
(d_0v)(X)=Xv,\qquad X=X_+,X_-\textrm{ and }v\in V^{SO_2},
\end{equation}
and
\begin{equation}
\label{eqndone}
d_1c=X_+c(X_-) - X_-c(X_+)\in V^{SO_2},\qquad c\in\hom_{SO_2}(\C X_+\oplus\C X_-,V).
\end{equation}

Denote the cohomology groups of this complex by $H^{*}(\mf{sl}_2,SO_2;V)$. Despite the presence of the symmetry above between the complexes about degree $1$, of course the cohomology groups do not enjoy such symmetry in general. Instead, something more interesting happens!

First, if $V$ is, for example, the trivial representation $\C$ of $SL_2(\R)$, then visibly
\[C^1(\mf{sl}_2,SO_2;\C)=0,\]
so that
\[H^0(\mf{sl}_2,SO_2;\C)=C^0(\mf{sl}_2,SO_2;\C)=\C\]
and
\[H^2(\mf{sl}_2,SO_2;\C)=C^2(\mf{sl}_2,SO_2;\C)=\C\cdot(X_+\wedge X_-)^\vee,\]
which are thus isomorphic; here we have denoted by $(X_+\wedge X_-)^\vee$ the linear map $\C\cdot(X_+\wedge X_-)\to\C$ sending $X_+\wedge X_-$ to $1$. So this symmetry remains present on the level of cohomology in this case. However, we can also view the trivial representation as an automorphic representation, spanned by the constant function $1$. Of course, when restricted from $SL_2(\A)$ to $SL_2(\R)$, it gives the trivial representation $\C$ of $SL_2(\R)$ above. Then we have the following theorem.

\begin{theorem}
\label{thmmainthmbaby}
The inclusion $\C\cdot 1\hookrightarrow\mc{A}(SL_2)$ of the span of the constant function into the space of automorphic forms for $SL_2$ induces injections
\[H^q(\mf{sl}_2,SO_2;\C)\hookrightarrow H^q(\mf{sl}_2,SO_2;\mc{A}(SL_2))\]
in all degrees $q$ except $q=2$, where this map is the zero map.
\end{theorem}

Of course, it is well known that $H^*(\mf{sl}_2,SO_2;\C)$ computes the inductive limit, over all level structures, of the cohomology of the modular curves with those level structures. These modular curves are open Riemann surfaces, and therefore have no $H^2$. So the theorem by itself is maybe not so interesting. The purpose of this section, however, is to write down a purely automorphic proof of this theorem, without appealing in any way to geometry.

Now of course the constant function $1$ has vanishing derivatives by $X_+$ and $X_-$, and so $d_0 1=0$ by \eqref{eqndzero}. Thus $1\in \mc{A}(SL_2)^{SO_2}$ gives a $0$-cocycle (not just a cochain) and thus provides a nonzero class in $H^0(\mf{sl}_2,SO_2;\mc{A}(SL_2))$, proving that the map
\[H^q(\mf{sl}_2,SO_2;\C)\to H^q(\mf{sl}_2,SO_2;\mc{A}(SL_2))\]
is an injection in degree $q=0$. It is injective in all other degrees $q\ne 2$ because the source group is zero.

So to prove the theorem, we must show the map
\[H^2(\mf{sl}_2,SO_2;\C)\to H^2(\mf{sl}_2,SO_2;\mc{A}(SL_2))\]
is zero, i.e., we need to show that the cochain
\[1\cdot (X_+\wedge X_-)^\vee\in C^2(\mf{sl}_2,SO_2;\mc{A}(SL_2))\]
is actually a coboundary. We will do this by explicitly constructing a $1$-cochain in automorphic forms whose image under the boundary map $d_1$ is the constant function.

The automorphic forms which will help us accomplish this task will be Eisenstein series, whose construction in this case we recall below, after some preliminaries on the induced representations needed to define them.

Let us write $\mc B$, $\mc T$, and $\mc N$ for the usual subgroups of $SL_2$,
\[\mc{B}=\Set{\pmat{*&*\\ 0&*}},\qquad \mc{T}=\Set{\pmat{*&0\\ 0&*}},\qquad \mc{N}=\Set{\pmat{1&*\\ 0&1}},\]
so that $\mc B$ is the a Borel with Levi decomposition $\mc{B}=\mc{T}\mc{N}$. By abuse of notation, we write $\delta$ for the usual modulus character of $\mc{B}(\R)$, $\mc{B}(\A_f)$, or $\mc{B}(\A)$, given in all cases by
\[\delta\left(\sm{t&*\\ 0&t^{-1}}\right)=\vert t^2\vert.\]

Let $K=SL_2(\widehat{\Z})\times SO_2\subset SL_2(\A)$, which is a maximal compact subgroup of $SL_2(\A)$. For $s\in\C$, we consider the (unitarily normalized) induced representation from the trivial character $1$ of $T(\A)$,
\[\Ind_{\mc{B}(\A)}^{SL_2(\A)}(1,s)=\Sset{\phi:SL_2(\A)\to\C}{\begin{array}{l}
\phi\textrm{ is smooth and }K\textrm{-finite, }\\
\textrm{and }\phi(bg)=\delta(b)^{s+1/2}\phi(g)\\
\textrm{for all }b\in \mc{B}(\A),\,\, g\in SL_2(\A)
\end{array}}.\]
This decomposes as a tensor product
\begin{equation}
\label{eqntensorofind}
\Ind_{\mc{B}(\A)}^{SL_2(\A)}(1,s)\cong\Ind_{\mc{B}(\A_f)}^{SL_2(\A_f)}(1,s)\otimes\Ind_{\mc{B}(\R)}^{SL_2(\R)}(1,s),
\end{equation}
where the definitions of either factor on the right hand side are similar as above.

Now for each $s\in\C$, there is a unique section $\phi_{f,s}^{\sph}\in \Ind_{B(\A_f)}^{SL_2(\A_f)}(1,s)$ such that $\phi_{f,s}^{\sph}(k_f)=1$ for any $k_f\in SL_2(\widehat{\Z})$. Similarly, there is a unique section $\phi_{\infty,s}^{(0)}\in\Ind_{\mc{B}(\R)}^{SL_2(\R)}(1,s)$ such that $\phi_{\infty,s}^{(0)}(k_\infty)=1$ for all $k_\infty\in SO_2$. We use this latter section to define two more sections at the archimedean place as follows.

First, recall that at $s=1/2$, there is an exact sequence
\[0\to D_{2}\oplus D_{-2} \to \Ind_{\mc{B}(\R)}^{SL_2(\R)}(1,1/2)\to \C\to 0;\]
here, $D_2$ is the weight $2$ holomorphic discrete series representation of $SL_2(\R)$, and likewise $D_{-2}$ is the weight $-2$ antiholomorphic discrete series representation of $SL_2(\R)$, both seen as admissible $(\mf{sl}_2,SO_2)$-modules. For clarity, the $\C$ at the end of this exact sequence is the trivial representation. This exact sequence is nonsplit, even though it is split as an exact sequence of $SO_2$-modules. In fact, the section $\phi_{\infty,1/2}^{(0)}$ spans the unique trivial $SO_2$-type in $\Ind_{\mc{B}(\R)}^{SL_2(\R)}(1,1/2)$, and so maps nontrivially into the $\C$ at the end; however, if we define
\[\phi_{\infty,1/2}^{(2)}=X_+\phi_{\infty,1/2}^{(0)}\textrm{ and }\phi_{\infty,1/2}^{(-2)}=X_-\phi_{\infty,1/2}^{(0)},\]
then $\phi_{\infty,1/2}^{(\pm 2)}\in D_{\pm 2}$, and $\phi_{\infty,1/2}^{(\pm 2)}$ spans the lowest $SO_2$-type in $D_{\pm 2}$ in both cases.

We can also put the sections $\phi_{\infty,1/2}^{(\pm 2)}$ into unique \textit{flat} sections $\phi_{\infty,s}^{(\pm 2)}\in\Ind_{\mc{B}(\R)}^{SL_2(\R)}(1,s)$ varying over $s\in\C$, where flatness here means that the restrictions $\phi_{\infty,s}^{(\pm 2)}|_{SO_2}$ are independent of $s$ as functions on $SO_2$. These sections therefore span the unique $SO_2$-types in $\Ind_{\mc{B}(\R)}^{SL_2(\R)}(1,s)$ of weight $\pm 2$.

With the decomposition \eqref{eqntensorofind} in mind, we define three sections $\phi_s^{(-2)},\phi_s^{(0)},\phi_s^{(2)}\in\Ind_{\mc{B}(\A)}^{SL_2(\A)}(1,s)$, which are flat now in the sense that their restrictions to $K$ do not depend on $s$, by letting
\[\phi_s^{(?)}=\phi_{f,s}^{\sph}\otimes\phi_{\infty,s}^{(?)},\qquad ?\in\{-2,0,2\}.\]
We build Eisenstein series using these sections.

In fact, for $\phi\in\{\phi_s^{(-2)},\phi_s^{(0)},\phi_s^{(2)}\}$, let us define the Eisenstein series by
\[E(\phi,s,g)=\sum_{\gamma\in \mc{B}(\Q)\backslash SL_2(\Q)}\phi(\gamma g),\qquad s\in\C,\,\, g\in SL_2(\A).\]
It is a classical fact that this series converges when $\re(s)$ is in a certain right half plane, and can be extended meromorphically to all of $\C$, in the sense that for each $g\in SL_2(\A)$, the function $E(\phi,s,g)$ in $s$ is meromorphic. In fact, when $\re(s)>0$, the expressions $E(\phi,s,g)$ are holomorphic and define automorphic forms on $SL_2(\A)$, except when $s=1/2$ and $\phi=\phi_s^{(0)}$; in this case the Eisenstein series $E(\phi^{(0)},s,g)$ (we omit the subscript $s$ from $\phi^{(0)}$ here because it is implicit in the data) has a pole of order $1$ at $s=1/2$, and its residue,
\[\res_{s=1/2}E(\phi^{(0)},s,g)=\lim_{s\to 1/2}(s-\tfrac{1}{2})E(\phi^{(0)},s,g)\]
is a nonzero multiple of the constant function of $g$.

Now let us return to cohomology. Let $a_+\in\C$ and $a_-\in\C$ be constants, which we will allow ourselves to specify later. With some foresight, we define a cochain $c\in C^1(\mf{sl}_2,SO_2;\mc{A}(SL_2))$ by specifying its values on $X_+$ and $X_-$ as follows:
\begin{equation}
\label{eqndefofcinbabycase}
c(X_{\pm})=a_\pm E(\phi^{(\pm 2)},1/2,\cdot).
\end{equation}
The Eisenstein series $E(\phi^{(\pm 2)},1/2,\cdot)$ is acted upon by the group $SO_2$ via the character of weight $\pm 2$, just as is the element $X_\pm$, and so this gives a well defined $1$-cochain.

In a moment we will compute that $(d_1c)(X_+\wedge X_-)$ is, as desired, a nonzero multiple of the constant function. In order to do this, we need some way of identifying the automorphic form $(dc)(X_+\wedge X_-)$, and this identification will be made via the constant term, as we describe now.

So for $\phi\in\{\phi_s^{(-2)},\phi_s^{(0)},\phi_s^{(2)}\}$, write $E_{\mc{N}}(\phi,s,g)$ for the constant term along $\mc B$,
\begin{equation}
\label{eqnconsttermbaby}
E_{\mc{N}}(\phi,s,g)=\int_{\mc{N}(\Q)\backslash \mc{N}(\A)}E(\phi,s,ng)\,dn,
\end{equation}
the integration being normalized, say, so that $\mc{N}(\Q)\backslash \mc{N}(\A)$ has measure $1$. Then in fact
\[E_N(\phi,s,g)=\phi(g)+M(s,\phi)(g),\]
where $M(s,\cdot):\Ind_{\mc{B}(\A)}^{SL_2(\A)}(1,s)\to\Ind_{\mc{B}(\A)}^{SL_2(\A)}(1,-s)$ is the intertwining operator defined by
\[M(s,\phi)(g)=\int_{\mc{N}(\A)}\phi_s(w_0ng)\, dn,\qquad w_0=\sm{0&-1\\ 1&0}\in SL_2(\Q);\]
this definition is correct if $\re(s)$ is in a certain right half plane, and again is extended to all $s$ by meromorphic continuation. In fact, this intertwining operator $M(s,\phi)$ has poles exactly when the Eisenstein series $E(\phi,s,\cdot)$ does.

One can evaluate the intertwining operator $M(s,\phi)(g)$ explicitly. Assume that $\phi_s\in\Ind_{\mc{B}(\A)}^{SL_2(\A)}(1,s)$ is a flat section that factorizes as $\phi_s=\phi_{f,s}^{\sph}\otimes\phi_{\infty,s}$ for some $\phi_{\infty,s}\in\Ind_{\mc{B}(\R)}^{SL_2(\R)}(1,s)$. Then it turns out that
\begin{equation}
\label{eqnintertwiningzeta}
M(s,\phi)=\frac{\zeta(2s)}{\zeta(2s+1)}\phi_{f,-s}^{\sph}\otimes M(s,\phi_\infty),
\end{equation}
where on the right hand side, the intertwined archimedean section $M(s,\phi_\infty)$, which is an element of $\Ind_{\mc{B}(\R)}^{SL_2(\R)}(1,-s)$, has an analogous meaning to the global one; namely,
\[M(s,\phi_\infty)(g)=\int_{\mc{N}(\R)}\phi_{\infty,s}(w_0ng)\,dn\]
when $\re(s)$ is large enough, and otherwise we make sense of this expression by analytic continuation. We will need the following.

\begin{proposition}
\label{propintertbaby}
The archimedean intertwining operator $M(s,\cdot):\Ind_{\mc{B}(\R)}^{SL_2(\R)}(1,s)\to \Ind_{\mc{B}(\R)}^{SL_2(\R)}(1,-s)$ vanishes at $s=1/2$ to order $1$ on the sections $\phi_{\infty,s}^{(\pm 2)}$.
\end{proposition}

This proposition follows from well known computations involving gamma functions, but we will give a different proof of it in a moment using the fact that $M(1/2,\phi_\infty^{(0)})\ne 0$, this latter fact being a special case of a general fact about the relationship between intertwining operators and the Langlands classification. The proof we give is what inspired our approach to the analogous fact in this paper, Theorem \ref{thmintertwining}. But for now let us finish the proof of Theorem \ref{thmmainthmbaby}.

Eisenstein series (and their residues) are determined entirely by their constant terms. So to prove that $(d_1c)(X_+\wedge X_-)$ is the constant function, and hence to prove Theorem \ref{thmmainthmbaby}, it suffices to show that the constant term of $(d_1c)(X_+\wedge X_-)$ is constant as a function of $g\in SL_2(\A)$. Now we compute by definition, using \eqref{eqndone} and \eqref{eqndefofcinbabycase}, that
\[(d_1c)(X_+\wedge X_-)=X_+c(X_-)-X_-c(X_+)=a_-X_+E(\phi^{(-2)},1/2,\cdot)-a_+X_-E(\phi^{(2)},1/2,\cdot),\]
where we recall that $a_+$ and $a_-$ are constants that are, for now, unspecified. This automorphic form has constant term, by \eqref{eqnconsttermbaby} and \eqref{eqnintertwiningzeta}, given by
\begin{multline}
\label{eqnconstofd1c}
(d_1c)(X_+\wedge X_-)_{\mc{N}}=a_-X_+\phi_{1/2}^{(-2)}-a_+X_-\phi_{1/2}^{(2)}\\
+\lim_{s\to 1/2}\frac{\zeta(2s)}{\zeta(2s+1)}\phi_{f,-s}^{\sph}\otimes (a_-X_+M(s,\phi_\infty^{(-2)})-a_+X_-M(s,\phi_\infty^{(2)})).
\end{multline}

The section $X_-\phi_{1/2}^{(2)}$ vanishes because $\phi_{\infty,1/2}^{(2)}$ is already in the lowest $SO_2$-type in $D_2$, and similarly $X_+\phi_{1/2}^{(-2)}=0$ as well. Also, by Proposition \ref{propintertbaby}, the simple pole of $\zeta(2s)$ at $s=1/2$ cancels the simple zeros of the intertwining operators in \eqref{eqnconstofd1c}. Thus there are nonzero constants $b_+$ and $b_-$ such that
\[(d_1c)(X_+\wedge X_-)_{\mc{N}}=\phi_{f,-1/2}^{\sph}\otimes (a_-b_-X_+\phi_{\infty,-1/2}^{(-2)}-a_+b_+X_-\phi_{\infty,-1/2}^{(2)}).\]
The fact that we obtain the nonzero multiples of the sections $\phi_{\infty-1/2}^{(-2)}$ and $\phi_{\infty,-1/2}^{(2)}$ on the right hand side just follows from the fact that the $SO_2$-types of weight $2$ and $-2$ in $\Ind_{\mc{B}(\R)}^{SL_2(\R)}(1,-1/2)$ are unique.

Now there is an exact sequence
\[0\to \C \to \Ind_{\mc{B}(\R)}^{SL_2(\R)}(1,-1/2)\to D_{2}\oplus D_{-2}\to 0,\]
dual to the one above, which is nonsplit. It follows that at least one of $X_+\phi_{\infty,-1/2}^{(-2)}$ or $X_-\phi_{\infty,-1/2}^{(2)}$ is a nonzero multiple of $\phi_{\infty,-1/2}^{(0)}$. So there are constants $b_+'$ and $b_-'$, at least one of which is nonzero (in fact, neither are zero, but we want to emphasize that we do not actually need to know this), such that
\[(d_1c)(X_+\wedge X_-)_\mc{N}=(a_-b_-b_-'-a_+b_+b_+')\phi_{f,-1/2}^{\sph}\otimes \phi_{\infty,-1/2}^{(0)}.\]
By choosing $a_-$ and $a_+$ appropriately, we can make
\[a_-b_-b_-'-a_+b_+b_+'\ne 0,\]
and so with this choice, the constant term $(d_1c)(X_+\wedge X_-)_\mc{N}$ is a nonzero multiple of $\phi_{f,-1/2}^{\sph}\otimes \phi_{\infty,-1/2}^{(0)}$. But this section is by definition $K$-fixed, and there is a unique such section in $\Ind_{\mc{B}(\A)}^{SL_2(\A)}(1,-1/2)$, namely the constant function on $SL_2(\A)$. This suffices to show that $(d_1c)(X_+\wedge X_-)_\mc{N}$, and hence $(d_1c)(X_+\wedge X_-)$, are nonzero constant functions, proving Theorem \ref{thmmainthmbaby} under the assumption that Proposition \ref{propintertbaby} is true.

We now turn to the proof of Proposition \ref{propintertbaby}. For simplicity, let us prove that $M(s,\phi_\infty^{(2)})$ vanishes to order $1$ at $s=1/2$; the proof for $\phi_\infty^{(-2)}$ in place of $\phi_\infty^{(2)}$ is completely analogous.

To do this, we first compute the difference between $\phi_{\infty,s}^{(2)}$ and $X_+\phi_{\infty,s}^{(0)}$. Of course, by our definition of $\phi_{\infty,s}^{(2)}$, these two sections agree at $s=1/2$, but they may (and will) differ at other values of $s$. But we know at least that for any fixed $s$, they are multiples of each other, because they are both weight $2$ and the $SO_2$-type in $\Ind_{\mc{B}(\R)}^{SL_2(\R)}(1,s)$ of weight $2$ is unique for any $s$.

Now by definition of $\phi_{\infty,s}^{(0)}$, we have $\phi_{\infty,s}^{(0)}(1)=1$ for all $s$. Let us write
\[X_+=\frac{1}{2}\pmat{1&0\\ 0&-1}+\frac{1}{2}H+i\pmat{0&1\\ 0&0},\]
where $H$ is as in \eqref{eqnHXplusXminus}, and our explicit definition of $X_+$ is given there. Then we compute, using the very definition of the action of the Lie algebra $\mf{sl}_2$ on the induced representation via exponentiation and differentiation, that
\begin{align*}
(X_+\phi_{\infty,s}^{(0)})(1)& =\frac{1}{2}\frac{d}{dt}\phi_{\infty,s}^{(0)}\sm{e^{t}&0\\ 0& e^{-t}}|_{t=0}+\frac{1}{2}(H\phi_{\infty,s}^{(0)})(1)+i\frac{d}{dt}\phi_{\infty,s}^{(0)}\sm{1& t\\ 0&1}|_{t=0}\\
&=\frac{1}{2}\frac{d}{dt}e^{2(s+1/2)t}\phi_{\infty,s}^{(0)}(1)|_{t=0}+0+i\frac{d}{dt}\phi_{\infty,s}^{(0)}(1)|_{t=0}\\
&=(s+\tfrac{1}{2})\phi_{\infty,s}^{(0)}(1)\\
&=s+\tfrac{1}{2};
\end{align*}
in the passage to the second line above, we use that $H\phi_{\infty,s}^{(0)}=0$ because $H\in\mf{so}_2$ and $\phi_{\infty,s}^{(0)}$ is weight $0$, along with the transformation laws for $\phi_{\infty,s}^{(0)}$ under left translation by $\mc{T}(\R)$ and $\mc{N}(\R)$. From this computation we deduce that $\phi_{\infty,1/2}^{(2)}(1)=1$ by taking $s=1/2$ above, and hence that
\[\phi_{\infty,s}^{(2)}(1)=1\textrm{ for all }s,\]
and therefore also that
\[X_+\phi_{\infty,s}^{(0)}=(s+\tfrac{1}{2})\phi_{\infty,s}^{(2)}.\]

Now, since the $SO_2$-type of weight $0$ in $\Ind_{\mc{B}(\R)}^{SL_2(\R)}(1,s)$ is unique for any $s$, we know that there is a holomorphic function $a(s)$ with $a(1/2)\ne 0$, such that
\[M(s,\phi_{\infty}^{(0)})=a(s)\phi_{\infty,-s}^{(0)}.\]
So we compute
\begin{align*}
M(s,\phi_{\infty}^{(2)})&=\frac{1}{s+\tfrac{1}{2}}M(s,X_+\phi_{\infty}^{(0)})\\
&=\frac{1}{s+\tfrac{1}{2}}X_+M(s,\phi_{\infty}^{(0)})\\
&=\frac{1}{s+\tfrac{1}{2}}a(s)X_+\phi_{\infty,-s}^{(0)}\\
&=\frac{-s+\tfrac{1}{2}}{s+\tfrac{1}{2}}a(s)\phi_{\infty,-s}^{(2)},
\end{align*}
which visibly has a zero of order $1$ at $s=1/2$. This completes the proof of Proposition \ref{propintertbaby}, and hence of Theorem \ref{thmmainthmbaby}.\qed

\indent

We remark that a careful reading of the proof Proposition \ref{propintertbaby}, including a working of the case involving $\phi_{\infty,s}^{(-2)}$, would reveal that actually $b_+=b_-$ and $b_+'=b_-'$ above. Thus it suffices to take $a_+=-a_-=1$ to prove Theorem \ref{thmmainthmbaby}. It is, however, a very convenient feature of the proof that we do not need to know this, and in fact we will not have such explicit knowledge of the objects which appear in the proof of Theorem \ref{thmmainthm} below.

\section{Results at the archimedean place}
\label{secreal}

This section is devoted to the results we need at the archimedean place. The main results are Theorem \ref{thmstdmodforind}, the cohomological computations of Section \ref{subsecgkcoh}, and Theorem \ref{thmintertwining}.

After setting the stage in Sections \ref{subsecsetup} and \ref{subsecds}, in Section \ref{subsecstandardmods} we study in detail the structure of certain induced representations, whose cohomology we then describe in Section \ref{subsecgkcoh}. Section \ref{subsecintertwining} uses these results to study the orders of vanishing of intertwining operators on these induced representations. Finally, we give a set of hypotheses in Section \ref{subsecaltsetup} which we then prove are equivalent to the setup in Sections \ref{subsecsetup} and \ref{subsecds}.

\subsection{Setup: Groups and Lie algebras}
\label{subsecsetup}

We fix notation that will be in play throughout this section. Let $G$ be a real, connected, linear, semisimple Lie group. Let $K$ be a maximal compact subgroup of $G$, and assume $K$ contains a compact Cartan subgroup for $G$; fix one such, call it $T$. Then $K$ is connected because $G$ is, as is $T$ (by, for instance, \cite[Lemma 4.43 (d)]{knvo}). Write $\mf{g}$, $\mf{k}$, and $\mf{t}$ for the complexified Lie algebras of $G$, $K$, and $T$, respectively.

Write $\Delta=\Delta(\mf{g},\mf{t})$ for the set of roots of $\mf{t}$ in $\mf{g}$. We fix an ordering on $\Delta$ and let $\Delta^+$ be the subset of positive roots in $\Delta$. We assume that $\Delta^+$ contains a simple root $\alpha_0$ which is noncompact, i.e., which does not occur in $\mf{k}$. We will make various constructions with this fixed simple root $\alpha_0$.

Now as in \cite[\S 2]{blank}, we normalize the root vectors coming from roots in $\Delta$ in the following standard way. Let $B$ be the Killing form on $\mf{g}$ and $\langle\cdot,\cdot\rangle$ the bilinear form on the dual space $\mf{t}^\vee$ of $\mf{t}$ coming from $B$. Fix a Cartan involution $\theta$ on $G$ which gives $K$. Then $T$ is $\theta$-stable. For each $\alpha\in\Delta$, we fix normalized root vectors $X_\alpha$ in $\mf{g}$ such that
\begin{equation}
\label{eqnnormofrootvecs}
B(X_\alpha,X_{-\alpha})=\frac{2}{\langle\alpha,\alpha\rangle}\qquad\textrm{and}\qquad\theta\overline{X}_\alpha=-X_{-\alpha},
\end{equation}
where $\overline{X}_\alpha$ denotes the complex conjugate of $X_\alpha$ with respect to the real Lie algebra of $G$. Let $H_\alpha\in\mf{t}$ be the vector $[X_\alpha,X_{-\alpha}]$. Then $\alpha(H_\alpha)=2$, and so $(H_\alpha,X_{\alpha},X_{-\alpha})$ form a normalized $\mf{sl}_2$-triple.

With these normalizations in mind, we now make constructions involving the fixed noncompact simple root $\alpha_0$. First, write
\[\mf{a}=\C(X_{\alpha_0}+X_{-\alpha_0})\subset \mf{g}.\]
Then $\mf{a}$ is an abelian Lie subalgebra of $\mf{g}$. We also write $\mf{t}'$ for the orthogonal complement of $H_{\alpha_0}$ in $\mf{t}$. Then the Lie algebra $\mf{h}$ defined by $\mf{h}=\mf{a}\oplus\mf{t}'$ is a $\theta$-stable Cartan subalgebra $\mf{g}$. In fact, let
\[C_{\alpha_0}=\Ad(\exp(\tfrac{\pi}{4}(X_{\alpha_0}-X_{-\alpha_0})))\]
be the Cayley transform associated with $\alpha_0$. Then $C_{\alpha_0}$ fixes $\mf{t}'$ and brings $X_{\alpha_0}+X_{-\alpha_0}$ to a nonzero real multiple of $H_{\alpha_0}$. The transpose $\tp{C_{\alpha_0}}$ of the Cayley transform therefore provides a bijection
\[\Delta(\mf{g},\mf{t})\overset{\sim}{\longrightarrow}\Delta(\mf{g},\mf{h}),\]
between the roots of $\mf{t}$ in $\mf{g}$ and those of $\mf{h}$ in $\mf{g}$.

Let $\mf{g}_{\R}$ be the real Lie algebra of $G$. We define $M_0^{\circ}$ to be the connected Lie subgroup of $G$ corresponding to the centralizer $C_{\mf{g}_{\R}}(\mf{a})$. Let $\mf{m}_0$ be the complexified Lie algebra of $M_0^\circ$, so $\mf{m}_0=Z_{\mf{g}_{\R}}(\mf{a})\otimes_{\R}\C$. We note that $\mf{t}'$ is a Cartan subalgebra of $\mf{m}_0$ and that $M_0^\circ\cap T$ is thus a Cartan subgroup of $M_0^\circ$ which is compact. Define $M_0=C_K(\mf{a})M_0^\circ$, so that indeed $M_0^\circ$ is the connected component of the identity in $M_0$. We also write
\begin{equation}
\label{eqndefofA}
A=\exp(\R(X_{\alpha_0}+X_{-\alpha_0}))
\end{equation}
and let $M=M_0A$. Then $M$ is the centralizer of $A$ in $G$. It is also the Levi component of a maximal parabolic subgroup of $G$. In fact, order the roots of $\mf{a}$ in $\mf{g}$ by whether they are positive when evaluated on the element $X_{\alpha_0}+X_{-\alpha_0}$. Then the subspace $\mf{n}$ of $\mf{g}$ generated by the positive root spaces for $\mf{a}$ is a nilpotent Lie subalgebra of $\mf{g}$, and there is a unipotent subgroup $N$ of $G$ whose complexified Lie algebra is $\mf{n}$, and such that $P=MN$ is a maximal parabolic subgroup of $G$ with Levi $M$.

Note that the construction of $M$ and $P$ only depend on $\alpha_0$ being a simple root in the ordering on $\Delta$, and otherwise do not depend on any other properties of this ordering.

\subsection{Setup: Discrete series}
\label{subsecds}

We continue with the setting and notation of Section \ref{subsecsetup}.

Let $\rho$ be the half sum of positive roots in $\Delta$. By a \textit{Harish-Chandra parameter} for $G$ we will mean a regular weight $\Lambda$ of $\mf{t}$ such that $\Lambda+\rho$ is integral. We fix throughout a $\Delta^+$-dominant Harish-Chandra parameter $\Lambda$ for $G$ such that
\begin{equation}
\label{eqnlanranlamrho}
\frac{2\langle\Lambda,\alpha_0\rangle}{\langle\alpha_0,\alpha_0\rangle}=\frac{2\langle\rho,\alpha_0\rangle}{\langle\alpha_0,\alpha_0\rangle}=1.
\end{equation}
Then $\Lambda|_{\mf{t}'}$ is a Harish-Chandra parameter for $M_0^\circ$ by \cite[Proposition 4.1]{blank}. It therefore determines a discrete series representation $\pi_0$ of $M_0^\circ$. In the same way, the parameter $\Lambda$ determines a discrete series representation of $G$ which we will denote by $D_+$.

Let $\rho_c$ and $\rho_n$ denote, respectively, the half sum of positive compact or noncompact roots in $\Delta$. Then $\Lambda-\rho_c+\rho_n$ is the Blattner parameter for $D_+$; that is, it is the highest weight of the lowest $K$-type in $D_+$. Let $\xi$ be the character of $T$ whose differential is $\Lambda-\rho_c+\rho_n$ (recall that $T$ is connected). Then $\xi|_{Z_{M_0^\circ}}$ is the central character of $\pi_0$. So we let $\widetilde{\pi}_0$ be the representation of $M_0^\circ Z_{M_0}$ given by
\[\widetilde{\pi}_0=\pi_0\otimes(\xi|_{Z_{M_0}}).\]
That is, for any $m\in M_0^\circ$ and $z\in Z_{M_0}$, we define the product $mz$ to act via $\widetilde{\pi}$ on the space of $\pi$ via $\xi(z)\pi(m)$. Then we induce to get a discrete series representation $\pi$ of $M_0$ by letting
\[\pi=\Ind_{M_0^\circ Z_{M_0}}^{M_0}(\widetilde{\pi}_0).\]
It turns out $\pi$ is irreducible, and we will even see this from an alternate interpretation of $\pi$ in Section \ref{subsecstandardmods}; see Corollary \ref{corpiisirr} below.

We now define another discrete series representation of $G$, as follows. Consider the parameter $\Lambda-\alpha_0$; we claim this parameter is dominant regular for the positive system $w_{\alpha_0}(\Delta^+)$ in $\Delta$, where $w_{\alpha_0}$ is the simple reflection in $\Delta$ corresponding to $\alpha_0$. Indeed, we compute using \eqref{eqnlanranlamrho} that if $\beta\in\Delta$, then
\begin{align*}
\langle\Lambda-\alpha_0,w_{\alpha_0}(\beta)\rangle=& \left\langle\Lambda-\alpha_0,\beta-\frac{2\langle\alpha_0,\beta\rangle}{\langle\alpha_0,\alpha_0\rangle}\alpha_0\right\rangle\\
=&\langle\Lambda,\beta\rangle-\langle\alpha_0,\beta\rangle-\frac{2\langle\alpha_0,\beta\rangle}{\langle\alpha_0,\alpha_0\rangle}\langle\Lambda,\alpha_0\rangle+\frac{2\langle\alpha_0,\beta\rangle}{\langle\alpha_0,\alpha_0\rangle}\langle\alpha_0,\alpha_0\rangle\\
=&\langle\Lambda,\beta\rangle,
\end{align*}
and so in fact, this shows that $\Lambda-\alpha_0=w_{\alpha_0}(\Lambda)$.

Thus we can define the discrete series representation $D_-$ of $G$ to be that corresponding to $\Lambda-\alpha_0$; its Blattner parameter is thus
\[(\Lambda-\alpha_0)-\rho_c+(\rho_n-\alpha_0)=(\Lambda-\rho_c+\rho_n)-2\alpha_0,\]
which differs from that of $D_+$ by $2\alpha_0$. Let $\xi'$ be the character of $T$ whose differential is the Blattner parameter of $D_-$. Then we claim that $\xi'|_{Z_{M_0}}=\xi|_{Z_{M_0}}$. Indeed, we will show below in Proposition \ref{propaboutgroups} (which does not depend on the material in this subsection) that $Z_{M_0}\subset T$ so that this makes sense, and that $Z_{M_0}/(M_0^\circ\cap Z_{M_0})$ has at most $2$ elements. But the root $\alpha_0$ vanishes on $T\cap M_0^\circ$; to see this, note first that it is easy to see that $T\cap M_0^\circ=C_{M_0^\circ}(\mf{t}')$ is the Cartan subgroup of $M_0^\circ$ attached to $\mf{t}'$, and is thus connected. Therefore it is equal to $\exp(\mf{g}_\R\cap\mf{t}')$. But $\alpha_0$ vanishes on $\mf{t}'$ since $\mf{t}'$ is the orthogonal complement of $H_{\alpha_0}$ in $\mf{t}$. Thus $\alpha_0$ vanishes on $Z_{M_0}\cap M_0^\circ$, and so $(2\alpha_0)|_{Z_{M_0}}$ is trivial since $Z_{M_0}/(M_0^\circ\cap Z_{M_0})$ has at most $2$ elements, proving the claim.

It follows that if $\pi_0'$ is the discrete series representation of $M_0^\circ$ with Harish-Chandra parameter $(\Lambda-\alpha_0)|_{\mf{t}'}$, then $\pi_0'=\pi_0$ because $\alpha_0|_{\mf{t}'}=0$. Hence
\[\Ind_{M_0^\circ Z_{M_0}}^{M_0}(\pi_0'\otimes(\xi'|_{Z_{M_0}}))=\pi=\Ind_{M_0^\circ Z_{M_0}}^{M_0}(\pi_0\otimes(\xi|_{Z_{M_0}})).\]

Let us write
\begin{equation}
\label{eqnsnaught}
s_0=(2\rho_P(X_{\alpha_0}+X_{-\alpha_0}))^{-1},
\end{equation}
where $\rho_P$ denotes the half sum of positive roots of $\mf{a}$ in $\mf{n}$, counted with multiplicity. The following theorem is due to Schmid in unpublished work, according to \cite{blank}; \textit{loc. cit.} reproves it and gives integral formulas for the intertwining maps which the theorem concerns. 

\begin{theorem}[Schmid, Blank]
\label{thmschmidblank}
There are nontrivial $(\mf{g},K)$-equivariant maps
\[\Ind_{P}^{G}(\pi,-s_0)\to D_+\qquad\textrm{and}\qquad\Ind_{P}^{G}(\pi,-s_0)\to D_-\]
where we recall that the induced representations above are unitarily normalized, and that we are implicitly taking smooth, $K$-finite vectors in them.
\end{theorem}

Actually, in the next subsection, we will give a separate proof of this theorem using cohomological induction. The proof we will give allows us to deduce more about the structure of the induced representation appearing above.

\subsection{Standard modules}
\label{subsecstandardmods}
We keep the notation of Sections \ref{subsecsetup} and \ref{subsecds}. In particular we have our subgroup $A$ and its centralizer $M=M_0A$. In this subsection, we describe a construction of the discrete series representations $\pi$, $D_+$ and $D_-$ from Section \ref{subsecds} can be used to reprove Theorem \ref{thmschmidblank} while also giving important information about all the constituents of the induced representation $\Ind_{P}^{G}(\pi,-s_0)$ from that theorem. The construction makes use of cohomological induction functors; see \cite{knvo} for a reference for this theory, though we will recall all the facts from this theory that we need.

Recall that we have written $\theta$ for the Cartan involution of $G$ which gives $K$; let us use the same notation $\theta$ for the corresponding Cartan involution on $\mf{g}$ which gives $\mf{k}$. Part of the data that is input into these cohomological induction functors involves a $\theta$-stable parabolic subalgebra of $\mf{g}$; we fix two such. Let $\mf{b}$ denote the Borel subalgebra of $\mf{g}$, with Levi factor $\mf{t}$, which is standard for the system of positive roots $\Delta^+$ we fixed on $\Delta(\mf{g},\mf{t})$. Let $\mf{q}=\C X_{-\alpha_0}\oplus\mf{b}$, which has Levi factor of rank $1$ containing $\pm\alpha_0$. Then the compact Cartan subgroup $T$ is the normalizer of $\mf{b}$ in $G$.

Let $H=C_G(\mf{h})$, where $\mf{h}=\mf{a}\oplus\mf{t}'$ is as in Section \ref{subsecsetup}. Then $H$ is a $\theta$-stable Cartan subgroup in the sense of \cite[pp. 239, 250]{knvo} by \cite[Proposition 4.38 (d)]{knvo}. By \cite[Lemma 4.37]{knvo}, we have that $H=T'A$ where $T'=H\cap K$.

Let $L$ denote the normalizer of $\mf{q}$ in $G$, and $\mf{l}$ its Lie algebra, which is the Levi factor of $\mf{q}$. Then $L$ contains the subgroup $A$, and the derived group $L^d$ of $L$ is split and simple of rank $1$. It is connected by \cite[Lemma 5.10]{knvo}, and we therefore have that $L^{\mr{der}}$ is isomorphic to $SL_2(\R)$ or $PSL_2(\R)$ (recall that $G$ has a faithful finite dimensional representation by assumption, so $L^{\mr{der}}$ cannot be a nontrivial cover of either of these). By the same reasoning as above, $H$ is a $\theta$-stable Cartan subgroup of $L$ as well.

\begin{proposition}
\label{propaboutgroups}
Let the notation be as above.
\begin{enumerate}[label=(\alph*)]
\item We have $H\subset M$ and $H\subset L$.
\item We have $Z_{M_0}\subset H$ and $Z_{M_0}\subset T$.
\item The groups $\pi_0(H)$ and $Z_{M_0}/(M_0^\circ \cap Z_{M_0})$ have at most two elements.
\end{enumerate}
\end{proposition}

\begin{proof}
For (a), first note that we have
\[H=C_G(\mf{h})\subset C_G(\mf{a})=C_G(A)=M,\]
proving the first inclusion.

For the second, we remark that by \cite[Proposition 4.38 (d)]{knvo}, we have
\begin{equation}
\label{eqnHequalsintngb}
H=C_G(\mf{h})=N_G(\mf{b}')\cap N_G(\theta\mf{b}'),
\end{equation}
where $\mf{b}'$ is any Borel subalgebra of $\mf{g}$ with Levi $\mf{h}$. Also, if $C_{\alpha_0}$ denotes the Cayley transform as in Section \ref{subsecsetup}, then $C_{\alpha_0}$ acts by the adjoint action via some element in $L_\C$, the complexification of $L$. Thus $C_{\alpha_0}(\mf{b})\subset\mf{q}$. Similarly, letting $w_{\alpha_0}$ denote the simple reflection in the Weyl group of $\mf{t}$ in $\mf{g}$ associated with $\alpha_0$, we also have $C_{\alpha_0}(w_{\alpha_0}(\mf{b}))\subset\mf{q}$. Since $w_{\alpha_0}(\mf{b})+\mf{b}=\mf{q}$, we therefore have
\begin{equation}
\label{eqnmfqequalsbsum}
\mf{q}=C_{\alpha_0}(\mf{b})+C_{\alpha_0}(w_{\alpha_0}(\mf{b})).
\end{equation}
So combining \eqref{eqnHequalsintngb} and \eqref{eqnmfqequalsbsum}, we get
\[H\subset N_G(C_{\alpha_0}(\mf{b}))\cap N_G(C_{\alpha_0}(w_{\alpha_0}(\mf{b})))\subset N_G\Bigl(C_{\alpha_0}(\mf{b})+C_{\alpha_0}(w_{\alpha_0}(\mf{b}))\Bigr)=N_G(\mf{q})=L,\]
as desired.

For (b), we have $\mf{h}\subset\mf{m}_0$ because $\mf{h}$ is abelian and contains $\mf{a}$, and $\mf{m}_0=C_{\mf{g}}(\mf{a})$. Therefore,
\[Z_{M_0}\subset C_G(\mf{m}_0)\subset C_G(\mf{h})=H,\]
by \eqref{eqnHequalsintngb}. This proves the first inclusion, and we will prove the second after we prove (c) below.

For (c), let $\varphi:SL_2(\R)\to L^{\mr{der}}$ be the homomorphism whose differential sends $\sm{1&0\\ 0&-1}$ to $X_{\alpha_0}+X_{-\alpha_0}\in\mf{a}$, and $\sm{1&1\\ 0&1}$ to $C_{\alpha_0}^{-1}(X_{\alpha_0})$. Then $\varphi$ is either an isomorphism or a $2$-to-$1$ map killing the center.

The group $T$ is connected, and is thus equal to $\exp(\mf{t})$. It follows that $\varphi^{-1}(T\cap L^{\mr{der}})$ is the usual $SO_2$ embedded in $SL_2(\R)$ via its standard embedding, because its Lie algebra is $C_{\alpha_0}(\mf{a})=\C H_{\alpha_0}$. Thus it also follows that $\exp(\mf{t}')\times\varphi(SO_2)=T$. Because $T=\exp(\mf{t})$ and $\mf{t}\subset\mf{l}$, we have $T\subset L$, and therefore $T=C_L(\mf{t})$. It follows that $Z_L\subset T$. But clearly the largest subgroup of $T$ which is central in $L$ is $\exp(\mf{t}')\times\varphi(\{\pm 1\})$, and therefore $Z_L=\exp(\mf{t}')\times\varphi(\{\pm 1\})$.

By a similar argument using (a), we have $H=C_L(\mf{h})$, and so $H$ contains $Z_L A$. But since $\varphi^{-1}(A)\cdot\{\pm 1\}$ is the standard diagonal torus in $SL_2(\R)$, it is its own centralizer in $L^{\mr{der}}$. Thus $C_{L^{\mr{der}}}(A)=A\cdot\varphi(\{\pm 1\})$, and it follows that $H=Z_L A=\exp(\mf{t}')\times\varphi(\{\pm 1\})\times A$, which has at most $2$ components. Thus $\#\pi_0(H)\leq 2$.

Now $\exp(\mf{h})\subset M_0^\circ$ since $\mf{h}\subset\mf{m}_0$, and so it follows that $H/(H\cap M_0^\circ)$ has at most $2$ elements. By the first inclusion in (b) we may intersect further with $Z_{M_0}$:
\[Z_{M_0}/(M_0^\circ \cap Z_{M_0})\cong (M_0^\circ Z_{M_0})/M_0^\circ\subset (M_0^\circ H)/M_0^\circ\cong H/(H\cap M_0^\circ),\]
which shows that $Z_{M_0}/(M_0^\circ \cap Z_{M_0})$ has at most $2$ elements as well.

We now finish the proof of (b). Since $Z_{M_0}\subset H$, and $H=\exp(\mf{t}')\times\varphi(\{\pm 1\})\times A$ by the proof of (c) above, we have $Z_{M_0}\subset \exp(\mf{t}')\times\varphi(\{\pm 1\})$; indeed, otherwise $Z_{M_0}$ would have nontrivial projection to $A$ through the product above, and would thus be infinite. But clearly $Z_{M_0}$ has finite center. Thus $Z_{M_0}\subset\exp(\mf{t}')\times\varphi(\{\pm 1\})\subset T$, as desired.
\end{proof}

We now introduce a character of $T$ and a character of $L$, both of which will be denoted by $\lambda$. Let $\rho$ be the half sum of positive roots in $\Delta$, and $\Lambda$ the $\Delta^+$-dominant Harish-Chandra parameter fixed in Section \ref{subsecds}. We write $\lambda$ for the character of $T$ whose differential is $\Lambda-\rho$. Then because $\Lambda$ is a Harish-Chandra parameter, this does indeed give a well defined character of $T$. By the assumption made in \eqref{eqnlanranlamrho}, the weight $\Lambda-\rho$ is orthogonal to $\alpha_0$ and is thus trivial on $\mf{t}\cap\Lie(L^{\mr{der}})_\C$. It therefore defines a character of $L$, which we also denote by $\lambda$ by abuse of notation.

We will \textit{cohomologically induce} $\lambda$ to $G$ in three different ways as follows. We can regard the character $\lambda$ of $L$ as a one-dimensional $(\mf{l},T)$-module; note that $L$ is connected and contains $T$, which is thus a maximal compact subgroup of $L$ contained in $K$, and therefore $L\cap K=T$. Then we can consider the cohomologically induced module
\[A_\mf{q}(\lambda)=\mc{R}_\mf{q}^{S}(\lambda);\]
here, the number $S$ is the number of compact roots in the radical of $\mf{q}$, the right hand side is the definition of the left hand side, and the right hand side itself is defined in \cite[Section V.1]{knvo}. In general, the notation $\mc{R}_\mf{q}^{i}$ represents the $i$th \textit{cohomological induction functor} from $(\mf{l},T)$-modules to $(\mf{g},K)$-modules along the $\theta$-stable parabolic subalgebra $\mf{q}$, so in particular $\mc{R}_\mf{q}^{S}(\lambda)$ is by definition a $(\mf{g},K)$-module. We will try to recall in detail all of the facts about these functors we will need.

Similarly, viewing $\lambda$ as a $(\mf{t},T)$-module, we can define the $(\mf{g},K)$-modules
\[A_\mf{b}(\lambda)=\mc{R}_\mf{b}^{S}(\lambda),\qquad A_{w_{\alpha_0}\mf{b}}(\lambda)=\mc{R}_{w_{\alpha_0}\mf{b}}^{S}(\lambda),\]
where $w_{\alpha_0}$ denotes the simple reflection in the Weyl group of $\Delta(\mf{g},\mf{t})$ associated with $\alpha_0$. Note that $S$ is also the number of compact roots in the radicals of $\mf{b}$ or $w_{\alpha_0}\mf{b}$.

In general, the number of compact roots in the radical of a given $\theta$-stable parabolic subalgebra is what is \textit{a priori} the middle degree for the cohomological induction functors. It turns out that these functors also vanish above this degree; see \cite[Theorem 5.35]{knvo}. For modules with infinitesimal character in the \textit{good range} (see \cite[Definition 0.49]{knvo}), the cohomological induction functors are concentrated in this middle degree as well \cite[Theorem 5.99 (b)]{knvo}. The modules denoted $\lambda$ above have infinitesimal character in this good range, and we will use this vanishing theorem for them below.

The cohomological induction functors also have an understood effect on infinitesimal characters. Since $(\mf{g},K)$ is in the \textit{Harish-Chandra class} (our $G$ being connected semisimple implies this condition, see \cite[Definition 4.29]{knvo} and the example afterward), by \cite[Corollary 5.25 (b)]{knvo} our cohomologically induced modules have infinitesimal characters given by the inducing module plus the half sum of all roots in the radical of the $\theta$-stable parabolic subalgebra. Thus, since the trivial representation of $L^{\mr{der}}$ has infinitesimal character $\frac{1}{2}\alpha_0$, it follows from this that all the modules $A_\mf{q}(\lambda)$, $A_\mf{b}(\lambda)$, and $A_{w_{\alpha_0}\mf{b}}(\lambda)$ have infinitesimal character given by the Weyl orbit of $\Lambda$.

Our goal now is to prove the following theorem.

\begin{theorem}
\label{thmstdmodforind}
Let the notation be as above.
\begin{enumerate}[label=(\alph*)]
\item The $(\mf{g},K)$-modules $A_\mf{q}(\lambda)$, $A_\mf{b}(\lambda)$, and $A_{w_{\alpha_0}\mf{b}}(\lambda)$ are all unitary and irreducible.
\item Let $\pi$ be the discrete series representation of $M$, and $D_+$ and $D_-$ those of $G$, introduced in Section \ref{subsecds}. Let $J$ be the unique irreducible quotient of $\Ind_P^G(\pi,s_0)$, with $s_0$ as in \eqref{eqnsnaught}. Then we have isomorphisms
\[A_\mf{b}(\lambda)\cong D_+,\qquad A_{w_{\alpha_0}\mf{b}}(\lambda)\cong D_-,\qquad A_\mf{q}(\lambda)\cong J.\]
\item We have exact sequences
\begin{equation}
\label{eqnsespluss0}
0\to D_+\oplus D_-\to \Ind_P^G(\pi,s_0)\to J\to 0,
\end{equation}
and
\begin{equation}
\label{eqnsesminuss0}
0\to J \to \Ind_P^G(\pi,-s_0)\to D_+\oplus D_-\to 0.
\end{equation}
\end{enumerate}
\end{theorem}

The proof will occupy the rest of this subsection. The main point will be to write the induced representation $\Ind_P^G(\pi,-s_0)$ as a \textit{standard module} in two different ways, and then invoke a theorem in Vogan's book \cite{voganbook} that says these two presentations as standard modules are equivalent. (We remark that this theorem is also, more or less, present in \cite{knvo} where they prove their Theorem 11.216, but the statement of that theorem is made only for Langlands quotients and not the standard modules themselves. We will prefer the presentation in \cite{voganbook} for our purposes.)

These two different kinds of standard modules are attached to two different kinds of data. One such kind of data is called a $\theta$\textit{-stable datum}, defined in \cite[Definition 6.5.1]{voganbook}. The standard module attached to a $\theta$-stable datum is of the form
\[(\textrm{cohomological induction})\circ(\textrm{parabolic induction}),\]
where the representation being induced is determined by this datum.

The other kind of data is a \textit{cuspidal datum}, defined in \cite[Definition 6.6.11]{voganbook}. The standard module attached to a cuspidal datum is of the reverse form,
\[(\textrm{parabolic induction})\circ(\textrm{cohomological induction}).\]

The link between these two different kinds of data is bridged by a third notion of data, called a \textit{character datum}, defined in \cite[Definition 6.6.1]{voganbook}. In fact, there is a bijection between the set of $\theta$-stable data for $G$ and the set of character data for $G$, which is given explicitly in \cite[Proposition 6.6.2]{voganbook}. Then, moreover, there is a surjection from the set of character data for $G$ to the set of cuspidal data for $G$, given explicitly in \cite[Lemma 6.6.12]{voganbook}. Thus, composing these two maps, one obtains a cuspidal datum for $G$ starting from a $\theta$-stable datum for $G$, and the point is that \cite[Theorem 6.6.15]{voganbook} asserts that the standard modules attached to both are isomorphic.

We now trace this line of reasoning in our setting. Consider the quadruple $(\mf{q},H,\lambda|_{T'},\nu)$ with $\mf{q}$, $H$, $\lambda$, and $T'$ defined previously in this subsection, and with $\nu$ the character $A\to\R_{>0}$ whose differential is $-\frac{1}{2}\tp{C}_{\alpha_0}(\alpha_0)$, i.e., it is the inverse of the square root of the modulus character for the standard Borel in $L$ with Levi $H$. Then $(\mf{q},H,\lambda|_{T'},\nu)$ is a $\theta$-stable datum as in \cite[Definition 6.5.1]{voganbook}. In fact, the only thing to check is that condition (c) of that definition, which asserts that $\lambda|_{T'}$ is \textit{fine} with respect to $L$. this condition means that $\lambda|_{L^{\mr{der}}\cap T'}$ is trivial. But this follows because $\langle\Lambda-\rho,\alpha_0\rangle=0$ by \eqref{eqnlanranlamrho}.

We build a standard module associated with $(\mf{q},H,\lambda|_{T'},\nu)$ as in \cite[Definition 6.5.2]{voganbook} as follows. Let $B_L$ be the standard Borel in $L$ with Levi $H$, so that $\nu$ is negative with respect to $B_L$ as required in \textit{loc. cit.} Then $(\lambda|_{T'})\otimes\nu$ is a character of $H$ and we can take the unitary parabolic induction to $L$, and then cohomologically induce to $G$; that is to say, our standard module is
\[X(\mf{q},H,\lambda|_{T'},\nu)=\mc{R}_\mf{q}^S(\Ind_{B_L}^L((\lambda|_{T'})\otimes\nu)).\]
We remark that the parabolic induction that appears above is still normalized to be unitary.

We now use the process of \cite[Proposition 6.6.2]{voganbook} to construct a character datum from $(\mf{q},H,\lambda|_{T'},\nu)$. Let $\mf{u}$ be the radical of $\mf{q}$, and let $\mf{p}=\mf{g}^{\theta=-1}$. Define a character $\Gamma$ of $H$ by
\[\Gamma|_A=\nu\quad\textrm{and}\quad \Gamma|_{T'}=\lambda|_{T'}\otimes\sideset{}{^\mathrm{top}}\bigwedge(\mf{u}\cap\mf{p}),\]
and a weight $\gamma$ of $\mf{h}$ by
\[\gamma|_{\mf{t}'}=(\Lambda-\alpha_0)|_{\mf{t}'}\quad\textrm{and}\quad\gamma|_{\mf{a}}=-\frac{1}{2}\tp{C}_{\alpha_0}(\alpha_0)=d\nu.\]
Then by \cite[Proposition 6.6.2]{voganbook}, the pair $(\Gamma,\gamma)$ is a character datum as in \cite[Definition 6.6.1]{voganbook}.

We pass from this character datum to a cuspidal datum. Consider the triple $(M,\delta,\nu)$ with $\nu$ and $M=M_0A$ as above, and with $\delta$ the discrete series representation of $M_0$ given by the following. Consider the pair $(\Gamma',\gamma')$ given by restriction of $(\Gamma,\gamma)$ to $T'$:
\[\Gamma'=\Gamma|_{T'}\quad\textrm{and}\quad \gamma'=\gamma|_{\mf{t}'}.\]
Then $(\Gamma',\gamma')$ is a character datum for $M_0$, and hence it comes from a cuspidal datum $(\mf{b}\cap\mf{m}_0,T',\lambda',1)$, where the entry $1$ is the trivial character of the trivial group, and $\lambda'$ is the character of $T'$ necessarily given by the formula
\begin{equation}
\label{eqnlambdaprimedef}
\lambda'=\Gamma|_{T'}\otimes\left(\sideset{}{^\mathrm{top}}\bigwedge(\mf{u}'\cap\mf{m}_0\cap\mf{p})\right)^{-1},
\end{equation}
where $\mf{u}'$ is the radical of $\mf{b}$. Letting $S'$ be the number of positive compact roots of $\mf{t}'$ in $\mf{m}_0$, we then define
\[\delta=\mc{R}_{\mf{b}\cap\mf{m}_0}^{S'}(\lambda'),\]
so that $\delta$ is the standard module $X(\mf{b}\cap\mf{m}_0,T',\lambda',1)$. The representation $\delta$ is indeed a discrete series representation of the (possibly disconnected) group $M_0$, as follows from the connected case by \cite[Theorem 11.178]{knvo}.

By \cite[Lemma 6.6.12]{voganbook}, the triple $(M,\delta,\nu)$ is a cuspidal datum as in \cite[Definition 6.6.11]{voganbook}. With this datum in hand, we can construct another standard module as follows. The parabolic subgroup $P$ of $G$ with Levi $M=M_0A$ defined in Section \ref{subsecsetup} is the parabolic defined by $-\nu$, this notion being explained in \cite[Definition 6.6.13]{voganbook}. Thus we can define the standard module
\[X(M,\delta,\nu)=\Ind_P^G(\delta\otimes\nu),\]
the parabolic induction again normalized unitarily. We then have the following, which is an immediate consequence of \cite[Theorem 6.6.15]{voganbook}.

\begin{proposition}
Notation as above, there is an isomorphism
\[X(\mf{q},H,\lambda|_{T'},\nu)\cong X(M,\delta,\nu).\]
That is,
\[\mc{R}_\mf{q}^S(\Ind_{B_L}^L((\lambda|_{T'})\otimes\nu))\cong\Ind_P^G(\mc{R}_{\mf{b}\cap\mf{m}_0}^{S'}(\lambda')\otimes\nu).\]
\end{proposition}

Now, to return to the proof of Theorem \ref{thmstdmodforind}, the first thing we will do is identify the standard module $X(M,\delta,\nu)$ above with the induced representation $\Ind_P^G(\pi,-s_0)$. Afterwards, we will identify the constituents of the other standard module $\mc{R}_\mf{q}^S(\Ind_{B_L}^L((\lambda|_{T'})\otimes\nu))$, which will then suffice to prove Theorem \ref{thmstdmodforind} except for the exact sequence \eqref{eqnsespluss0} in part (c) of this theorem. We will then use an argument involving Hermitian duals and the other exact sequence \eqref{eqnsesminuss0} to complete the proof entirely.

Thus, we begin with the following.

\begin{proposition}
There is an isomorphism $\delta\cong\pi$.
\end{proposition}

\begin{proof}
We recall that $\pi$ is obtained in three steps as follows. First, we take $\pi_0$ to be the discrete series representation of $M_0^\circ$ with Harish-Chandra parameter $\Lambda|_{\mf{t}'}$. Then we take $\widetilde{\pi}_0=\pi_0\otimes(\xi|_{Z_{M_0}})$ as a representation of $M_0^\circ Z_{M_0}$, where $\xi$ is the character of $T$ whose differential is the Blattner parameter $\Lambda-\rho_c+\rho_n$. Finally, we take $\pi=\Ind_{M_0^\circ Z_{M_0}}^{M_0}(\widetilde{\pi}_0)$.

We may alternatively describe the representation $\widetilde{\pi}_0$ as the representation of $M_0^\circ Z_{M_0}$ whose restriction to $M_0^\circ$ is $\pi_0$ and whose central character is $\xi|_{Z_{M_0}}$. Note that the character $\Gamma$ that occurs in our $\theta$-stable datum is exactly $\xi|_{T'}$ by definition. Thus the character $\lambda'$ of \eqref{eqnlambdaprimedef} has $\lambda'|_{Z_{M_0}}=\xi|_{Z_{M_0}}$, since the restriction to $Z_{M_0}$ of the character $\bigwedge^\mathrm{top}(\mf{u}'\cap\mf{m}_0\cap\mf{p})$ is trivial as $Z_{M_0}$ centralizes $\mf{m}_0$. (We also note that the restriction $\lambda'|_{Z_{M_0}}$ makes sense in the first place since Proposition \ref{propaboutgroups} (b) implies $Z_{M_0}\subset T'$.) To avoid confusion, let us write $(\mc{R}_{\mf{t}',T'}^{\mf{m}_0,(K\cap M_0^\circ)Z_{M_0}})^{S'}$ for the cohomological induction functor from $(\mf{t}',T')$-modules to $(\mf{m}_0,(K\cap M_0^\circ)Z_{M_0})$-modules, and $(\mc{R}_{\mf{t}',T'}^{\mf{m}_0,K\cap M_0})^{S'}$ for that from $(\mf{t}',T')$-modules to $(\mf{m}_0,K\cap M_0)$-modules. Then \cite[Proposition 11.192]{knvo} tells us that $(\mc{R}_{\mf{t}',T'}^{\mf{m}_0,(K\cap M_0^\circ)Z_{M_0}})^{S'}(\lambda')$ is the discrete series representation of $M_0^\circ Z_{M_0}$ with lowest $K$-type with a highest weight, call it $\mu$, such that
\[\mu|_{(T')^\circ}=(\lambda'|_{(T')^\circ})\cdot\sideset{}{^\mathrm{top}}\bigwedge(\mf{u}'\cap\mf{m}_0\cap\mf{p}),\]
and such that $\mu|_{Z_{M_0}}=\xi|_{Z_{M_0}}$. But by \cite[Lemma 5.3.29]{voganbook}, we have
\[\sideset{}{^\mathrm{top}}\bigwedge(\mf{u}'\cap\mf{m}_0\cap\mf{p})|_{(T')^\circ}=\sideset{}{^\mathrm{top}}\bigwedge(\mf{u}'\cap\mf{p})|_{(T')^\circ}.\]
Thus the differential of $\mu$ is the Blattner parameter of $\widetilde{\pi}_0$, and $\mu|_{Z_{M_0}}$ is the central character of $\widetilde{\pi}_0$, hence $(\mc{R}_{\mf{t}',T'}^{\mf{m}_0,(K\cap M_0^\circ)Z_{M_0}})^{S'}(\lambda')$ is exactly $\widetilde{\pi}_0$.

Finally, \cite[Formula (5.8)]{knvo} then tells us that
\[(\mc{R}_{\mf{t}',T'}^{\mf{m}_0,K\cap M_0})^{S'}(\lambda')=\Ind_{M_0^\circ Z_{M_0}}^{M_0}((\mc{R}_{\mf{t}',T'}^{\mf{m}_0,(K\cap M_0^\circ)Z_{M_0}})^{S'}(\lambda')).\]
The induction on the left hand side is, by above, isomorphic to $\Ind_{M_0^\circ Z_{M_0}}^{M_0}(\widetilde{\pi}_0)$, which is exactly the definition of $\pi$, and this proves the proposition.
\end{proof}

\begin{corollary}
We have an isomorphism
\[X(M,\delta,\nu)\cong\Ind_P^G(\pi,-s_0).\]
\end{corollary}

\begin{proof}
By the proposition it suffices to prove that $\delta_P^{-s_0}|_{A}=\nu$. But by definition of $s_0$ in \eqref{eqnsnaught}, the character $\delta_P^{-s_0}$ differentiates to the one taking $(X_{\alpha_0}+X_{-\alpha_0})$ to $-1$; the character $\nu$ does the same.
\end{proof}

We now begin to describe the other standard module $X(\mf{q},H,\lambda|_{T'},\nu)$. Recall that the derived group $L^{\mr{der}}$ of $L$ is either $SL_2(\R)$ or $PSL_2(\R)$. If $B_2$ denotes the standard upper triangular Borel of $SL_2(\R)$, then the principal series $\Ind_{B_2}^{SL_2(\R)}(\delta_{B_2}^{-1/2})$ decomposes with quotient representation the direct sum of the weight $2$ and weight $-2$ discrete series, and unique irreducible subrepresentation given by the trivial representation. We also recall that $\lambda$ has trivial restriction to $L^{\mr{der}}\cap T$ because of \eqref{eqnlanranlamrho}. So let $\delta_{\pm 2}(\lambda)$ denote the discrete series representation of $L$ with Harish-Chandra parameter $\lambda\pm\frac{1}{2}\alpha_0$. Then $\delta_{\pm 2}(\lambda)|_{L^{\mr{der}}}$ is the weight $\pm 2$ discrete series representation of $L^{\mr{der}}$. For the moment, let us write $\lambda_L$ for the character $\lambda$ viewed as a character of $L$, to distinguish it from the character $\lambda$ of $T$. Then it follows that the induced representation $\Ind_{B_L}^L((\lambda|_{T'})\otimes\nu)$ sits in an exact sequence
\[0\to\lambda_L\to\Ind_{B_L}^L((\lambda|_{T'})\otimes\nu)\to\delta_2(\lambda)\oplus\delta_{-2}(\lambda)\to 0.\]

By the vanishing theorem \cite[Theorem 5.99 (b)]{knvo} and the nature of cohomological induction functors as derived functors, cohomological induction is exact in the good range. It follows that we have an exact sequence
\begin{equation}
\label{eqnexseqforstd}
0\to A_\mf{q}(\lambda)\to X(\mf{q},H,\lambda|_{T'},\nu)\to\mc{R}_\mf{q}^S(\delta_2(\lambda))\oplus\mc{R}_\mf{q}^S(\delta_{-2}(\lambda))\to 0
\end{equation}
Each of the constituents $A_\mf{q}(\lambda)$, $\mc{R}_\mf{q}^S(\delta_{2}(\lambda))$, and $\mc{R}_\mf{q}^S(\delta_{-2}(\lambda))$ is irreducible by the irreducibility theorem \cite[Theorem 8.2]{knvo}, and thus the isomorphisms
\[X(\mf{q},H,\lambda|_{T'},\nu)\cong X(M,\delta,\nu)\cong\Ind_P^G(\pi,-s_0)\]
already established show that $A_{\mf{q}}(\lambda)$ is an irreducible subrepresentation of $\Ind_P^G(\pi,-s_0)$. Strictly speaking, we haven't shown that $\pi$ itself is irreducible yet. We will deduce this below from the following.

\begin{proposition}
\label{propisosbetweends}
We have isomorphisms
\[\mc{R}_\mf{q}^S(\delta_{2}(\lambda))\cong A_\mf{b}(\lambda)\cong D_+,\quad\textrm{and}\quad \mc{R}_\mf{q}^S(\delta_{-2}(\lambda))\cong A_{w_{\alpha_0}\mf{b}}(\lambda)\cong D_-.\]
\end{proposition}

\begin{proof}
By \cite[Theorem 11.178 and Proposition 11.192]{knvo}, the representation $A_\mf{b}(\lambda)$ is the discrete series representation with Blattner parameter $\lambda+2\rho_n$, which is exactly $D_+$, and similarly $A_{w_{\alpha_0}\mf{b}}(\lambda)\cong D_-$. Thus we must show
\[\mc{R}_\mf{q}^S(\delta_{2}(\lambda))\cong A_\mf{b}(\lambda)\quad\textrm{and}\quad \mc{R}_\mf{q}^S(\delta_{-2}(\lambda))\cong A_{w_{\alpha_0}\mf{b}}(\lambda).\]

By \textit{loc. cit.}, we also have
\[\delta_2(\lambda)\cong\mc{R}_{\mf{b}\cap\mf{l}}^{0}(\lambda)\quad\textrm{and}\quad \delta_{-2}(\lambda)\cong\mc{R}_{(w_{\alpha_0}\mf{b})\cap\mf{l}}^{0}(\lambda),\]
where the cohomological induction functors above are from $(\mf{t},T)$-modules to $(\mf{l},T)$-modules. Now there is a spectral sequence \cite[Theorem 11.77]{knvo} which computes the composition of two cohomological induction functors. In the good range, the vanishing theorem \cite[Theorem 5.99 (b)]{knvo} implies that it collapses immediately, and applying this in our case, we get
\[\mc{R}_\mf{q}^S(\mc{R}_{\mf{b}\cap\mf{l}}^{0}(\lambda))\cong A_\mf{b}(\lambda)\quad\textrm{and}\quad\mc{R}_\mf{q}^S(\mc{R}_{(w_{\alpha_0}\mf{b})\cap\mf{l}}^{0}(\lambda))\cong A_{w_{\alpha_0}\mf{b}}(\lambda).\]
This implies the proposition.
\end{proof}

\begin{corollary}
\label{corpiisirr}
The representation $\pi$ (equivalently, $\delta$) is irreducible.
\end{corollary}

\begin{proof}
Note that the proposition along with \eqref{eqnexseqforstd} imply that there is only one nontempered constituent of $\Ind_P^G(\pi,-s_0)$ (which must be $A_\mf{q}(\lambda)$). By the Langlands classification along with the fact that $s_0>0$, we must have $\pi$ is irreducible.
\end{proof}

Thus we have now shown (b) and the second half of (c) of Theorem \ref{thmstdmodforind} with the unique irreducible subrepresentation of $\Ind_P^G(\pi,-s_0)$ in place of $J$ (which we have not yet shown is indeed $J$ itself), and (a) follows from the irreducibility theorem \cite[Theorem 8.2]{knvo} and unitarizability \cite[Theorem 9.1]{knvo}. Thus, to complete the proof of Theorem \ref{thmstdmodforind}, it suffices to establish the existence of an exact sequence
\[0\to A_\mf{b}(\lambda)\oplus A_{w_{\alpha_0}\mf{b}}(\lambda)\to\Ind_P^G(\pi,s_0)\to A_\mf{q}(\lambda)\to 0.\]

To do this, we apply the \textit{Hermitian dual} functor (see \cite[\S VI.2]{knvo}), denoted $(\cdot)^h$, to the exact sequence
\[0\to A_\mf{q}(\lambda)\to \Ind_P^G(\pi,-s_0)\to A_\mf{b}(\lambda)\oplus A_{w_{\alpha_0}\mf{b}}(\lambda)\to 0.\]
This functor is contravariant and exact, and it preserves unitary modules. Applying it to the above sequence, using the unitarizability of the representations on either end, gives
\[0\to A_\mf{b}(\lambda)\oplus A_{w_{\alpha_0}\mf{b}}(\lambda)\to\Ind_P^G(\pi,-s_0)^h\to A_\mf{q}(\lambda)\to 0.\]
But by \cite[Corollary 11.59]{knvo} and the fact that $s_0$ is real, we have
\[\Ind_P^G(\pi,-s_0)^h\cong \Ind_P^G(\pi^h,\overline{s}_0)\cong \Ind_P^G(\pi,s_0),\]
again because $\pi$ is unitary as it is a discrete series representation. This is the exact sequence we wanted to establish, finishing the proof of Theorem \ref{thmstdmodforind}.\qed

\indent

We note the following corollary of the isomorphism $J\cong A_\mf{q}(\lambda)$ which will be important later.

\begin{corollary}
\label{corlowestKtypeJ}
The representation $J$ has a unique lowest $K$-type whose highest weight is $\Lambda-\rho_c+\rho_n-\alpha_0$.
\end{corollary}

\begin{proof}
By \cite[Theorem 5.80 and Corollary 5.72]{knvo}, the $K$-type with highest weight $\lambda\otimes\sideset{}{^\mathrm{top}}\bigwedge(\mf{u}\cap\mf{p})$ is the lowest $K$-type in $A_\mf{q}(\lambda)\cong J$; this weight is exactly
\[\lambda+2\rho_n-\alpha_0=\Lambda-\rho+2\rho_n-\alpha_0=\Lambda-\rho_c+\rho_n-\alpha_0,\]
as desired.
\end{proof}

\subsection{$(\mf{g},K)$-cohomology}
\label{subsecgkcoh}
We keep the notation of Sections \ref{subsecsetup} and \ref{subsecds}. In this subsection we discuss the $(\mf{g},K)$-cohomology of the induced representations from Section \ref{subsecstandardmods}.

Let $s_0$ be the particular number from \eqref{eqnsnaught}. Then by Theorem \ref{thmstdmodforind} (c) we have the two exact sequences
\begin{equation*}
0\to D_+\oplus D_-\to \Ind_P^G(\pi,s_0)\to J\to 0,
\end{equation*}
and
\begin{equation*}
0\to J \to \Ind_P^G(\pi,-s_0)\to D_+\oplus D_-\to 0,
\end{equation*}
of \eqref{eqnsespluss0} and \eqref{eqnsesminuss0}, where $D_+$, $D_-$ and $\pi$ are the discrete series representations from Section \ref{subsecds}, and $J$ is the unique irreducible quotient of $\Ind_P^G(\pi,s_0)$. We examine the $(\mf{g},K)$-cohomology exact sequences (with coefficients) associated with these short exact sequences.

Recall (see \cite{BW} for example) that for any $(\mf{g},K)$-module $\Pi$, the $i$th $(\mf{g},K)$-cohomology group of $\Pi$, denoted $H^i(\mf{g},K;\Pi)$, is computed as the cohomology of the complex $C^\bullet(\mf{g},K;\Pi)$ where
\[C^q(\mf{g},K;\Pi)=\hom_K(\sideset{}{^q}\bigwedge(\mf{g}/\mf{k}),\Pi),\]
where $K$ acts on $\mf{g}/\mf{k}$ by the adjoint action, and where the boundary map $d:C^q(\mf{g},K;\Pi)\to C^{q+1}(\mf{g},K;\Pi)$ is given by the following. Let $X_0,\dotsc,X_q\in\mf{g}$ and $c\in C^q(\mf{g},K;\Pi)$. Then $dc$ is given by the formula
\begin{multline}
\label{eqndefofbdrymap}
(dc)(X_0\wedge\dotsb\wedge X_q)=\sum_{i=0}^q (-1)^iX_i c(X_0\wedge\dotsb\wedge\widehat{X}_i\wedge\dotsb\wedge X_q)\\+\sum_{0\leq i<j\leq q}(-1)^{i+j}c([X_i,X_j]\wedge X_0\wedge\dotsb\wedge\widehat{X}_i\wedge\dotsb\wedge\widehat{X}_j\wedge\dotsb\wedge X_q),
\end{multline}
where the terms $X_i$ in the wedge products are viewed as their images in $\mf{g}/\mf{k}$, and the hats mean we omit that corresponding term from the wedge. Then it turns out that this boundary map $d$ is well defined, i.e., independent of the choice of the $X_i$'s up to addition of elements in $\mf{k}$.

We now show the following.

\begin{proposition}
\label{propcomplexesconc}
Let $\Lambda$ be the weight from Section \ref{subsecds}, and let $E_{\Lambda-\rho}$ be the irreducible representation of $G(\C)$ of highest weight $\Lambda-\rho$, where $\rho$ is as usual the half sum of positive roots of the Cartan subalgebra $\mf{t}$ of Section \ref{subsecsetup} in $\mf{g}$. Let $d=\frac{1}{2}\dim_\C(\mf{g}/\mf{k})$.
\begin{enumerate}[label=(\alph*)]
\item The complexes $C^\bullet(\mf{g},K;D_\pm\otimes E_{\Lambda-\rho}^\vee)$ equal their cohomology, are concentrated in degree $d$, and are $1$ dimensional.
\item The complex $C^\bullet(\mf{g},K;J\otimes E_{\Lambda-\rho}^\vee)$ equals its cohomology, is concentrated in degrees $d-1$ and $d+1$, and is $1$ dimensional in both of these degrees.
\item For any $s\in\C$, the complex $C^\bullet(\mf{g},K;\Ind_P^G(\pi,s)\otimes E_{\Lambda-\rho}^\vee)$ is concentrated in degrees $d-1$, $d$, and $d+1$, where it is, respectively, $1$, $2$, and $1$ dimensional.
\end{enumerate}
\end{proposition}

\begin{proof}
By \cite[Proposition II.3.1 (b)]{BW}, the complexes in parts (a) and (b) equal their cohomology groups because $D_{\pm}$ and $J$ are unitary by Theorem \ref{thmstdmodforind} (a) and (b). Thus, because we have assumed $G$, and hence $K$, is connected, parts (a) and (b) follow immediately from \cite[Theorem 5.5]{VZ}, as well as Theorem \ref{thmstdmodforind} (b) where we have written our representations in the form $A_{\mf{q}}(\lambda)$. The statement of part (c) for the point $s=s_0$ then follows from this and the exact sequence \eqref{eqnsespluss0}. Hence it follows for all $s$, since the set of $K$-types in the induced representation $\Ind_P^G(\pi,s)$ is independent of $s$.
\end{proof}

We make the following remark which will be useful for us in a few situations later.

\begin{remark}
\label{remfourcochains}
By Proposition \ref{propcomplexesconc}, we have four linearly independent cochains
\[c^{d\pm 1}:\sideset{}{^{d\pm 1}}\bigwedge(\mf{g}/\mf{k})\to\Ind_P^G(\pi,s_0)\otimes E_{\Lambda-\rho}^\vee,\quad\textrm{and}\quad c^{d,\pm}:\sideset{}{^d}\bigwedge(\mf{g}/\mf{k})\to\Ind_P^G(\pi,s_0)\otimes E_{\Lambda-\rho}^\vee\]
which we describe explicitly below up to nonzero scalar multiple, and these cochains span the $(\mf{g},K)$-cohomology complex for $\Ind_P^G(\pi,s_0)\otimes E_{\Lambda-\rho}^\vee$. (The exact scalar multiple will not matter for any of our applications.) Since the $K$-types of $\Ind_P^G(\pi,s)$ are independent of $s$, similar descriptions can be made for the cocycles spanning the $(\mf{g},K)$-cohomology complex for $\Ind_P^G(\pi,s)\otimes E_{\Lambda-\rho}^\vee$

To describe these cochains, note first that $\bigwedge^{d-1}(\mf{g}/\mf{k})$ contains the $1$-dimensional subspace $\bigwedge^{d-1}(\mf{u}/(\mf{u}\cap\mf{k}))$, where $\mf{u}$ is the radical of the $\theta$-stable parabolic subalgebra $\mf{q}$ defined in Section \ref{subsecstandardmods}, above Proposition \ref{propaboutgroups}. This is a $\mf{t}$-stable subspace of weight $2\rho_n-\alpha_0$, and since $\alpha_0$ is simple, it is a highest weight of $\bigwedge^{d-1}(\mf{g}/\mf{k})$. Thus $\bigwedge^{d-1}(\mf{g}/\mf{k})$ contains the representation with this highest weight $2\rho_n-\alpha_0$. As well, the representation $J\otimes E_{\Lambda-\rho}^\vee$, and hence also $\Ind_P^G(\pi,s_0)\otimes E_{\Lambda-\rho}^\vee$, contains a $K$-type with highest weight
\[(\Lambda-\rho_c+\rho_n-\alpha_0)-(\Lambda-\rho)=2\rho_n-\alpha_0\]
occurring in the tensor product of the $K$-type in $J$ (or in $\Ind_P^G(\pi,s_0)$) of highest weight $\Lambda-\rho_c+\rho_n-\alpha_0$ given to us by Corollary \ref{corlowestKtypeJ}, with the $K$-type in $E_{\Lambda-\rho}^\vee$ of lowest weight $-(\Lambda-\rho)$. The cochain $c^{d-1}$ must therefore send the $K$-type in $\bigwedge^{d-1}(\mf{g}/\mf{k})$ of highest weight $2\rho_n-\alpha_0$ isomorphically onto this $K$-type in $\Ind_P^G(\pi,s_0)\otimes E_{\Lambda-\rho}^\vee$, and vanish on all other $K$-types in $\bigwedge^{d-1}(\mf{g}/\mf{k})$.

Similarly, the subspace
\[(\C X_{\alpha_0})\wedge(\C X_{-\alpha_0})\wedge\sideset{}{^{d\pm 1}}\bigwedge(\mf{u}/(\mf{u}\cap\mf{k}))\subset\sideset{}{^{d+1}}\bigwedge(\mf{g}/\mf{k})\]
provides a unique $K$-type in $\bigwedge^{d+1}(\mf{g}/\mf{k})$ of highest weight $2\rho_n-\alpha_0$, and the cochain $c^{d+1}$ must send this $K$-type isomorphically onto the $K$-type in $\Ind_P^G(\pi,s_0)\otimes E_{\Lambda-\rho}^\vee$ of highest weight $2\rho_n-\alpha_0$ described above, and vanish on all other $K$-types in $\bigwedge^{d+1}(\mf{g}/\mf{k})$.

Finally, the subspaces
\[(\C X_{\alpha_0})\wedge\sideset{}{^{d-1}}\bigwedge(\mf{u}/(\mf{u}\cap\mf{k}))\subset\sideset{}{^{d}}\bigwedge(\mf{g}/\mf{k})\]
and
\[(\C X_{-\alpha_0})\wedge\sideset{}{^{d-1}}\bigwedge(\mf{u}/(\mf{u}\cap\mf{k}))\subset\sideset{}{^{d}}\bigwedge(\mf{g}/\mf{k})\]
give unique $K$-types of highest weights $2\rho_n$ and $2\rho_n-2\alpha_0$, respectively, in $\bigwedge^{d}(\mf{g}/\mf{k})$. Similar considerations as above show that $D_+\otimes E_{\Lambda-\rho}^\vee$ and $D_-\otimes E_{\Lambda-\rho}^\vee$ contain these respective $K$-types, and they occur in the tensor product of the lowest $K$-types of $D_+$ or $D_-$ with the $K$-type in $E_{\Lambda-\rho}^\vee$ of lowest weight $-(\Lambda-\rho)$. Thus the two cochains $c^{d,+}$ and $c^{d,-}$ are given by projecting onto the unique $K$-type in $\bigwedge^{d}(\mf{g}/\mf{k})$ of highest weights $2\rho_n$ and $2\rho_n-2\alpha_0$, respectively, then mapping these isomorphically onto the $K$-types in $D_+\otimes E_{\Lambda-\rho}^\vee$ or $D_-\otimes E_{\Lambda-\rho}^\vee$ just described, and then finally composing with the inclusions $D_{\pm}\to \Ind_P^G(\pi,s_0)$.
\end{remark}

\begin{proposition}
\label{propcohofinduced}
Notation as in Proposition \ref{propcomplexesconc}, we have
\[\dim_{\C}H^q(\mf{g},K;\Ind_P^G(\pi,s_0)\otimes E_{\Lambda-\rho}^\vee)=\begin{cases}
1 & \textrm{if }q=d,d+1;\\
0 & \textrm{otherwise,}
\end{cases}\]
and
\[\dim_{\C}H^q(\mf{g},K;\Ind_P^G(\pi,-s_0)\otimes E_{\Lambda-\rho}^\vee)=\begin{cases}
1 & \textrm{if }q=d-1,d;\\
0 & \textrm{otherwise.}
\end{cases}\]
\end{proposition}

\begin{proof}
Since $\pi$ is a discrete series representation of the group $M_0$, its $(\mf{m}_0,K\cap P)$-cohomology is concentrated in one degree and is one dimensional. By \cite[Theorem III.3.3]{BW}, which computes the cohomology of induced representations, the cohomology groups
\[H^q(\mf{g},K;\Ind_P^G(\pi,\pm s_0)\otimes E_{\Lambda-\rho}^\vee)\]
 either vanish, or they are $1$-dimensional and concentrated in two consecutive degrees. Proposition \ref{propcomplexesconc} (c) implies that $H^q(\mf{g},K;\Ind_P^G(\pi,\pm s_0)\otimes E_{\Lambda-\rho}^\vee)$ is concentrated in degrees $d-1,d,d+1$, and the long exact sequence associated with the short exact sequence of \eqref{eqnsespluss0} and \eqref{eqnsesminuss0}, along with Proposition \ref{propcomplexesconc} (a) and (b), give isomorphisms
\[H^{d\pm 1}(\mf{g},K;\Ind_P^G(\pi,\pm s_0)\otimes E_{\Lambda-\rho}^\vee)\cong H^{d\pm 1}(\mf{g},K;J\otimes E_{\Lambda-\rho}^\vee)\cong \C.\]
Therefore we must also have
\[H^{d}(\mf{g},K;\Ind_P^G(\pi,\pm s_0)\otimes E_{\Lambda-\rho}^\vee)\cong\C,\]
and vanishing elsewhere.
\end{proof}

The tables below, which follow from the two propositions above, depict the long exact cohomology sequence associated with the short exact sequences \eqref{eqnsespluss0} and \eqref{eqnsesminuss0}. For them, let us use the shorthand notations
\[h^q(\cdot)=\dim_{\C}H^q(\mf{g},K;(\cdot)\otimes E_{\Lambda-\rho}^\vee),\qquad I_{\pm}=\Ind_P^G(\pi,\pm s_0),\qquad D=D_+\oplus D_-.\]
Then we have:
\begin{center}
\begin{tabular}{| c || c | c | c |} 
\hline
$q$ & $h^q(D)$ & $h^q(I_{+})$ & $h^q(J)$ \\
\hline\hline
$d-1$ & $0$ & $0$ & $1$ \\ 
\hline
$d$ & $2$ & $1$ & $0$ \\
\hline
$d+1$ & $0$ & $1$ & $1$ \\
\hline
\end{tabular}
\qquad\qquad
\begin{tabular}{| c || c | c | c |}
\hline
$q$ & $h^q(J)$ & $h^q(I_{-})$ & $h^q(D)$ \\
\hline\hline
$d-1$ & $1$ & $1$ & $0$ \\ 
\hline
$d$ & $0$ & $1$ & $2$ \\
\hline
$d+1$ & $1$ & $0$ & $0$ \\
\hline
\end{tabular}
\end{center}
In particular, we note that when $s=-s_0$ and we are in the case of the right hand table, then we may observe the following corollary.

\begin{corollary}
\label{corbdrynontrivial}
The boundary map
\[H^{d}(\mf{g},K;(D_+\oplus D_-)\otimes E_{\Lambda-\rho}^\vee)\to H^{d+1}(\mf{g},K;J\otimes E_{\Lambda-\rho}^\vee),\]
associated with the exact sequence \eqref{eqnsesminuss0}, is nontrivial.
\end{corollary}

This fact will be crucial for the second proof of our main theorem, Theorem \ref{thmmainthm} below, and the first of the two tables above is crucial for the first proof.

\subsection{Intertwining operators}
\label{subsecintertwining}
The content of this subsection plays no role in the first proof we give of Theorem \ref{thmmainthm} below, but is crucial for the second. Nevertheless, we believe Theorem \ref{thmintertwining}, which is proved here, is of independent interest, and so might be its proof.

We keep the notation of Sections \ref{subsecsetup} and \ref{subsecds}, as well as the notation set at the beginning of Section \ref{subsecstandardmods}. To get started, note that Weyl group of the Cartan subgroup $H$ of $G$ has an element $w_0$, defined on the Lie algebra level as the transposed Cayley transform of the simple reflection corresponding to $\alpha_0$. One checks easily that $w_0Mw_0^{-1}=M$, and that $w_0Pw_0^{-1}=\overline{P}$, where $\overline{P}$ is the parabolic opposite to $P$.

Let $\pi^{w_0}$ be the representation of $M_0$ defined by $\pi\circ(\Ad(w_0))$, where we recall that $\pi$ is the discrete series representation of $M_0$ defined in Section \ref{subsecds}. We start by proving the following.

\begin{lemma}
\label{lemisopiw0pi}
There is an isomorphism $\pi^{w_0}\cong\pi$.
\end{lemma}

\begin{proof}
Let $L$ be the $\theta$-stable Levi of Section \ref{subsecstandardmods}. Since the Cayley transform $C_{\alpha_0}$ is given by the adjoint action by an element of $(L^{\mr{der}})_{\C}$, as is the simple reflection associated with $\alpha_0$, we have that $w_0$ is as well. Therefore $w_0$ fixes $\mf{t}'$ pointwise, and hence also the Harish-Chandra parameter $\Lambda|_{\mf{t}'}$ of the discrete series representation $\pi_0$ of $M_0^\circ$ used to define $\pi$ in Section \ref{subsecds}. Since $Z_{M_0}/(M_0^\circ \cap Z_{M_0})$ has at most two elements by Proposition \ref{propaboutgroups} (c), it is also fixed pointwise by the element $w_0$ because it is fixed setwise. Thus $w_0$ fixes the representation $\widetilde{\pi}_0$ of $M_0^\circ Z_{M_0}$ used to define $\pi$. Finally, since $\pi=\Ind_{M_0^\circ Z_{M_0}}^{M_0}(\widetilde{\pi}_0)$, we have that $w_0$ fixes $\pi$ as well.
\end{proof}

Using the above lemma, recall that one can then define the \textit{intertwining operator}
\[M(s,\cdot):\Ind_P^G(\pi,s)\to\Ind_{\overline{P}}^G(\pi,s);\]
for $\re(s)$ sufficiently large, by the convergent integral
\[M(s,\phi)(g)=\int_N\phi(w_0ng)\,dn,\]
and by meromorphic (actually analytic when $\re(s)>0$) continuation otherwise. Note that the target induced representation should really be $\Ind_{\overline{P}}^G(\pi^{w_0},s)$, but we are invoking the lemma just proved. Also, this representation has that
\[\Ind_{\overline{P}}^G(\pi,s)\cong\Ind_{P}^G(\pi,-s),\]
hence we may view $M(s,\cdot)$ as a map
\[M(s,\cdot):\Ind_P^G(\pi,s)\to\Ind_{P}^G(\pi,-s).\]

Let us be more precise about the meaning of meromorphicity and meromorphic continuation here. Fix any $s_1\in\C$. For any $\phi\in\Ind_P^G(\pi,s_1)$, the Iwasawa decomposition allows us to put $\phi$ into a unique \textit{flat section}; that is, there is a unique family $\phi_s\in\Ind_P^G(\pi,s)$ for $s\in\C$, such that $\phi_{s_1}=\phi$ and $\phi_s(k)$ is independent of $s$ for any $k\in K$. Then as $s$ varies, the section
\[M(s,\phi_s)\in\Ind_{P}^G(\pi,-s)\]
is a finite sum
\[M(s,\phi_s)=\sum_{i}c_i(s)\phi_{i,-s}\]
for some meromorphic functions $c_i(s)$ which are holomorphic when $\re(s)>0$, and some flat sections $\phi_{i,s}\in\Ind_{P}^G(\pi,s)$.

Let $s_0$ be the particular point of \eqref{eqnsnaught}. By Theorem \ref{thmstdmodforind} (c), the induced representation $\Ind_{P}^G(\pi,-s_0)$ has unique irreducible subrepresentation $J$, with quotient
\[\Ind_{P}^G(\pi,-s_0)/J\cong D_+\oplus D_-.\]
Here $D_+$ and $D_-$ are the discrete series representations of $G$ from Section \ref{subsecds}, with Harish-Chandra parameters given by $\Lambda$ and $\Lambda-\alpha_0$, respectively. The same theorem also tells us then that $\Ind_P^G(\pi,s_0)$ has the same constituents $D_+$, $D_-$ and $J$, but in the opposite order; there is an exact sequence
\[0\to D_+\oplus D_-\to \Ind_P^G(\pi,s_0)\to J\to 0.\]
The Langlands classification also tells us that the intertwining operator $M(s_0,\phi)$ vanishes on $D_+\oplus D_-$ and gives a nontrivial map of $(\mf{g},K)$-modules
\[M(s_0,\cdot):\Ind_P^G(\pi,s_0)/(D_+\oplus D_-)\to J\subset\Ind_{P}^G(\pi,-s_0).\]
Our goal in this subsection is to prove that this vanishing is to order $1$. Namely:

\begin{theorem}
\label{thmintertwining}
Let $\phi_s\in \Ind_{P}^G(\pi,s)$ be any flat section such that $\phi_{s_0}\in D_+\oplus D_-$. Then the limit
\[\lim_{s\to s_0}\frac{1}{s-s_0}M(s,\phi_s)\modulo{J}\]
gives a well defined $(\mf{g},K)$-equivariant isomorphism
\[(D_+\oplus D_-)\overset{\sim}{\longrightarrow}\Ind_{P}^G(\pi,-s_0)/J.\]
\end{theorem}

To prove Theorem \ref{thmintertwining}, we will pass from the group setting to the Lie algebra setting in the following way. The induced representation $\Ind_{P}^G(\pi,s)$ may be seen as a space of linear maps $\phi:U(\mf{g})\to\pi$ satisfying certain properties. These properties include
\[\phi(X_{\mf{a}}u)=(2s+1)\rho_P(X_{\mf{a}})\phi(u),\qquad\textrm{for all }u\in U(\mf{g})\textrm{ and }X_{\mf{a}}\in\mf{a},\]
where $\rho_P$ is the half sum of roots of $\mf{a}$ in the Lie algebra of $P$. We also have
\[\phi(X_{\mf{m}_0}u)=X_{\mf{m}_0}\cdot\phi(u),\qquad\textrm{for all }u\in U(\mf{g})\textrm{ and }X_{\mf{m}_0}\in\mf{m}_0,\]
as well as\
\[\phi(X_{\mf{n}}u)=0,\qquad\textrm{for all }u\in U(\mf{g})\textrm{ and }X_{\mf{n}}\in\mf{n}.\]
See, for example, \cite[III.2]{BW}. The universal enveloping algebra $U(\mf{g})$ will then act on $\Ind_{P}^G(\pi,s)$ by right translation in the variable,
\[(u\phi)(u')=\phi(u'u),\qquad\textrm{for all }u,u'\in U(\mf{g}).\]
If $\phi_s\in \Ind_{P}^G(\pi,s)$ is a flat section, this means in the Lie algebra picture that $\phi_s(u)$ is independent of $s$ for $u\in U(\mf{k})$.

In the following, we denote by $\mf{p}\subset\mf{g}$ the $(-1)$-eigenspace for the Cartan involution which gives $\mf{k}$ (not to be confused with the complex Lie algebra of $P$, to which we will grant no special notation). We start with the following easy lemma which will come in handy.

\begin{lemma}
\label{lemhandylemma}
Let $X\in\mf{p}$, and let $u\in U(\mf{k})$ be a monomial of some degree $d$, i.e., $u=Y_1Y_2\dotsb Y_d$ where $Y_1,\dotsc,Y_d\in\mf{k}$. Then there are elements $X_1,\dotsc X_l\in\mf{p}$ for some nonnegative integer $l$, as well as monomials $u_1,\dotsc, u_l\in U(\mf{k})$ of degree strictly smaller than $d$, such that
\[uX=Xu+\sum_{i=1}^l X_i u_i.\]
\end{lemma}

\begin{proof}
We induct on $d$, the lemma being obvious when $d=0$. If $d>0$, we write
\[uX=Y_1\dotsb Y_{d-1}Y_d X=(Y_1\dotsb Y_{d-1}XY_d)+(Y_1\dotsb Y_{d-1}[Y_d,X]).\]
Since $[Y_d,X]\in\mf{p}$, we can apply the induction hypothesis to both terms on the right hand side above.
\end{proof}

We return to induced representations now.

\begin{lemma}
\label{lemsectionslinear}
Let $\phi_s\in\Ind_{P}^G(\pi,s)$ be a flat section, and let $X\in\mf{g}$. Then there are flat sections $\phi_s^{(0)},\phi_s^{(1)}\in\Ind_{P}^G(\pi,s)$ such that
\[X\phi_s=s\phi_s^{(1)}+\phi_s^{(0)}.\]
\end{lemma}

\begin{proof}
First we note that when $s$ is fixed, the section $\phi_s$ belongs to a direct sum of finitely many $K$-types for $\Ind_{P}^G(\pi,s)$ by admissibility. The $K$-types of $\Ind_{P}^G(\pi,s)$ are independent of $s$, and so if $\phi_1',\dotsc,\phi_l'$ form a basis for a finite direct sum of $K$-types in $\Ind_{P}^G(\pi,s_1)$ for fixed $s_1\in\C$, then the flat sections $\phi_{1,s}',\dotsc,\phi_{l,s}'$ they give rise to form a basis of the direct sum of the same $K$-types for any $s\in\C$ upon specialization. It is straightforward to show that if $\pi$ is a $(\mf{g},K)$ module and $v\in\pi$ is in a $K$-type of some highest weight $\mu$, then the sum of $K$-types to which $Xv$ belongs have highest weight at most $\mu+2\beta_h$ where $\beta_h$ is the highest root. It thus follows that there is a fixed finite collection of $K$-types, the sum of which contains $X\phi_s$ for any $s\in\C$.

Thus $X\phi_s$ may be written as
\begin{equation}
\label{eqnxphissum}
X\phi_s=\sum_{i=1}^n a_i(s)\phi_{i,s},
\end{equation}
for some $\C$-valued functions $a_i(s)$ and linearly independent flat sections $\phi_{i,s}$ (this linear independence being independent of $s$). The transformation laws satisfied by sections in $\Ind_{P}^G(\pi,s)$, along with the Iwasawa decomposition, imply that such sections are determined by their values on $U(\mf{k})$. So the linear independence of the sections $\phi_{i,s}$ implies that there are elements $u_1,\dotsc, u_n\in U(\mf{k})$ such that the vectors $v_i=(\phi_{i,s}(u_1),\dotsc,\phi_{i,s}(u_n))\in \pi^{\oplus n}$,  for $i=1,\dotsc,n$, are linearly independent in $\pi^{\oplus n}$. Thus there is a linear functional $L:\pi\to\C$ such that the vectors $L(v_i)$ defined by
\[L(v_i)=(L(\phi_{i,s}(u_1)),\dotsc,L(\phi_{i,s}(u_n)))\]
are linearly independent in $\C^n$. Writing $B$ for the $n\times n$ matrix with entries in $\pi$ whose $i$th column is $\tp{v}_i$, we thus have that $L(B)\in GL_n(\C)$, where by $L(B)$ we of course mean the matrix obtained by evaluating $L$ on every entry of $B$.

Now for $s\in\C$, let $\mbf{a}(s)$ be the vector
\[\mbf{a}(s)=(a_1(s),\dotsc,a_n(s))\in\C^n,\]
and let $\mbf{w}(s)$ be the vector
\[\mbf{w}(s)=((X\phi_s)(u_1),\dotsc,(X\phi_s)(u_n))\in\pi^{\oplus n}.\]
Then what we have said may be rephrased as
\begin{equation}
\label{eqnwsasb}
\mbf{w}(s)=\mbf{a}(s)B.
\end{equation}

Let us write $X=X_{\mf{p}}+X_{\mf{k}}$ with respective parts in $\mf{p}$ and $\mf{k}$. Then for any $i=1,\dotsc,n$, we have
\[(X\phi_s)(u_i)=\phi_s(u_iX_{\mf{p}})+\phi_s(u_iX_{\mf{k}}).\]
So by Lemma \ref{lemhandylemma}, there are elements $Y_1,\dotsc,Y_l\in\mf{p}$ and $u_{i,1},\dotsc,u_{i,l}\in U(\mf{k})$ for some nonnegative integer $l$ such that
\[(X\phi_s)(u_i)=\phi_s(u_iX_{\mf{k}})+\sum_{j=1}^l\phi_s(Y_j u_{i,j}).\]
For each $j=1,\dotsc,l$, let us write $Y_j$ as
\[Y_j=Y_{j,\mf{m}_0}+Y_{j,\mf{a}}+Y_{j,\mf{n}}+Y_{j,\mf{k}},\]
in any way, with respective parts in $\mf{m}_0$, $\mf{a}$, $\mf{n}$ and $\mf{k}$. Then we have
\[(X\phi_s)(u_i)=\phi_s(u_iX_{\mf{k}})+\sum_{j=1}^l[(2s+1)\rho_P(Y_{j,\mf{a}})\phi_s(u_{i,j})+Y_{j,\mf{m}_0}\phi_s(u_{i,j})+\phi_s(Y_{j,\mf{k}}u_{i,j})].\]
Thus
\begin{multline*}
L((X\phi_s)(u_i))=\left(\sum_{j=1}^l 2\rho_P(Y_{j,\mf{a}})L(\phi_s(u_{i,j}))\right)s\\
+L(\phi_s(u_iX_{\mf{k}}))+\sum_{j=1}^l\rho_P(Y_{j,\mf{a}})L(\phi_s(u_{i,j}))+L(Y_{j,\mf{m}_0}\phi_s(u_{i,j}))+L(\phi_s(Y_{j,\mf{k}}u_{i,j})),
\end{multline*}
where the only term that depends on $s$ is the $s$ at the end of the first line. In other words, we may write the evaluation $L(\mbf{w}(s))$ as
\[L(\mbf{w}(s))=s\mbf{c}_1+\mbf{c}_0\]
for some vectors $\mbf{c}_0,\mbf{c}_1\in\C^n$. Substituting this into \eqref{eqnwsasb} gives
\[s\mbf{c}_1+\mbf{c}_0=\mbf{a}(s)L(B),\]
i.e.,
\[\mbf{a}(s)=s\mbf{c}_1L(B)^{-1}+\mbf{c}_0L(B)^{-1}.\]
Thus, in view of \eqref{eqnxphissum}, the lemma is proved.
\end{proof}

We remark that the only property of $\pi$ we used in the proof is admissibility.

\begin{proof}[Proof (of Theorem \ref{thmintertwining})]
We proceed in several steps.

\textit{Step 1.} We begin by defining some flat families of cochains for the induced representations $\Ind_P^G(\pi,s)$ as $s$ varies. Recall that $D_+$ has Harish-Chandra parameter $\Lambda$ described in Section \ref{subsecds}. Let
\[\lambda_0=\Lambda-\rho_c+\rho_n-\alpha_0=\Lambda-\rho+2\rho_n-\alpha_0\]
so that $\lambda_0\pm\alpha_0$ is the Blattner parameter for $D_\pm$, and so that by Corollary \ref{corlowestKtypeJ}, we have that $\lambda_0$ is the highest weight of the lowest $K$-type of $J$. 

For a dominant weight $\mu$ of $K$, write $V_\mu$ for the representation of $K$ of highest weight $\mu$. Let $E_{\Lambda-\rho}$ be the representation of $G(\C)$ of highest weight $\Lambda-\rho$. Then by Proposition \ref{propcomplexesconc}, we have four linearly independent cochains
\[c^{d\pm 1}:\sideset{}{^{d\pm 1}}\bigwedge(\mf{g}/\mf{k})\to\Ind_P^G(\pi,s_0)\otimes E_{\Lambda-\rho}^\vee,\quad\textrm{and}\quad c^{d,\pm}:\sideset{}{^d}\bigwedge(\mf{g}/\mf{k})\to\Ind_P^G(\pi,s_0)\otimes E_{\Lambda-\rho}^\vee\]
which are described explicitly in Remark \ref{remfourcochains}, and these cochains span the $(\mf{g},K)$-cohomology complex for $\Ind_P^G(\pi,s_0)\otimes E_{\Lambda-\rho}^\vee$. 

We can then define cochains
\[c_s^{d\pm 1}:\sideset{}{^{d\pm 1}}\bigwedge(\mf{g}/\mf{k})\to\Ind_P^G(\pi,s)\otimes E_{\Lambda-\rho}^\vee,\quad\textrm{and}\quad c_s^{d,\pm}:\sideset{}{^d}\bigwedge(\mf{g}/\mf{k})\to\Ind_P^G(\pi,s)\otimes E_{\Lambda-\rho}^\vee\]
for any $s$ as follows. For any $Y\in\bigwedge^{d-1}(\mf{g}/\mf{k})$, write
\[c^{d-1}(Y)=\sum_i v_i\otimes e_i,\]
with $v_i\in\Ind_P^G(\pi,s_0)$ lying, by Remark \ref{remfourcochains}, in the unique $K$-type in $\Ind_P^G(\pi,s)$ of highest weight $\lambda_0$, and with $e_i\in E_{\Lambda-\rho}^\vee$. Then we define
\[c_s^{d-1}(Y)=\sum_i v_{i,s}\otimes e_i,\]
where $v_{i,s}$ is the flat section which specializes to $v_i$ at $s=s_0$. It is easy to check that this cocycle is well defined for any $s$. We make a similar definition for the cochains $c_s^{d,\pm}$ and $c_s^{d+1}$, and it follows from Proposition \ref{propcomplexesconc} (c) that these cochains span the $(\mf{g},K)$-cohomology complex for $\Ind_P^G(\pi,s)\otimes E_{\Lambda-\rho}^\vee$.

\textit{Step 2.} We now express the coboundary $dc_s^{d-1}$ in terms of the cocycles $c_s^{\pm d}$.

For $s\ne \pm s_0$, the infinitesimal character of the representation $\Ind_P^G(\pi,s)$ does not match that of $E_{\Lambda-\rho}$, and therefore by \cite[Proposition II.3.1 (a)]{BW}, its cohomology vanishes if $s\ne \pm s_0$ and the representation $\Ind_P^G(\pi,s)$ is unitary. In particular,
\[H^i(\mf{g},K;\Ind_P^G(\pi,s)\otimes E_{\Lambda-\rho}^\vee)=0,\quad\textrm{for }s\textrm{ imaginary}.\]
Thus the $(\mf{g},K)$-cohomology complex for $\Ind_P^G(\pi,s)\otimes E_{\Lambda-\rho}^\vee$ is exact when $s$ is imaginary, and so $dc_s^{d-1}\ne 0$, say
\[dc_s^{d-1}=a^+(s)c_s^{d,+}+a^-(s)c_s^{d,-},\]
for some functions $a^+(s)$ and $a^-(s)$ of $s$, one of which is not identically zero. Note that for $X_0,\dotsc,X_{d-1}\in\mf{p}$, the section $dc_s^{d-1}(X_0\wedge\dotsb\wedge X_{d-1})$ takes the form
\begin{equation}
\label{eqndcsdminus1sum}
dc_s^{d-1}(X_0\wedge\dotsb\wedge X_{d-1})=\sum_{i=0}^{d-1} (-1)^{i}X_i\cdot c_s^{d-1}(X_0\wedge\dotsb\wedge\widehat{X}_i\wedge\dotsb\wedge X_{d-1});
\end{equation}
we remark that the terms in the boundary involving commutators vanish since $[\mf{p},\mf{p}]\subset\mf{k}$. It follows from Lemma \ref{lemsectionslinear} then that the functions $a^+(s)$ and $a^-(s)$ are linear in $s$. Also, by Proposition \ref{propcohofinduced}, we have
\[H^{d-1}(\mf{g},K;\Ind_P^G(\pi,-s_0)\otimes E_{\Lambda-\rho}^\vee)\ne 0,\]
and so $c_{-s_0}^{d-1}$ is a cocycle and $dc_{-s_0}^{d-1}=0$. Thus there are constants $a_0^+,a_0^-\in\C$, at least one of which is nonzero, such that
\begin{equation}
\label{eqndcsdminus1linear}
dc_s^{d-1}=(s+s_0)(a_0^+c_s^{d,+}+a_0^-c_s^{d,-}).
\end{equation}

\textit{Step 3.} Let us now assume that $a_0^+\ne 0$. We will show then that
\[\lim_{s\to s_0}\frac{1}{s-s_0}M(s,\phi_s^+)\notin J,\]
for some flat section $\phi_s^+\in \Ind_{P}^G(\pi,s)$ such that $\phi_{s_0}^+\in D_+$. To do this, fix $X_0,\dotsc,X_{d-1}\in\mf{p}$ be such that $c_s^{d,+}(X_0\wedge\dotsb\wedge X_{d-1})\ne 0$. Write
\begin{equation}
\label{eqncsdplussum}
c_s^{d,+}(X_0\wedge\dotsb\wedge X_{d-1})=\sum_j w_{j,s}^+\otimes e_j,
\end{equation}
where, by Remark \ref{remfourcochains}, the $w_{j,s}^+$ are flat sections in the unique $K$-type in $\Ind_P^G(\pi,s)$ of highest weight $\lambda_0+\alpha_0$, and $e_j\in E_{\Lambda-\rho}^\vee$, and at least one $w_{j,s}^+$ is nonzero. Since, by the same remark, the cocycle $c_s^{d,+}$ is only nontrivial on the $K$-type of highest weight $2\rho_n$ in $\bigwedge^{d}\mf{p}$ and $c_s^{d,-}$ is only nontrivial on the $K$-type of highest weight $2\rho_n-2\alpha_0$ in $\bigwedge^{d}\mf{p}$, we can and will assume $X_0,\dotsc,X_{d-1}$ are chosen so that
\[c_s^{d,-}(X_0\wedge\dotsb\wedge X_{d-1})=0.\]

For each $i=0,\dotsc,d-1$, let us write
\[c_s^{d-1}(X_0\wedge\dotsb\wedge\widehat{X}_i\wedge\dotsb\wedge X_{d-1})=\sum_{l}v_{i,l,s}\otimes e_{i,l}\]
where the $v_{i,l,s}$ are flat sections in the unique $K$-type in $\Ind_P^G(\pi,s)$ of highest weight $\lambda_0$, and $e_{i,l}\in E_{\Lambda-\rho}^\vee$. Then we combine this with \eqref{eqndcsdminus1sum}, \eqref{eqndcsdminus1linear}, and \eqref{eqncsdplussum} to find that
\begin{equation*}
\sum_{i=0}^{d-1}\sum_{l}(-1)^{i}[(X_i v_{i,l,s})\otimes e_{i,l}+v_{i,l,s}\otimes (X_i e_{i,l})]=a_0^+(s+s_0)\sum_j w_{j,s}^+\otimes e_j.
\end{equation*}
Let $f\in E_{\Lambda-\rho}$ be a vector such that $\sum_j e_j(f)M(s,w_{j,s}^+)$ is not identically zero as $s$ varies, which exists because $\sum_j M(s,w_{j,s}^+)\otimes e_j$ is not identically zero. Then we get
\begin{equation*}
\sum_{i=0}^{d-1}\sum_{l}(-1)^{i}[e_{i,l}(f)(X_i v_{i,l,s})+(X_i e_{i,l})(f)v_{i,l,s}]=a_0^+(s+s_0)\sum_j e_j(f)w_{j,s}^+,
\end{equation*}
with both sides not identically zero in $s$

We then apply the intertwining operator $M(s,\cdot)$ to get
\begin{equation*}
\sum_{i=0}^{d-1}\sum_{l}(-1)^{i}[e_{i,l}(f)(X_i M(s,v_{i,l,s}))+(X_i e_{i,l})(f)M(s,v_{i,l,s})]=a_0^+(s+s_0)\sum_j e_j(f)M(s,w_{j,s}^+),
\end{equation*}
with both sides not identically zero in $s$. Since the $K$-type of highest weight $\lambda_0$ in $\Ind_{P}^G(\pi,-s)$ or $\Ind_P^G(\pi,s)$ is unique, and in $\Ind_P^G(\pi,s_0)$ it is not contained in $D_+\oplus D_-$, there is a function $c_0(s)$, holomorphic for $\re(s)>0$, with $c_0(s_0)=1$ and such that
\[M(s,v_{i,l,s})=c_0(s)v_{i,l,-s}'\]
where $v_{i,l,s}'$ is the flat section in $\Ind_{P}^G(\pi,s)$ specializing to $M(s_0,v_{i,l,s})$ at $s=-s_0$. Thus we obtain
\begin{equation*}
c_0(s)\sum_{i=0}^{d-1}\sum_{l}(-1)^{i}[e_{i,l}(f)X_i v_{i,l,-s}'+(X_i e_{i,l})(f)v_{i,l,-s}']=a_0^+(s+s_0)\sum_j e_j(f)M(s,w_{j,s}^+).
\end{equation*}
The right hand side vanishes at $s=s_0$ because the intertwining operator annihilates discrete series when $\re(s)>0$. By Lemma \ref{lemsectionslinear}, the sum on the left hand side is linear in $s$ and thus may be written
\[\sum_{i=0}^{d-1}\sum_{l}(-1)^{i}[e_{i,l}(f)(X_i v_{i,l,-s}')+(X_i e_{i,l})(f)v_{i,l,-s}']=(s-s_0)\phi_{-s}^{(1)}\]
for some nonvanishing flat section $\phi_{s}^{(1)}\in\Ind_{P}^G(\pi,s)$. Whence,
\begin{equation*}
c_0(s)(s-s_0)\phi_{-s}^{(1)}=a_0^+(s+s_0)\sum_j e_j(f)M(s,w_{j,s}^+).
\end{equation*}
Thus since $c_0(s_0)=1$, $s_0\ne 0$, and $a_0^+\ne 0$, we have
\[0\ne\frac{1}{2s_0a_0^+}\phi_{-s_0}^{(1)}=\lim_{s\to s_0}\frac{1}{s-s_0}\sum_j e_j(f)M(s,w_{j,s}^+).\]
Thus, taking $\phi_{s}^+=\sum_j e_j(f)w_{j,s}^+$, we have $\phi_{s_0}^+\in D^+$ and
\[\lim_{s\to s_0}\frac{1}{s-s_0}M(s,\phi_s^+)\ne 0.\]
But $\phi_{s}^+$ is in the $K$-type of highest weight $\lambda_0+\alpha_0$, hence so is the limit above, and therefore it is not in $J$. An easy argument then shows that this defines a nonzero $(\mf{g},K)$-equivariant map $D_+\to\Ind_{P}^G(\pi,-s_0)/J$.

If instead $a_0^-\ne 0$, then a completely analogous proof shows that
\[\lim_{s\to s_0}\frac{1}{s-s_0}M(s,\phi_s^-)\notin J,\]
for some flat section $\phi_s^-\in \Ind_{P}^G(\pi,s)$ such that $\phi_{s_0}^-\in D_-$. So if both $a_0^+$ and $a_0^-$ are nonzero, we are done. Otherwise, one of them is zero and the other is not, as we will assume in the next step in order to finish the proof.

\textit{Step 4.} Say now that $a_0^-=0$ but $a_0^+\ne 0$. We will show now that, under this assumption, there is a flat section $\phi_s^-\in \Ind_{P}^G(\pi,s)$ with $\phi_{s_0}^-\in D_-$ such that
\[\lim_{s\to s_0}\frac{1}{s-s_0}M(s,\phi_s^-)\notin J.\]
This will complete the proof of the theorem when $a_0^-=0$, and the case where $a_0^-\ne 0$ but $a_0^+=0$ is completely analogous.

So assuming $a_0^-=0$ but $a_0^+\ne 0$, we find by \eqref{eqndcsdminus1linear} and the exactness of the $(\mf{g},K)$-cohomology complex for $\Ind_P^G(\pi,s)\otimes E_{\Lambda-\rho}^\vee$ when $s$ is purely imaginary that $dc_s^{d,-}$ is not identically zero in $s$. Since
\[H^{d+1}(\mf{g},K;\Ind_P^G(\pi,s_0)\otimes E_{\Lambda-\rho}^\vee)\ne 0\]
by Proposition \ref{propcomplexesconc} (a), every cochain in $C^{d}(\mf{g},K;\Ind_P^G(\pi,s_0)\otimes E_{\Lambda-\rho}^\vee)$ is a cocycle, and thus $dc_s^{d,-}$ vanishes at $s=s_0$. An argument like the one that justified \eqref{eqndcsdminus1linear} then shows that
\[dc_s^{d,-}=b_0(s-s_0)c_s^{d+1},\]
for some nonzero $b_0\in\C$. Thus we may find elements $X_0,\dotsc,X_d\in\mf{p}$ such that
\[dc_s^{d,-}(X_0\wedge\dotsb\wedge X_d)\ne 0.\]
Like in Step 3, pairing with an appropriate element of $E_{\Lambda-\rho}$ gives us an equality
\[\sum_{i=0}^d (-1)^i X_i w_{i,s}^-=(s-s_0)v_s',\]
where both sides are not identically zero in $s$, and $v_s'\in\Ind_P^G(\pi,s)$ is a flat section in the $K$-type of highest weight $\lambda_0$, and $w_{i,s}^-\in\Ind_P^G(\pi,s)$ are flat sections such that $w_{i,s_0}^-\in D^-$. Applying the intertwining operator and dividing by $(s-s_0)$ then gives
\[\frac{1}{s-s_0}\sum_{i=0}^d (-1)^i X_i M(s,w_{i,s}^-)=M(s,v_s').\]
Since $v_s'$ is in the $K$-type of highest weight $\lambda_0$, and such a $K$-type is present in $J$ and not $D_+$ or $D_-$, we have that $M(s_0,v_s')\ne 0$, whence
\[\lim_{s\to s_0}\frac{1}{s-s_0}\sum_{i=0}^d (-1)^i X_i M(s,w_{i,s}^-)\ne 0.\]
It then follows that
\[\lim_{s\to s_0}\frac{1}{s-s_0}M(s,w_{i,s}^-)\ne 0\]
for some $i$, and an examination of $K$-types shows that this element is not in $J$, completing the proof.
\end{proof}

\subsection{An alternate setup}
\label{subsecaltsetup}
Our main theorem below, Theorem \ref{thmmainthm}, will be stated with hypotheses at the archimedean place that are rather different than the setup we introduced in Sections \ref{subsecsetup} and \ref{subsecds} on which the rest of Section \ref{secreal} has been based. So in this subsection, we prove that these two setups are equivalent in the context of connected groups.

So let $G$ be a real, connected, linear, semisimple Lie group, and let $K$ be a maximal compact subgroup of $G$, given to us by a fixed Cartan involution on $G$ which we call $\theta$. Let $T$ be a $\theta$-stable compact Cartan subgroup of $G$, which is thus contained in $K$. Write $\mf{g}$, $\mf{k}$, and $\mf{t}$ for the respective Lie algebras, and write $\Delta=\Delta(\mf{g},\mf{t})$. Normalize root vectors $X_\alpha$ for $\alpha\in\Delta$ with respect to the Killing form $B$ on $\mf{g}$ by \eqref{eqnnormofrootvecs}.

So far the setup is the same as at the beginning of Section \ref{subsecsetup}, except that we are not fixing an ordering on $\Delta$, nor are we fixing a noncompact simple root. Instead, assume we are given a maximal parabolic subgroup $P'$ of $G$ with Langlands decomposition $P'=M_0'A'N'$. Let $M'=M_0'A'$, and let $\mf{m}'$, $\mf{m}_0'$, $\mf{a}'$, and $\mf{n}'$ be the complexified Lie algebras associated with $M'$, $M_0'$, $A'$, and $N'$, respectively. We first prove the following.

\begin{lemma}
\label{lemrootalpha}
With the setup as above, let $\mf{h}'$ be any Cartan subalgebra of $\mf{g}$ containing $\mf{a}'$. Then there is a root $\alpha$ of $\mf{h}$ which vanishes on the orthogonal complement of $\mf{a}'$ in $\mf{h}'$, and it is unique up to sign.
\end{lemma}

\begin{proof}
Let $\mf{t}$ be the complexified Lie algebra of $T$, and $\mf{t}'=\mf{t}\cap\mf{m}_0'$, where $\mf{m}_0'$ is the complexified Lie algebra of $M_0'$. Then $\mf{h}'$ is conjugate to $\mf{a}'\oplus\mf{t}'$ by an element of the complexification $G_\C$ of $G$. Therefore it suffices to prove the lemma for $\mf{a}'\oplus\mf{t}'$ in place of $\mf{h}$.

Now by \cite[Proposition 6.59]{knappbook}, the subalgebra $\mf{a}'\oplus\mf{t}'$ can be conjugated by an element $g_1$ of $G$ to a $\theta$-stable Cartan subalgebra of $\mf{g}$, which we call $\mf{h}''$; actually we will prove the lemma for $\Ad(g_1)(\mf{a}')$ in $\mf{h}''$. Now a root of $\mf{h}''$ is real if and only if it vanishes on $\Ad(g_1)(\mf{t}')$. Since $\mf{a}'\oplus\mf{t}'$, and hence $\mf{h}''$, are not maximally compact, by \cite[Proposition 6.70]{knappbook} there is a real root of $\mf{h}''$. This root is unique up to sign since the root system of $\mf{g}$ is reduced. This proves the lemma.
\end{proof}

Let us now fix a discrete series representation $\pi'$ of $M_0'$. Then we have the following.

\begin{proposition}
\label{propaltsetup}
Notation as above, let $\widetilde{\alpha}$ be the restriction to $\mf{a}'$ of a root $\alpha$ of any Cartan subalgebra $\mf{h}'$ containing $\mf{a}'$ and vanishing on the orthogonal complement of $\mf{a}'$ in $\mf{h}'$; assume moreover that $\alpha$ is chosen with the sign so that $\widetilde{\alpha}$ is in $\mf{n}'$. Such a root $\alpha$ exists by Lemma \ref{lemrootalpha}.

Write $s_0'$ for the positive real number such that $\frac{1}{2}\widetilde{\alpha}=2s_0'\rho_{P'}$, where $\rho_{P'}$ is the half sum of roots of $A'$ in $N'$. Let us assume that the induced representation $\Ind_{P'}^G(\pi',s_0')$ has infinitesimal character matching that of a finite dimensional representation $E$ of the complexification $G_\C$ of $G$. Finally, assume that the unique irreducible quotient, call it $J'$, of $\Ind_{P'}^G(\pi',s_0')$, is unitary.

First, if $J'\otimes E^\vee$ has nonvanishing $(\mf{g},K)$-cohomology, then there are:
\begin{itemize}
\item A system of positive roots $\Delta^{+}$ in $\Delta$;
\item A simple noncompact root $\alpha_{0}$ in $\Delta^{+}$;
\item A Harish-Chandra parameter $\Lambda$ for $\mf{t}$ in $\mf{g}$ which is dominant for $\Delta^{+}$ and such that
\[\frac{2\langle\Lambda,\alpha_{0}\rangle}{\langle\alpha_{0},\alpha_{0}\rangle}=1;\]
\end{itemize}
such that, letting $P=M_0AN$ be as defined as in Section \ref{subsecsetup} using these $\Delta^{+}$ and $\alpha_0$, and letting $\pi$ and $s_0$ be as in defined in Section \ref{subsecds} using $\Lambda$ in addition, then the following holds: There is a $g_0\in K$ such that:
\begin{itemize}
\item We have $M=g_0M'g_0^{-1}$ and $P=g_0P'g_0^{-1}$, and
\item We have $\pi\cong\pi'\circ\Ad(g_0^{-1})$ as representations of $M$, and
\item We have $s_0=s_0'$.
\end{itemize}

Second, if on the other hand $H^*(\mf{g},K;J'\otimes E^\vee)=0$, then actually
\[H^*(\mf{g},K;\Ind_{P'}^G(\pi',s_0')\otimes E^\vee)=0,\]
as well.
\end{proposition}

\begin{proof}
Let $T_0'$ be a compact Cartan subgroup of $M_0$ contained in $K$ (which exists since $M_0$ has discrete series by assumption), and $\mf{t}_0'$ its Lie algebra. Then $\mf{a}'\oplus\mf{t}_0'$ is a Cartan subalgebra of $\mf{g}$.

We will establish this proposition by proving a sequence of claims.

\textit{Claim 1}. Let $s_1'$ be any complex number with $0<\re(s_1')<s_0'$, and let $\sigma$ be any irreducible admissible representation of $M_0'$. Then the infinitesimal character of (any constituent of) $\Ind_{P'}^G(\sigma,s_1')$ is different from that of $\Ind_{P'}^G(\pi',s_0')$.

\textit{Proof of claim}. We view the infinitesimal characters in question as Weyl orbits of complex weights of $\mf{a}'\oplus\mf{t}_0'$. Let $\chi_0$ and $\chi_1$ be the infinitesimal characters of $\Ind_{P'}^G(\pi',s_0')$ and $\Ind_{P'}^G(\sigma,s_1')$, respectively. Let $\lambda_0$ be any representative of $\chi_0$. Then $\lambda_0$ is an integral weight because $\chi_0$ is the infinitesimal character of a finite dimensional algebraic representation of $G_\C$. In particular, we have that
\[2\frac{\langle\lambda_0|_{\mf{a}'},\widetilde{\alpha}\rangle}{\langle\widetilde{\alpha},\widetilde{\alpha}\rangle}=2\frac{\langle\lambda_0,\alpha\rangle}{\langle\alpha,\alpha\rangle}\in\Z,\]
where the equality holds because $\widetilde{\alpha}$ vanishes on $\mf{t}_0'$ and $\mf{t}_0'$ is orthogonal to $\mf{a}'$.

On the other hand, let $\lambda_1$ be a representative of $\chi_1$ such that $\lambda_1|_{\mf{t}_0'}$ is a representative of the infinitesimal character of $\sigma$, and such that $\lambda_1|_{\mf{a}'}=\pm\frac{1}{2}(s_1/s_0')\widetilde{\alpha}$. Then
\[2\re\left(\frac{\langle\lambda_1|_{\mf{a}'},\widetilde{\alpha}\rangle}{\langle\widetilde{\alpha},\widetilde{\alpha}\rangle}\right)=\pm\frac{\re(s_1)}{s_0'}\notin\Z.\]
Therefore $\lambda_1$ does not lie in the Weyl orbit $\chi_0$, proving the claim.

\textit{Claim 2}. Let $U$ be any irreducible constituent of $\Ind_{P'}^G(\pi',s_0')$ different from $J'$. Then $U$ is tempered.

\textit{Proof of claim}. If $U$ were nontempered, then by a property of the Langlands classification, namely \cite[Proposition IV.4.13]{BW}, it is a constituent of some parabolic induction of the form $\Ind_{P'}^G(\sigma,s_1')$ with $\re(s_1')<s_0'$, contradicting the previous claim.

\textit{Claim 3}. If $H^*(\mf{g},K;J'\otimes E^\vee)\ne 0$, then $H^q(\mf{g},K;J'\otimes E^\vee)$ is nonvanishing for some $q\ne d$; if instead $H^*(\mf{g},K;J'\otimes E^\vee)=0$, then
\[H^*(\mf{g},K;\Ind_{P'}^G(\pi',s_0')\otimes E^\vee)=0,\]
and so the second statement of the proposition is true.

\textit{Proof of claim}. By \cite[Proposition III.5.3]{BW} and Claim 2 above, any irreducible constituent $U$ of $\Ind_{P'}^G(\pi',s_0')$ different from $J'$ has that $H^q(\mf{g},K;U\otimes E^\vee)=0$ when $q\ne d$. If the same were true for $J'$ in place of such $U$, then
\begin{equation}
\label{eqngkcohofind}
H^q(\mf{g},K;\Ind_{P'}^G(\pi',s_0')\otimes E^\vee)
\end{equation}
would be concentrated in degree $q=d$. But by \cite[Theorem III.3.3]{BW}, the cohomology group above is concentrated and nonvanishing in two consecutive degrees if it is nonvanishing at all. Thus it must vanish under the assumption that $H^*(\mf{g},K;J'\otimes E^\vee)=0$ (as well as the cohomology of any such $U$ as above, though this is not needed to conclude), proving the second statement of the proposition. If instead $H^*(\mf{g},K;J'\otimes E^\vee)\ne 0$, then indeed \eqref{eqngkcohofind} is nonvanishing and concentrated in degree $q=d$, giving a contradiction. This proves the claim.

So from now on we assume that $H^*(\mf{g},K;J'\otimes E^\vee)\ne 0$, and we aim to prove the first statement of the proposition.

\textit{Claim 4}. Under this assumption, we have
\[\dim_{\C}H^q(\mf{g},K;J'\otimes E^\vee)=\begin{cases}
1 & \textrm{if }q=d-1,d+1;\\
0 & \textrm{otherwise.}
\end{cases}\]

\textit{Proof of claim}. Let
\[\mc{K}'=\ker(\Ind_{P'}^G(\pi',s_0')\to J').\]
Then as in the proof of Claim 3 above, we have that
\[H^q(\mf{g},K;\mc{K}'\otimes E^\vee)\]
is concentrated in degree $d$, and
\[H^q(\mf{g},K;\Ind_{P'}^G(\pi',s_0')\otimes E^\vee)\]
is concentrated in two consecutive degrees and is $1$-dimensional in each. Analysis of the long exact $(\mf{g},K)$-cohomology sequence associated with the short exact sequence
\[0\to \mc{K}'\otimes E^\vee\to \Ind_{P'}^G(\pi',s_0')\otimes E^\vee\to J'\otimes E^\vee\to 0\]
then leads to the desired conclusion.

For the next claims, we must use the theory of cohomological induction which was employed in Section \ref{subsecstandardmods}.

\textit{Claim 5}. There is a system of positive roots $\Delta^+$ in $\Delta$, and a simple noncompact root $\alpha_0$, and a set $S_c$ of simple compact roots in $\Delta^+$ which are orthogonal to $\alpha_0$, such that the following holds: Let $\tilde{\mf{q}}$ be the $\theta$-stable parabolic subalgebra which is standard with respect to $\Delta^+$ and whose standard Levi $\tilde{\mf{l}}$ contains exactly the simple roots $\alpha_0$ and the ones in $S_c$. Let $\tilde{L}=N_G(\tilde{\mf{q}})$ the corresponding Levi subgroup in $G$. Write $\tilde{\mf{q}}=\tilde{\mf{l}}\oplus\tilde{\mf{u}}$, where $\tilde{\mf{u}}$ is the radical of $\tilde{\mf{q}}$. Then there is a unitary character $\tilde{\lambda}$ of $L$ whose differential $d\tilde{\lambda}$ satisfies
\[\langle\beta,d\tilde{\lambda}|_{\mf{t}}\rangle\geq 0,\qquad\textrm{for every root }\beta\textrm{ of }\mf{t}\textrm{ in }\tilde{\mf{u}},\]
such that $J'\cong A_{\tilde{\mf{q}}}(\tilde{\lambda})$, in the notation of Section \ref{subsecstandardmods}.

\textit{Proof of claim}. Because $J'$ is assumed to be unitary and cohomological with respect to $E^\vee$, appealing to \cite[Theorem 5.6]{VZ} gives us objects $\Delta^+$, $\tilde{\mf{q}}=\tilde{\mf{l}}\oplus\tilde{\mf{u}}$, and $\tilde{\lambda}$ satisfying the conclusions of the claim, except without any information about the roots which are contained in $\tilde{\mf{l}}$. So we must show that there is a simple noncompact root $\alpha_0$ in $\tilde{\mf{l}}$, and that the other simple roots contained in $\tilde{\mf{l}}$ are compact and orthogonal to $\alpha_0$.

Let $R$ be the number of noncompact roots in $\tilde{\mf{u}}$. Then by \cite[Theorem 5.5]{VZ}, we have
\[H^q(\mf{g},K,A_{\tilde{\mf{q}}}(\tilde{\lambda})\otimes E^\vee)\cong\hom_{\tilde{\mf{l}}\cap\mf{k}}(\sideset{}{^{q-R}}\bigwedge(\tilde{\mf{l}}/(\tilde{\mf{l}}\cap\mf{k})),\C),\]
where we have written $\mf{k}$ for the complexified Lie algebra of $K$, and the functor $\hom_{\tilde{\mf{l}}\cap\mf{k}}$ means $\C$-linear maps that are equivariant with respect to the adjoint action by $\tilde{\mf{l}}\cap\mf{k}$. Thus by Claim 4 above, we have
\[\dim_{\C}\hom_{\tilde{\mf{l}}\cap\mf{k}}(\sideset{}{^{q-R}}\bigwedge(\tilde{\mf{l}}/(\tilde{\mf{l}}\cap\mf{k})),\C)=\begin{cases}
1 & \textrm{if }q=d-1,d+1;\\
0 & \textrm{otherwise.}
\end{cases}\]
There are always the obvious isomorphisms
\[\sideset{}{^0}\bigwedge(\tilde{\mf{l}}/(\tilde{\mf{l}}\cap\mf{k}))\cong\C\]
and
\[\sideset{}{^{\mathrm{top}}}\bigwedge(\tilde{\mf{l}}/(\tilde{\mf{l}}\cap\mf{k}))\cong\C;\]
the former must therefore contribute to degree $q=d-1$, and the latter to degree $d+1$, from which it follows that $R=d-1$ and
\[\dim_{\C}(\tilde{\mf{l}}/(\tilde{\mf{l}}\cap\mf{k}))=2.\]
This proves that there is a unique positive noncompact root $\alpha_0$ in $\tilde{\mf{l}}$, which must therefore be simple. It is orthogonal to any other simple root in $\tilde{\mf{l}}$; indeed, if $\beta$ is any such root, then $\beta$ is compact, and were it not orthogonal to $\alpha_0$, then $\alpha_0+\beta$ would be a root, necessarily in $\tilde{\mf{l}}$. But then $\alpha_0+\beta$ must also be noncompact, contradicting the uniqueness of $\alpha_0$. This finishes the proof of the claim.

\textit{Claim 6}. With $\Delta^+$ and $\alpha_0$ as in Claim 5, let $\mf{q}$ be the $\Delta^+$-standard $\theta$-stable parabolic subgroup whose Levi is $\mf{l}=\mf{t}\oplus\C X_{\alpha_0}\oplus\C X_{-\alpha_0}$, and let $L=N_G(\mf{q})$ its corresponding Levi subgroup. Also, let $\lambda=\tilde{\lambda}|_{L}$. Then $A_{\tilde{\mf{q}}}(\tilde{\lambda})\cong A_{\mf{q}}(\lambda)$.

\textit{Proof of claim}. As in the proof of Proposition \ref{propisosbetweends}, this follows immediately from the spectral sequence of \cite[Theorem 11.77]{knvo}, using that the roots from $\tilde{\mf{l}}$ omitted in $\mf{l}$ are all compact by Claim 5.

\textit{Claim 7}. Let $\Delta^+$ and $\alpha_0$ be as in Claim 5 above. Write $\rho$ for the half sum of positive roots in $\Delta^+$, and let $\Lambda=(d\lambda|_{\mf{t}})+\rho$ with $\lambda$ as in Claim 6, which is by definition a Harish-Chandra parameter for $\mf{t}$ in $\mf{g}$. Then the proposition is true for this choice $\Delta^+$, $\alpha_0$, and $\Lambda$.

\textit{Proof of claim}. By Theorem \ref{thmstdmodforind}, the Langlands data for $A_{\mf{q}}(\lambda)$, hence for $J'$ by Claims 5 and 6 above, are given by the parabolic $P=M_0AN$ from Section \ref{subsecsetup}, the discrete series representation $\pi$ from Section \ref{subsecds}, and the character $\delta_P^{s_0}$ of $A$ with $s_0$ as in \eqref{eqnsnaught}. By the Langlands classification, using that $P$ and $P'$ are maximal (so that there is no ambiguity about the parabolic subgroups being used to induced tempered representations) we must then have, up to conjugation by an element of $G$, that $P$ and $P'$ coincide, that $\pi$ and $\pi'$ coincide, and that $s_0=s_0'$. This concludes the proof.
\end{proof}

\section{Maximal parabolic Eisenstein series}
\label{secadele}

\subsection{Preliminaries on Eisenstein series}
\label{subseceisenstein}

We now depart a bit from the notation of Section \ref{secreal}. Let $\mc{G}$ be a connected, semisimple algebraic group over $\Q$. We fix a factorizable maximal compact subgroup $K\subset \mc{G}(\A)$ and write $K_f$ and $K_\infty$ for, respectively, the finite adelic and archimedean factors of $K$.

Fix a maximal parabolic $\Q$-subgroup $\mc P$ of $\mc G$ with Levi decomposition $\mc P=\mc{M}\mc{N}$. Let $\pi$ be a cuspidal automorphic representation of $\mc{M}(\A)$, and let $\chi_{\pi}$ be its central character. We assume that $\chi_\pi|_{A_{\mc{M}}(\R)^\circ}$ is trivial, where $A_{\mc{M}}$ is the maximal $\Q$-split torus in the center of $\mc{M}$. This can always be arranged by twisting.

For $s$ varying in $\C$, we can consider the induced representation
\[\Ind_{\mc{P}(\A)}^{\mc{G}(\A)}(\pi,s)\]
where we recall (see the section on notation in the introduction) that the induction is unitarily normalized, and implicitly we are taking smooth $K$-finite vectors.

Let $\phi_s\in\Ind_{\mc{P}(\A)}^{\mc{G}(\A)}(\pi,s)$ be a \textit{flat section}, i.e., it is a family of elements of $\Ind_{\mc{P}(\A)}^{\mc{G}(\A)}(\pi,s)$, one for each $s\in\C$, such that, as functions on $K$, we have $\phi_s|_{K}$ is independent of $s$. Any given section $\phi_{s_1}\in\Ind_{\mc{P}(\A)}^{\mc{G}(\A)}(\pi,s_1)$, for a fixed $s_1\in\C$, can be put into a unique flat section.

With this setup we can now form an Eisenstein series in the usual way. In fact, we can view $\pi$ as a subspace of
\[L^2(\mc{M}(\Q)A_{\mc{M}}(\R)^\circ\backslash \mc{M}(\A)),\]
and so it makes sense to evaluate smooth vectors in $\pi$ at the identity $1_{\mc{M}}$ of $\mc{M}(\A)$. Thus we can define
\[E(\phi,s,g)=\sum_{\gamma\in \mc{P}(\Q)\backslash \mc{G}(\Q)}\phi_s(\gamma g)(1_{\mc{M}}).\]
The above expression is convergent and holomorphic for $\re(s)$ sufficiently large, and defined for general $s$ by meromorphic continuation; that this can be done is due to Langlands \cite{LanglandsES}. The precise meaning of meromorphicity here is as follows: Given a bounded region $U$ in $\C$, there are points $s_1,\dotsc,s_n\in\C$ (not necessarily distinct) such that, for each $g\in \mc{G}(\A)$, we have that
\[(s-s_1)\dotsb(s-s_n)E(\phi,s,g)\]
is a holomorphic function of $s\in\C$.

Since $\mc P$ is maximal, it is known that:
\begin{itemize}
\item The Eisenstein series $E(\phi,s,g)$ has at most a finite number of poles for $\re(s)\geq 0$;
\item All the poles of $E(\phi,s,g)$ when $\re(s)\geq 0$ lie on the segment $(0,1/2]$;
\item All the poles of $E(\phi,s,g)$ when $\re(s)\geq 0$ are simple.
\end{itemize}
See \cite[Proposition in \S IV.1.11]{MW} and \cite[Remark in \S IV.3.12]{MW} for the statements, and \cite[\S IV.3.12]{MW} for their proofs.

Now given any $s_0\in\C$ and nonnegative integer $N$, as well as $U$ and $s_1,\dotsc,s_n$ as above, we have that the function of $g\in \mc{G}(\A)$ defined by
\[\frac{d^N}{ds^N}(s-s_1)\dotsb(s-s_n)E(\phi,s,g)|_{s=s_0},\]
is an automorphic form on $\mc{G}(\A)$. For $\re(s_0)>0$, we denote by $\mc{A}_{\pi,s_0}$ the space of all such automorphic forms obtained in this way as both $\phi_s\in\Ind_{\mc{P}(\A)}^{\mc{G}(\A)}(\pi,s)$ and $N\in\Z_{\geq 0}$ vary; that is, the space $\mc{A}_{\pi,s_0}$ is the space of all Eisenstein series, their derivatives and residues, obtained from $\pi$ at $s=s_0$. By results of Franke and Schwermer \cite[Theorem 1.4]{FS}, the space $\mc{A}_{\pi,s_0}$ is a direct summand of the space of all automorphic forms for $\mc{G}$.

Now for the moment let us assume $s_0>0$ is such that $E(\phi,s,g)$ has a pole at $s=s_0$ for some $\phi_s\in\Ind_{\mc{P}(\A)}^{\mc{G}(\A)}(\pi,s)$. Then by above, this pole is simple, and we can consider the subspace
\[\mc{L}(\pi,s_0)\subset\mc{A}_{\pi,s_0}\]
defined as the space generated by all possible residues at $s=s_0$ of such Eisenstein series $E(\phi,s,\cdot)$ for $\phi_s\in\Ind_{\mc{P}(\A)}^{\mc{G}(\A)}(\pi,s)$. Then $\mc{L}(\pi,s_0)$ is in fact isomorphic to the unique irreducible quotient (i.e., Langlands quotient, whence the notation) of $\Ind_{\mc{P}(\A)}^{\mc{G}(\A)}(\pi,s_0)$.

We would like to describe the quotient $\mc{A}_{\pi,s_0}/\mc{L}(\pi,s_0)$ in this case. We can do so as an induced representation, after we take into account the presence of derivatives. To this end, let us introduce a structure of a $\mc{P}(\A_f)\times(\mf{p},K_\infty\cap \mc{P}(\R))$-module on the full symmetric algebra $\Sym(\C)$ in a way that depends on $s_0$; here, we denote by $\mf{p}$ the complexified Lie algebra of $\mc P$, and
\[\Sym(\C)=\bigoplus_{n=0}^\infty \Sym^n(\C),\]
with the induced algebra structure. The reader may refer to \cite{FS} where this $\mc{P}(\A_f)\times(\mf{p},K_\infty\cap \mc{P}(\R))$-module structure is defined in general, and also \cite[\S 2.1]{mung2cohart} for an exposition; in the former of these two sources the relevant object is denoted $S(\check{\mf{a}}_{\mc{P}})$, and in the latter it is denoted $\Sym(\mf{a}_{\mc{P},0}^\vee)$. Below, in this case of maximal parabolics, we will be implicitly identifying this space $\mf{a}_{\mc{P},0}^\vee=\Lie(A_{\mc{M}})_\C^\vee$ with $\C$ via the affine map given by $(2s+1)\rho_{\mc{P}}\mapsto s$.

To define a $\mc{P}(\A_f)\times(\mf{p},K_\infty\cap \mc{P}(\R))$-module structure on $\Sym^n(\C)$, let $s_0\in\C$ be arbitrary but fixed for the moment. We view $\Sym(\C)$ as a space of holomorphic differential operators on $\C$ supported at $s_0$. Thus $\Sym(\C)$ acts on functions $f(s)$ holomorphic near $s=s_0$ as follows: If $s_1,\dotsc,s_n\in\C$, so that $S=s_1\otimes\dotsb\otimes s_n\in\Sym^n(\C)$, then
\[S(f(s))=s_1\dotsb s_n\frac{d^n}{ds^n}f(s)|_{s=s_0}.\]
Then we define a $\mc{P}(\A_f)\times(\mf{p},K_\infty\cap \mc{P}(\R))$-module structure on $\Sym^n(\C)$ just below, and we will denote by $\Sym^n(\C)_{s_0}$ the resulting $\mc{P}(\A_f)\times(\mf{p},K_\infty\cap \mc{P}(\R))$-module to emphasize the dependence on $s_0$. First, for $p\in \mc{P}(\A_f)$ and $D\in\Sym^n(\C)_{s_0}$, let
\[(pD)(f(s))=D(\delta_{\mc{P}}^{s+1/2}(p)f(s)),\]
for $f(s)$ holomorphic near $s=s_0$. Next, let $\mc{M}_0$ be the joint kernel of all characters of $\mc M$, so that $\mc{P}=\mc{M}_0A_{\mc{M}}\mc{N}$; let $\mf{p}=\mf{m}_0\oplus\mf{a}_{\mc{M}}\oplus\mf{n}$ be the corresponding decomposition of the Lie algebra. Then for $X\in\mf{a}_{\mc{M}}$ and $D\in\Sym^n(\C)_{s_0}$ and $f(s)$ holomorphic near $s=s_0$, let
\[(XD)(f(s))=D((2s+1)\rho_{\mc{P}}(X)f(s)).\]
Finally, we let $\mf{m}_0\oplus\mf{n}$ and $K_\infty\cap \mc{P}(\R)$ act trivially on $\Sym^n(\C)_{s_0}$. This defines the promised $\mc{P}(\A_f)\times(\mf{p},K_\infty\cap \mc{P}(\R))$-module structure on $\Sym^n(\C)_{s_0}$.

We then have the following theorem of Grbac \cite[Theorem 3.1]{grbac}.

\begin{theorem}[Grbac]
\label{thmgrbac}
When $s_0>0$ and some Eisenstein series coming from $\pi$ has a pole at $s=s_0$, we have an isomorphism of $\mc{G}(\A_f)\times(\mf{g},K_\infty)$-modules
\[\mc{A}_{\pi,s_0}/\mc{L}(\pi,s_0)\cong\Ind_{\mc{P}(\A)}^{\mc{G}(\A)}(\pi\otimes\Sym(\C)_{s_0}),\]
where there is no implicit normalization in the parabolic induction, and $\mc{L}(\pi,s_0)$ is the residual representation described above.
\end{theorem}

In fact, the isomorphism of Grbac's theorem above is given by Franke's \textit{mean value map} \cite[p. 235]{franke}, or more accurately, its restriction to $\mc{A}_{\pi,s_0}$. More precisely, let
\[\mbf{MW}:\Ind_{\mc{P}(\A)}^{\mc{G}(\A)}(\pi\otimes\Sym(\C)_{s_0})\to\mc{A}_{\pi,s_0}\]
be defined as follows. First, we make an identification of vector spaces,
\begin{equation}
\label{eqnmapPsi}
\Psi:\Ind_{\mc{P}(\A)}^{\mc{G}(\A)}(\pi)\otimes\Sym(\C)\cong\Ind_{\mc{P}(\A)}^{\mc{G}(\A)}(\pi\otimes\Sym(\C)_{s_0}),
\end{equation}
given by $\Psi(\phi\otimes D)=\phi_D$ where $\phi_D(g)=\phi(g)\otimes D$. (It is possible to upgrade this to an isomorphism of $\mc{G}(\A_f)\times(\mf{g},K_\infty)$-modules by giving an explicit structure of $\mc{G}(\A_f)\times(\mf{g},K_\infty)$-module on the right hand side as in \cite[p. 155]{LS}. However we do not need to explain how to do this here.) Then with this identification, we write a Laurent series about $s=s_0$,
\[DE(\phi,s,g)=\sum_{n=-N}^\infty F_n(g)(s-s_0)^n\]
where $F_n$ are automorphic forms on $\mc{G}(\A)$, and define $\mbf{MW}$ by
\[\mbf{MW}(\Psi(\phi\otimes D))=F_0.\]

The map $\mbf{MW}$ is not a homomorphism of $\mc{G}(\A_f)\times(\mf{g},K_\infty)$-modules in the presence of a pole at $s=s_0$ for the Eisenstein series defined by $\pi$. However, it is well defined modulo the residual representation $\mc{L}(\pi,s_0)$, and Grbac's theorem says exactly that it becomes an isomorphism then.

We note the following consequence of the definitions for later use.

\begin{proposition}
\label{propinclofinds}
Fix any $s_0\in\C$. The degree $0$ subspace $\C=\Sym^{0}(\C)\subset\Sym(\C)_{s_0}$ is stable under the action of $\mc{P}(\A_f)\times(\mf{p},K_\infty\cap \mc{P}(\R))$ and is given by the character $\delta_{\mc{P}}^{s_0+1/2}$. Hence this gives rise to a natural inclusion
\[\Ind_{\mc{P}(\A)}^{\mc{G}(\A)}(\pi,s_0)\hookrightarrow\Ind_{\mc{P}(\A)}^{\mc{G}(\A)}(\pi\otimes\Sym(\C)_{s_0}).\]
\end{proposition}

\begin{proof}
Immediate from the definitions.
\end{proof}

We now turn to the constant terms of Eisenstein series. Let $\mc{P}'$ be another parabolic $\Q$-subgroup of $\mc G$, say with Levi decomposition $\mc{P}'=\mc{M}'\mc{N}'$. Recall that the \textit{constant term} $F_{\mc{N}'}(g)$ of an automorphic form $F$ along $\mc{P}'$ is defined by the integral
\[F_{N'}(g)=\int_{\mc{N}'(\Q)\backslash \mc{N}'(\A)}F(n'g)\,dn',\]
where, say, the Haar measure gives $\mc{N}'(\Q)\backslash \mc{N}'(\A)$ volume $1$. Since $\mc{N}'(\Q)\backslash \mc{N}'(\A)$ is compact, this is always defined and convergent, and it gives a function $\mc{G}(\A)\to\C$. We denote the constant term of an Eisenstein series $E(\phi,s,g)$ along $\mc{P}'$ by $E_{\mc{N}'}(\phi,s,g)$.

It is a fact that Eisenstein series, their residues, and their derivatives are orthogonal to cusp forms. It follows from this that, given any two automorphic forms $F$ and $F'$ which are linear combinations of Eisenstein series and their residues and derivatives, if the constant terms of $F$ and $F'$ along all parabolic $\Q$-subgroups $\mc{P}'$ of $\mc{G}$ agree, then in fact $F=F'$.

Now the constant terms $E_{\mc{N}'}(\phi,s,g)$ of an Eisenstein series $E(\phi,s,g)$ are also meromorphic in $s$, and in fact $E(\phi,s,g)$ has a pole at $s=s_0$ if and only if one of its constant terms does. Thus to locate poles of Eisenstein series, it becomes important to understand their constant terms explicitly. The \textit{Langlands--Shahidi method} allows us to do this using intertwining operators and $L$-functions. In the main body of this paper we will only need to go as far as to describe the connection with intertwining operators, and we do this just below; the connection with $L$-functions will not be necessary for us except to discuss a few examples in Section \ref{subsecexamples}, and so we omit that part of the discussion here.

Similarly to the archimedean case described above in Section \ref{subsecintertwining}, we can define a global intertwining operator, and to do this, it is convenient to fix a maximal split torus in $\mc{G}$, which we call $\mc{T}_0$, and a system of positive roots for $\mc{T}_0$ in $\mc{G}$, with respect to which $\mc{P}$ is standard. Given a rational element $w$ of the Weyl group of $\mc{T}_0$ in $\mc G$, and assuming that there is another standard parabolic $\Q$-subgroup $\mc{P}'=\mc{M}'\mc{N}'$ of $\mc{G}$ such that $\mc{M}'=w^{-1}\mc{M}w$ and such that $w^{-1}\mc{P}w$ is opposite to $\mc{P}'$, we can define an \textit{intertwining operator}
\[M_w(s,\cdot):\Ind_{\mc{P}(\A)}^{\mc{G}(\A)}(\pi,s)\to\Ind_{\mc{P}'(\A)}^{\mc{G}(\A)}(\pi^w,-s),\]
where $\pi^w=\pi\circ\Ad(w)$; this operator is given by
\[M(s,\phi)=\int_{\mc{N}(\A)}\phi_s(wng)\,dn\]
in regions where it is convergent, defining a holomorphic function in such, and then by meromorphic continuation otherwise. We then have the following relationship between constant terms of Eisenstein series and intertwining operators.

\begin{theorem}[Langlands--Shahidi]
\label{thmLS}
Given the ordering on the roots of $\mc{T}_0$ in $\mc{G}$ described above, let $\mc{P}'$ be a standard parabolic $\Q$-subgroup of $\mc G$ with Levi decomposition $\mc{P}'=\mc{M}'\mc{N}'$. Then if there is no element $w$ in the Weyl group of $\mc{T}_0$ in $\mc G$ such that $w^{-1}\mc{M}w=\mc{M}'$, then
\[E_{\mc{N}'}(\phi,s,g)=0.\]
Otherwise, if there is such a $w$ with $w^{-1}\mc{M}w=\mc{M}'$, then we have
\[E_{\mc{N}'}(\phi,s,g)=\begin{cases}
\phi_s(g)(1_{\mc{M}})+M_w(s,\phi)(g)(1_{\mc{M}})&\textrm{if }\mc{M}'=\mc{M}\textrm{ and }\mc{P}'\textrm{ is opposite to }w^{-1}\mc{P}w;\\
M_w(s,\phi)(g)(1_{\mc{M}})&\textrm{if }\mc{M}'\ne\mc{M}\textrm{ but }\mc{P}'\textrm{ is opposite to }w^{-1}\mc{P}w;\\
\phi_s(g)(1_{\mc{M}})&\textrm{otherwise.}\\
\end{cases}\]
Moreover, the last case above occurs if and only if $\mc{P}=\mc{P}'$ and $w=1$ is the unique Weyl element such that $w^{-1}\mc{M}w=\mc{M}'$.
\end{theorem}

\begin{proof}
See \cite[Theorem 6.2.1]{shahidiES}.
\end{proof}

\subsection{The main theorem}

Recall that, given an irreducible finite dimensional representation $E$ of $\mc{G}(\C)$, in the introduction we defined the \textit{automorphic cohomology} of $\mc{G}$ to be
\[H^*(\mc{G},E^\vee)=H^*(\mf{g},K_\infty^\circ;\mc{A}_E(\mc{G})\otimes E^\vee),\]
where $\mc{A}_E(\mc{G})$ is the space of automorphic forms on $\mc{G}(\A)$ which are annihilated by a power of the kernel $J_E$ of the action of the center of $\mc{U}(\mf{g})$ on $E$. Our main theorem below describes a particular part of this automorphic cohomology coming from residual Eisenstein series. We state it in full.

\begin{theorem}
\label{thmmainthm}
Let $\mc{G}$ be a connected, semisimple algebraic group over $\Q$ such that $\mc{G}(\R)$ has discrete series, let $\mc{P}=\mc{M}\mc{N}$ a maximal parabolic $\Q$-subgroup with given Levi decomposition, and let $A_{\mc{M}}$ be the maximal $\Q$-split torus in the center of $\mc{M}$. Assume that $\mc{P}(\R)$ is still a maximal parabolic subgroup of $\mc{G}(\R)$. Let $\pi$ be a unitary cuspidal automorphic representation of $\mc M$ with central character $\chi$. Assume that the archimedean component $\chi_\infty$ of $\chi$ is trivial on $A_{\mc{M}}(\R)^\circ$, and that the archimedean component $\pi_\infty$ of $\pi$ is discrete series.

Let $\widetilde{\alpha}$ be the character of $A_{\mc{M}}(\R)^\circ$ whose differential is the restriction of a root $\alpha$ of some Cartan subalgebra containing the complexified Lie algebra of $A_{\mc{M}}(\R)^\circ$ such that $\alpha$ vanishes on its orthogonal complement in that Cartan subalgebra; assume moreover that $\alpha$ is chosen with the sign so that $\widetilde{\alpha}$ is a root of $A_{\mc{M}}(\R)^\circ$ in $\mc{N}(\R)$. Such a root $\alpha$ exists by Lemma \ref{lemrootalpha}. Let $s_0$ be the positive real number such that $\frac{1}{2}\widetilde{\alpha}=2s_0\rho_{\mc{P}}|_{A_{\mc{M}}(\R)^\circ}$, where $\rho_{\mc{P}}$ is the half sum of roots of $A_{\mc{M}}$ in $\mc{N}$. Assume
\[\Ind_{\mc{P}(\R)}^{\mc{G}(\R)}(\pi_{\infty},s_0)\]
has infinitesimal character matching that of an irreducible, finite dimensional representation $E$ of $\mc{G}(\C)$.

Now assume the following:
\begin{enumerate}[label=(\roman*)]
\item There is a flat section
\[\phi_s\in\Ind_{\mc{P}(\A)}^{\mc{G}(\A)}(\pi,s)\]
such that the Eisenstein series $E(\phi,s,g)$ has a pole at $s=s_0$.
\item The corresponding residual representation $\mc{L}(\pi,s_0)$ is cohomological.
\end{enumerate}
Let us write $d=\frac{1}{2}\dim(\mc{G}(\R)/K_\infty^\circ)$, where $K_\infty$ is a maximal compact subgroup of $\mc{G}(\R)$, and $\mf{g}$ for the complexified Lie algebra of $\mc{G}$. Then:
\begin{enumerate}[label=(\alph*)]
\item We have that there is an integer $m>0$ such that
\[\dim_{\C}H^q(\mf{g},K_\infty^\circ;\mc{L}(\pi,s_0)_\infty\otimes E^\vee)=\begin{cases}
m&\textrm{if }q=d-1\textrm{ or }d+1;\\
0&\textrm{otherwise}.
\end{cases}\]
\item The natural map 
\[H^{d-1}(\mf{g},K_\infty^\circ;\mc{L}(\pi,s_0)\otimes E^\vee)\to H^{d-1}(G,E^\vee),\]
induced by including $\mc{L}(\pi,s_0)$ into the space $\mc{A}_E(\mc{G})$, is injective.
\item On the other hand, the natural map 
\[H^{d+1}(\mf{g},K_\infty^\circ;\mc{L}(\pi,s_0)\otimes E^\vee)\to H^{d+1}(G,E^\vee),\]
again induced by including $\mc{L}(\pi,s_0)$ into the space $\mc{A}_E(\mc{G})$, is zero.
\end{enumerate}
Thus the residual Eisenstein cohomology coming from $\pi$ is concentrated in degree $d-1$.
\end{theorem}

The rest of this subsection will be devoted to the proof, and we give two proofs of part (c). We will thus split the rest of this subsection into three parts: First, we reduce to the case of connected $\mc{G}(\R)$ and prove parts (a) and (b). Afterwards, we give a short but inexplicit proof of (c) which is in the spirit of \cite{RSson1}. Once this is complete, we offer an alternative proof which is longer, but which explicitly computes the primitive of the class in the image of the map in part (c). We believe that for future work which branches off from the results in this paper, it will be necessary to build on this latter method in order to make significant progress, which is why we have decided to include it.

\subsubsection*{Reduction to the connected case; proofs of (a) and (b)}

We wish to apply the results of Section \ref{secreal}, but we immediately face the rather annoying problem that $\mc{G}(\R)$ may not be connected as a real Lie group. To resolve this issue, we prove a few lemmas that will allow us to reduce to the connected case.

\begin{lemma}
\label{lemrestrofindarch}
Let $w$ be an element in the Weyl group of $\mc{G}$ preserving $\mc{M}$. Denote by $\pi_\infty^w$ the representation on the same space as $\pi_\infty$ but with action given by $\pi_\infty\circ\Ad(w)$. Then for any $s\in\C$, the map
\[\Ind_{\mc{P}(\R)}^{\mc{G}(\R)}(\pi_\infty^w,s)\to\Ind_{\mc{P}(\R)\cap \mc{G}(\R)^\circ}^{\mc{G}(\R)^\circ}(\pi_\infty^w,s)\]
given by
\[\phi\mapsto \phi|_{\mc{G}(\R)^\circ}\]
is an isomorphism of $(\mf{g},K_\infty^\circ)$-modules.
\end{lemma}

\begin{proof}
For simplicity we write the proof only in the case when $w=1$; the proof in any other case is exactly the same.

The map in question is easily checked to be well defined and a map of $(\mf{g},K_\infty^\circ)$-modules. We must check that it is bijective.

Now since $\mc{P}$ is a maximal parabolic in $\mc{G}$ and $\mc{P}(\R)$ is still a maximal parabolic in $\mc{G}(\R)$, there is a maximally $\Q$-split torus $\mc{T}_s$ in $\mc G$ contained in $\mc M$ such that $\mc{T}_s(\R)$ is a maximally split Cartan subgroup of $\mc G$ contained in $\mc{M}(\R)$. By \cite[Lemma 0.4.2 (a)]{voganbook}, the group $\mc{T}_s(\R)$, and hence also $\mc{P}(\R)$, meets every component of $\mc{G}(\R)$. Thus the double coset space
\[\mc{P}(\R)\backslash \mc{G}(\R)/\mc{G}(\R)^\circ\cong (\mc{P}(\R)\cap K_\infty)\backslash K_\infty/K_\infty^\circ\]
is trivial. It follows that any $\phi\in\Ind_{\mc{P}(\R)}^{\mc{G}(\R)}(\pi_\infty,s)$ is completely determined by its restriction to $\mc{G}(\R)^\circ$. Therefore the map in question is injective.

It is also surjective; given $\phi'\in\Ind_{\mc{P}(\R)\cap \mc{G}(\R)^\circ}^{\mc{G}(\R)^\circ}(\pi_\infty,s)$, we can define $\phi\in\Ind_{\mc{P}(\R)}^{\mc{G}(\R)}(\pi_\infty,s)$ as follows. By above, for any $g\in \mc{G}(\R)$ we may write $g=p g_0$ with $p\in \mc{P}(\R)$ and $g_0\in \mc{G}(\R)^\circ$. Then set
\[\phi(g)=p\phi'(g_0).\]
It is easy to check that this is well defined and that $\phi|_{\mc{G}(\R)^\circ}=\phi'$, and the surjectivity follows.
\end{proof}

From now on, let $\mc{M}_0$ be the joint kernel of all characters of $\mc M$, so that $\mc{M}=\mc{M}_0 A_{\mc{M}}$. Let $\mf{m}_0$ and $\mf{a}_{\mc{M}}$ be the corresponding complexified Lie algebras. By \cite[Lemma II.5.6]{BW}, each module $\pi_\infty$ is the direct sum of finitely many irreducible $(\mf{m}_0,K_\infty^\circ\cap \mc{P}(\R))$-modules, say $\pi_{\infty,1},\dotsc,\pi_{\infty,n}$ for some positive integer $n$, all necessary in the discrete series. Correspondingly, for each element $w$ in the Weyl group of $\mc{G}$ preserving $\mc{M}$, we have a natural decomposition
\[\pi_\infty^w\cong\bigoplus_{i=1}^n\pi_{\infty,i}^w.\]

Thus we have decompositions
\[\Ind_{\mc{P}(\R)}^{\mc{G}(\R)}(\pi_\infty^w,s)\to\bigoplus_{i=1}^n\Ind_{\mc{P}(\R)\cap \mc{G}(\R)^\circ}^{\mc{G}(\R)^\circ}(\pi_{\infty,i}^w,s).\]
These decompositions preserve flat sections as $s$ varies. Indeed, for each $i$ with $1\leq i\leq n$, let
\[\pr_i^w:\pi_\infty^w\to\pi_{\infty,i}^w\]
be the corresponding projection. Then $\pr_i\otimes 1$ is the natural projection
\[\pi_\infty^w\otimes\delta_{\mc{P}(\R)}^{s+1/2}\to\pi_{\infty,i}^w\otimes\delta_{\mc{P}(\R)}^{s+1/2}.\]
So if $\phi_s\in\Ind_{\mc{P}(\R)}^{\mc{G}(\R)}(\pi_\infty^w,s)$ is a flat section, then the corresponding section of the right hand side of the above isomorphism is
\[((\pr_1\otimes 1)\circ(\phi_s|_{\mc{G}(\R)^\circ}),\dotsc,(\pr_n\otimes 1)\circ(\phi_s|_{\mc{G}(\R)^\circ})),\]
and the values of each entry on $K_\infty^\circ$ of each entry of this tuple is easily seen to be independent of $s$ if those of $\phi_s$ on $K_\infty$ are.

Furthermore, if $\phi_s\in\Ind_{\mc{P}(\R)}^{\mc{G}(\R)}(\pi_\infty,s)$ is flat, if $\re(s)$ is sufficiently large, and if $w\mc{P}w^{-1}$ is opposite to $\mc{P}$, then we have that
\[\int_{\mc{N}(\R)}\phi_s(wn(\cdot))\,dn\]
is sent, under the map above, to the tuple
\[\left((\pr_1\otimes 1)\circ\left(\int_{\mc{N}(\R)}\phi_s(wn(\cdot))\,dn\right),\dotsc,(\pr_n\otimes 1)\circ\left(\int_{\mc{N}(\R)}\phi_s(wn(\cdot))\,dn\right)\right).\]
This discussion then justifies the following lemma by analytic continuation.

\begin{lemma}
\label{leminterdecomp}
For any $w$ in the Weyl group of $\mc{G}$ preserving $\mc M$ and sending $\mc P$ to its opposite, the decomposition
\[\pi_\infty\cong\bigoplus_{i=1}^n\pi_{\infty,i}\]
above induces the following commutative diagram:
\[\xymatrix{
\Ind_{\mc{P}(\R)}^{\mc{G}(\R)}(\pi_\infty,s) \ar[r]^-{\sim} \ar[d]^{M(s,\cdot)} & \bigoplus_{i=1}^n\Ind_{\mc{P}(\R)\cap \mc{G}(\R)^\circ}^{\mc{G}(\R)^\circ}(\pi_{\infty,i},s)\ar[d]^{\oplus_i M(s,\cdot)}\\
\Ind_{\mc{P}(\R)}^{\mc{G}(\R)}(\pi_\infty^{w},-s) \ar[r]^-{\sim} & \bigoplus_{i=1}^n\Ind_{\mc{P}(\R)\cap \mc{G}(\R)^\circ}^{\mc{G}(\R)^\circ}(\pi_{\infty,i}^w,-s),}\]
where the intertwining operators $M(s,\cdot)$ are defined with respect to $w$.
\end{lemma}

Now the representation $\mc{L}(\pi,s_0)_\infty$ is unitary, because it is the archimedean component of the residual representation $\mc{L}(\pi,s_0)$, and residual Eisenstein series appear in the discrete spectrum $L_{\disc}^2(\mc{G}(\Q)\backslash \mc{G}(\A))$. This representation is assumed to be cohomological by condition (ii) of Theorem \ref{thmmainthm}, and so it follows that there is an index $i$ with $1\leq i\leq n$ as above such that $\Ind_{\mc{P}(\R)}^{\mc{G}(\R)}(\pi_{\infty,i},s)$ has a unitary quotient $Q_i$ with nontrivial $(\mf{g},K_\infty^\circ)$-cohomology. Let $\mc{I}_{\mathrm{coh}}$ be the nonempty set of indices $i$ with $1\leq i\leq n$ such that $\Ind_{\mc{P}(\R)}^{\mc{G}(\R)}(\pi_{\infty,i},s)$ admits a unitary quotient $Q_i$ with nontrivial $(\mf{g},K_\infty^\circ)$-cohomology. Then for any $i\in \mc{I}_{\mathrm{coh}}$, since $Q_i$ is unitary, it splits as a direct sum, and thus it follows that the unique irreducible quotient of $\Ind_{\mc{P}(\R)}^{\mc{G}(\R)}(\pi_{\infty,i},s)$ is $Q_i$ and hence is unitary and cohomological. By Proposition \ref{propaltsetup}, we are thus in the situation of Sections \ref{subsecsetup} and \ref{subsecds}, and so we have the following.

Let us fix a (necessarily connected) compact Cartan subgroup $T$ of $\mc{G}(\R)^\circ$ contained in $K_\infty^\circ$ with complexified Lie algebra $\mf{t}\subset\mf{g}$. Write $\Delta=\Delta(\mf{g},\mf{t})$ for the set of roots of $\mf{t}$ in $\mf{g}$. Let $\theta$ be a Cartan involution on $\mc{G}(\R)^\circ$ giving $K_\infty^\circ$. Fix root vectors $X_\alpha$ for each $\alpha\in\Delta$, normalized with respect to the Killing form $B$ on $\mf{g}$ so that \eqref{eqnnormofrootvecs} holds. Then by Proposition \ref{propaltsetup}, there are:
\begin{itemize}
\item A system of positive roots $\Delta^{+,i}$ in $\Delta$;
\item A simple noncompact root $\alpha_{0,i}$ in $\Delta^{+,i}$;
\item A Harish-Chandra parameter $\Lambda_i$ for $\mf{t}$ in $\mf{g}$ which is dominant for $\Delta^{+,i}$ and such that
\[\frac{2\langle\Lambda_i,\alpha_{0,i}\rangle}{\langle\alpha_{0,i},\alpha_{0,i}\rangle}=1;\]
\end{itemize}
such that, if we set the following notation:
\begin{itemize}
\item Let $\mf{a}=\C(X_{\alpha_{0,i}}+X_{-\alpha_{0,i}})$;
\item Let $A=\exp(\R(X_{\alpha_{0,i}}+X_{-\alpha_{0,i}}))$;
\item Let $\mf{t}'$ be the orthogonal complement of $[X_{\alpha_{0,i}},X_{-\alpha_{0,i}}]$ in $\mf{t}$;
\item Let $\mf{h}$ be the Cartan subalgebra of $\mf{g}$ given by $\mf{h}=\mf{a}\oplus\mf{t}'$;
\item Let $\mf{g}_\R=\Lie(\mc{G}(\R))$ be the real Lie algebra of $\mc{G}(\R)$;
\item Let $M_0^\circ$ be the connected subgroup of $\mc{G}(\R)^\circ$ with real Lie algebra given by the centralizer in the real Lie algebra of $\mc{G}$ of $(\mf{a})$;
\item Let $M_0=C_{K_\infty^\circ}(\mf{a})M_0^\circ$;
\item Let $C_{\alpha_{0,i}}=\Ad(\exp(\tfrac{\pi}{4}(X_{\alpha_{0,i}}-X_{-\alpha_{0,i}})))$ be the Cayley transform associated with $\alpha_{0,i}$;
\item Let $\tp{C}_{\alpha_{0,i}}:\Delta\overset{\sim}{\longrightarrow}\Delta(\mf{g},\mf{h})$ be the transpose of the Cayley transform;
\item Let $\rho_{c,i}$ and $\rho_{n,i}$ be, respectively, the half sums of compact and noncompact roots in $\Delta^{+,i}$;
\item Let $\xi_i$ be the character of $T$ whose differential is $\Lambda_i-\rho_{c,i}+\rho_{n,i}$;
\item Let $\pi_{0,i}$ be the discrete series representation of $M_0^\circ$ with Harish-Chandra parameter $\Lambda_i|_{\mf{t}'}$;
\end{itemize}
then there is a $g_{0,i}\in \mc{G}(\R)^\circ$ such that the following holds:
\begin{itemize}
\item The parabolic subgroup $g_{0,i}(\mc{P}(\R)\cap \mc{G}(\R)^\circ)g_{0,i}^{-1}$ of $\mc{G}(\R)^\circ$ is the one with Levi $M_0A=g_{0,i}(\mc{M}(\R)\cap \mc{G}(\R)^\circ)g_{0,i}^{-1}$ and unipotent radical $g_{0,i}\mc{N}(\R)g_{0,i}^{-1}$ whose $A$-roots are those which are positive on $X_{\alpha_{0,i}}+X_{-\alpha_{0,i}}$;
\item We have $(\pi_{\infty,i}\circ\Ad(g_{0,i}^{-1}))|_{M_0}\cong\Ind_{M_0^\circ Z_{M_0}}^{M_0}(\pi_{0,i}\otimes(\xi_i|_{Z_{M_0}}))$;
\item We have that $s_0=(2\rho_{g_{0,i}(\mc{P}(\R)\cap \mc{G}(\R)^\circ)g_{0,i}^{-1}}(X_{\alpha_{0,i}}+X_{-\alpha_{0,i}}))^{-1}$.

\end{itemize}
Note that the property $M_0A=g_{0,i}(\mc{M}(\R)\cap \mc{G}(\R)^\circ)g_{0,i}^{-1}$ above determines $A$ as $A=g_{0,i} A_{\mc{M}}(\R)^\circ g_{0,i}^{-1}$, and so determines the root $\alpha_{0,i}$ up to sign. One checks easily then that the notations above that have no reference to the index $i$ indeed carry no dependence on $i$.

Note also that $\Ind_{M_0^\circ Z_{M_0}}^{M_0}(\pi_{0,i}\otimes(\xi_i|_{Z_{M_0}}))$ is the representation playing the role of the discrete series representation which was called $\pi$ in Section \ref{secreal}. The results in Section \ref{secreal} stated for $\Ind_P^G(\pi,\pm s_0)$ thus apply in our case to the induced representation
\[\Ind_{\mc{P}(\R)\cap \mc{G}(\R)^\circ}^{\mc{G}(\R)^\circ}(\pi_{\infty,i},s_0),\qquad \textrm{for }i\in\mc{I}_{\mr{coh}},\]
because it is isomorphic to
\begin{multline*}
\Ind_{g_{0,i}(\mc{P}(\R)\cap \mc{G}(\R)^\circ)g_{0,i}^{-1}}^{\mc{G}(\R)^\circ}(\pi_{\infty,i}\circ\Ad(g_{0,i}^{-1}),s_0)\\
\cong\Ind_{g_{0,i}(\mc{P}(\R)\cap \mc{G}(\R)^\circ)g_{0,i}^{-1}}^{\mc{G}(\R)^\circ}(\Ind_{M_0^\circ Z_{M_0}}^{M_0}(\pi_{0,i}\otimes(\xi_i|_{Z_{M_0}})),s_0).
\end{multline*}

The main point is that we can now apply the results of Sections \ref{subsecstandardmods}, \ref{subsecgkcoh}, and \ref{subsecintertwining} in our situation. But also, we note that since $\alpha_{0,i}$ is determined up to sign, we can define an element $\widetilde{w}_0$ of the Weyl group of $\mf{h}$ in $\mf{g}$ as the simple reflection about $\tp{C}_{\alpha_{0,i}}(\alpha_{0,i})$, and this is independent of $i\in\mc{I}_{\mathrm{coh}}$. Then the element $w_0=g_{0,i}\widetilde{w}_0 g_{0,i}^{-1}$ preserves $\mc{M}$ and sends $\mc{P}$ to its opposite. Therefore the intertwining operator $M(s_0,\cdot)$ on $\Ind_{\mc{P}(\R)\cap \mc{G}(\R)^\circ}^{\mc{G}(\R)^\circ}(\pi_{\infty,i},s_0)$ defined with respect to $w_0$ has image isomorphic to the unique irreducible quotient, which we denote by $\mc{L}(\pi_{i,\infty},s_0)$, of $\Ind_{\mc{P}(\R)\cap \mc{G}(\R)^\circ}^{\mc{G}(\R)^\circ}(\pi_{\infty,i},s_0)$ for any $i$ with $1\leq i\leq n$. Thus by Lemma \ref{leminterdecomp}, we have a decomposition
\[\mc{L}(\pi,s_0)_\infty\cong\bigoplus_{i=1}^n\mc{L}(\pi_{\infty,i},s_0),\]
and so every summand on the right hand side above is already unitary. Thus $\mc{I}_{\mathrm{coh}}$ is the set of indices $i$ such that $\mc{L}(\pi_{\infty,i},s_0)$ is cohomological (with the condition on unitarity being automatic \textit{a posteriori}).

We now would like to compute the $(\mf{g},K_\infty^\circ)$-cohomology of the induced representation
\[\Ind_{\mc{P}(\A)}^{\mc{G}(\A)}(\pi\otimes\Sym(\C)_{s_0}),\]
where, for $s\in\C$, the algebra $\Sym(\C)_s$ is the symmetric algebra of Section \ref{subseceisenstein} with the $\mc{P}(\A_f)\times(\mf{p},K_{\infty}\cap \mc{P}(\R))$-module structure described above Theorem \ref{thmgrbac}. Here again we are writing $\mf{p}$ for the complexified Lie algebra of $\mc P$. The aforementioned cohomology is computed by the following.

\begin{lemma}
\label{lemcohofgloinduced}
Let $m=\#\mc{I}_{\mathrm{coh}}$. We have an isomorphism of $\mc{G}(\A_f)$-modules,
\[H^q(\mf{g},K_\infty^\circ;\Ind_{\mc{P}(\A)}^{\mc{G}(\A)}(\pi\otimes\Sym(\C)_{s_0})\otimes E^\vee)\cong\begin{cases}
\Ind_{\mc{P}(\A_f)}^{\mc{G}(\A_f)}(\pi_f,s_0)^{\oplus m}&\textrm{if }q=d;\\
0&\textrm{otherwise,}
\end{cases}\]
where $\pi_f$ is the finite adelic component of $\pi$.
\end{lemma}

\begin{proof}
Let $i$ in $\mc{I}_{\mathrm{coh}}$. Note that $\Ad(g_{0,i}^{-1})(\mf{h})$ is a $\theta$-stable Cartan subalgebra of $\mf{g}$ contained in the complexified Lie algebra $\mf{m}$ of $\mc M$; we recall that $\mf{h}$ is the Cartan subalgebra defined in the bulleted list of notation above. Let us write $\mf{h}_{0,i}=\Ad(g_{0,i}^{-1})(\mf{h})$, and let $\Delta_{0,i}^+$ be the set of positive roots in $\Delta(\mf{g},\mf{h}_{0,i})$ obtained from $\tp{C}_{\alpha_{0,i}}(\Delta^{+,i})$ by $\Ad(g_{0,i}^{-1})$. Write $\Lambda_{0,i}=(\tp{C}_{\alpha_{0,i}}(\Lambda_{i}))\circ\Ad(g_{0,i})$, which is a weight of $\mf{h}_{0,i}$. In fact, if we write $\rho_{0,i}$ for the half sum of roots in $\Delta_{0,i}^+$, then $\Lambda_{0,i}-\rho_{0,i}$ is the highest weight of $\mf{h}_{0,i}$ in $E$. Let $W(\mf{g},\mf{h}_{0,i})$ and $W(\mf{m},\mf{h}_{0,i})$ be the Weyl groups of $\mf{h}$ in $\mf{g}$ and $\mf{m}$, respectively. Write $W^P$ for the set of minimal length representatives of the quotient $W(\mf{m}_0,\mf{h})\backslash W(\mf{g},\mf{h}_{0,i})$. Then since
\[H^*(\mf{g},K_\infty^\circ;\Ind_{\mc{P}(\R)\cap \mc{G}(\R)^\circ}^{\mc{G}(\R)^\circ}(\pi_{i,\infty},s_0)\otimes E^\vee)\]
is nonzero by Proposition \ref{propcohofinduced}, we have that \cite[Theorem III.3.3]{BW} tells us there is an element $w_i\in W^P$ such that the following holds: We have
\[(-w_i\Lambda_{0,i})|_{\mf{a}_\mc{M}}=\frac{1}{2}\widetilde{\alpha}=2s_0\rho_{\mc{P}},\]
and the infinitesimal character of $\pi_{\infty,i}$ is the orbit of $(-w_i\Lambda_{0,i})|_{\mf{h}\cap\mf{m}_0}$, and moreover there is an isomorphism
\begin{multline*}
H^q(\mf{g},K_\infty^\circ;\Ind_{\mc{P}(\R)\cap \mc{G}(\R)^\circ}^{\mc{G}(\R)^\circ}(\pi_{i,\infty},s_0)\otimes E^\vee)\\
\cong\bigoplus_{j=0}^1\left(\sideset{}{^j}\bigwedge\mf{a}_{\mc{M}}\right)\otimes H^{q-\ell(w_i)-j}(\Ad(g_{0,i}^{-1})\mf{m}_0,K_{\infty}^\circ\cap \mc{P}(\R);\pi_{\infty,i}\otimes F_{w_i\Lambda_{0,i}-\rho_{0,i}}^\vee),
\end{multline*}
where $F_{w_i\Lambda_{0,i}-\rho_{0,i}}$ is the representation of highest weight $(w_i\Lambda_{0,i}-\rho_{0,i})|_{\mf{h}_{0,i}\cap\mf{m}_0}$.

On the other hand, by Proposition \ref{propcohofinduced}, we have
\[\dim_{\C}H^q(\mf{g},K_\infty^\circ;\Ind_{\mc{P}(\R)\cap \mc{G}(\R)^\circ}^{\mc{G}(\R)^\circ}(\pi_{i,\infty},s_0)\otimes E^\vee)=\begin{cases}
1 & \textrm{if }i=d,d+1;\\
0 & \textrm{otherwise.}
\end{cases}\]
Thus, writing $d'=\frac{1}{2}\dim (g_{0,i}^{-1}M_0g_{0,i})/(\mc{P}(\R)\cap K_\infty^\circ)$ for the middle dimension for $M_0$, we have $\ell(w_i)=d-d'$ for any $i\in\mc{I}_{\mathrm{coh}}$, as the discrete series representation $\pi_{\infty,i}$ can only have cohomology in middle degree $d'$. If $i$ with $1\leq i\leq n$ is not in $\mc{I}_{\mathrm{coh}}$, then by (the second statement of) Proposition \ref{propaltsetup},
\[H^*(\mf{g},K_\infty^\circ;\Ind_{\mc{P}(\R)\cap \mc{G}(\R)^\circ}^{\mc{G}(\R)^\circ}(\pi_{i,\infty},s_0)\otimes E^\vee)=0.\]
Hence by the same token, it follows from \cite[Theorem 2.2.3]{mung2cohart} that
\begin{multline*}
H^q(\mf{g},K_\infty^\circ;\Ind_{\mc{P}(\A)}^{\mc{G}(\A)}(\pi\otimes\Sym(\C)_{s_0})\otimes E^\vee)\\
\cong\bigoplus_{i\in\mc{I}_{\mathrm{coh}}}\Ind_{\mc{P}(\A_f)}^{\mc{G}(\A_f)}(\pi_f,s_0)\otimes H^{q-d+d'}(\Ad(g_{0,i}^{-1})\mf{m}_0,K_{\infty}^\circ\cap \mc{P}(\R);\pi\otimes F_{w\Lambda_{0,i}-\rho_{0,i}}).
\end{multline*}
By above, this reduces exactly to the statement of the lemma.
\end{proof}

We will need the following lemma, which is similar to Lemma \ref{lemrestrofindarch}.

\begin{lemma}
\label{lemrestrofindglobal}
For any $s\in\C$, the map
\[\Ind_{\mc{P}(\A)}^{\mc{G}(\A)}(\pi\otimes\Sym(\C)_{s})\to\Ind_{\mc{P}(\A_f)\times(\mc{P}(\R)\cap \mc{G}(\R)^\circ)}^{\mc{G}(\A_f)\times \mc{G}(\R)^\circ}(\pi\otimes\Sym(\C)_{s})\]
given by
\[\phi\mapsto \phi|_{\mc{G}(\A_f)\times \mc{G}(\R)^\circ}\]
is an isomorphism of $\mc{G}(\A_f)\times(\mf{g},K_\infty^\circ)$-modules; here, of course, there is no implicit normalization in the induced representations above.
\end{lemma}

\begin{proof}
The proof is exactly analogous to that of Lemma \ref{lemrestrofindarch} and thus omitted.
\end{proof}

We can then decompose
\[\Ind_{\mc{P}(\A_f)\times(\mc{P}(\R)\cap \mc{G}(\R)^\circ)}^{\mc{G}(\A_f)\times \mc{G}(\R)^\circ}(\pi\otimes\Sym(\C)_{s})=\bigoplus_{i=1}^n\Ind_{\mc{P}(\A_f)\times(\mc{P}(\R)\cap \mc{G}(\R)^\circ)}^{\mc{G}(\A_f)\times \mc{G}(\R)^\circ}((\pi_f\otimes\pi_{\infty,i})\otimes\Sym(\C)_{s}),\]
and similarly
\[\mc{L}(\pi,s_0)\cong\bigoplus_{i=1}^n\mc{L}(\pi,s_0)_f\otimes\mc{L}(\pi_{\infty,i},s_0).\]
Taking $(\mf{g},K_\infty^\circ)$-cohomology of the modules in the latter of these two formulas tensored with $E^\vee$, and appealing to Proposition \ref{propcomplexesconc} (b), gives (a) of Theorem \ref{thmmainthm} with $m=\#\mc{I}_{\mathrm{coh}}$.

Now let us consider the short exact sequence of $\mc{G}(\A_f)\times(\mf{g},K_\infty)$-modules,
\begin{equation}
\label{eqngrbacses}
0\to\mc{L}(\pi,s_0)\to\mc{A}_{\pi,s_0}\to\Ind_{\mc{P}(\A)}^{\mc{G}(\A)}(\pi\otimes\Sym(\C)_{s_0})\to 0,
\end{equation}
coming from Grbac's Theorem \ref{thmgrbac} above. Note that $\mc{A}_{\pi,s_0}$ is a direct summand of $\mc{A}_E(\mc{G})$ by consideration of infinitesimal characters. By Lemma \ref{lemcohofgloinduced}, the $(\mf{g},K_\infty^\circ)$-cohomology of the quotient module tensored with $E^\vee$ is concentrated in degree $d$, and by (a) of Theorem \ref{thmmainthm}, the $(\mf{g},K_\infty^\circ)$-cohomology of the submodule tensored with $E^\vee$ is concentrated in degrees $d-1$ and $d+1$. We thus get that the $(\mf{g},K_\infty^\circ)$-cohomology of $\mc{A}_{\pi,s_0}\otimes E^\vee$ is concentrated in degrees $d-1$, $d$, and $d+1$. In fact, the long exact sequence associated with \eqref{eqngrbacses} tensored with $E^\vee$ gives an isomorphism
\begin{equation}
\label{eqnisoindegdminusone}
H^{d-1}(\mf{g},K_\infty^\circ;\mc{L}(\pi,s_0)\otimes E^\vee)\cong H^{d-1}(\mf{g},K_\infty^\circ;\mc{A}_{\pi,s_0}\otimes E^\vee),
\end{equation}
and an exact sequence
\begin{multline}
\label{eqnseqwithbdry}
0\to H^{d}(\mf{g},K_\infty^\circ;\mc{A}_{\pi,s_0}\otimes E^\vee)\to H^{d}(\mf{g},K_\infty^\circ;\Ind_{\mc{P}(\A)}^{\mc{G}(\A)}(\pi\otimes\Sym(\C)_{s_0})\otimes E^\vee)\\
\overset{\partial}{\longrightarrow} H^{d+1}(\mf{g},K_\infty^\circ;\mc{L}(\pi,s_0)\otimes E^\vee)\to H^{d+1}(\mf{g},K_\infty^\circ;\mc{A}_{\pi,s_0}\otimes E^\vee)\to 0,
\end{multline}
where $\partial$ is a boundary map.

Now the isomorphism \eqref{eqnisoindegdminusone} implies (b) of Theorem \ref{thmmainthm}. Thus we must now prove (c). As mentioned, we present two proofs.

\subsubsection*{First proof of (c)}

This first proof is in the style of that of \cite[Theorem 1.4.4]{RSson1}. We first note that there is a map of $\mc{G}(\A_f)\times(\mf{g},K_\infty)$-modules
\[\mc{E}:\Ind_{\mc{P}(\A)}^{\mc{G}(\A)}(\pi\otimes\Sym(\C)_{s_0})\to\mc{A}_{\pi,s_0},\]
given as follows. Let $\Psi$ be the map from \eqref{eqnmapPsi}. Choose a small disc around $s_0\in\C$, calling the boundary $C$, and assume this disc is small enough so that the only possible pole of an Eisenstein series coming from $\pi$ in this disc is at $s=s_0$. Then we set
\[\mc{E}\left(\Psi\left(\phi\otimes \frac{d^k}{ds^k}\right)\right)=\frac{k!}{2\pi i}\int_{C}\frac{E(\phi,s,\cdot)}{(s-s_0)^{k}}\, ds,\]
for all $\phi\in\Ind_{\mc{P}(\A)}^{\mc{G}(\A)}(\pi)$ and all monomials $D\in\Sym(\C)_{s_0}$ which are pure of degree $D$; compare \cite[p. 96]{RSson1}. It is not difficult to check that this is indeed a map of $\mc{G}(\A_f)\times(\mf{g},K_\infty)$-modules, using the fact that the poles of Eisenstein series at $s=s_0$ coming from $\pi$ are at most simple. Moreover, the simplicity of such poles implies that $\mc{E}$ is surjective, by definition of $\mc{A}_{\pi,s_0}$.

In fact, the kernel of $\mc{E}$ can be described. Write $\iota$ for the map
\[\iota:\Ind_{\mc{P}(\A)}^{\mc{G}(\A)}(\pi,s_0)\hookrightarrow\Ind_{\mc{P}(\A)}^{\mc{G}(\A)}(\pi\otimes\Sym(\C)_{s_0})\]
from Proposition \ref{propinclofinds}. The source induced representation has the natural surjective map to $\mc{L}(\pi,s_0)$, and we denote the kernel of this map by $\mc{K}$. Then by Grbac's theorem, Theorem \ref{thmgrbac} above, we must have that $\iota(\mc{K})=\ker(\mc{E})$.

Now it follows from the cohomology computations of Section \ref{subsecgkcoh} (in particular, the $(d+1)$-degree row of the first table above Corollary \ref{corbdrynontrivial}) that the induced map
\[H^{d+1}(\mf{g},K_\infty^\circ;\Ind_{\mc{P}(\A)}^{\mc{G}(\A)}(\pi,s_0)\otimes E^\vee)\to H^{d+1}(\mf{g},K_\infty^\circ;\mc{L}(\pi,s_0)\otimes E^\vee)\]
is surjective, because it is after splitting into components over all $i\in\mc{I}_{\mr{coh}}$. From the long exact sequence associated with the short exact sequence
\[0\to\mc{K}\to\Ind_{\mc{P}(\A)}^{\mc{G}(\A)}(\pi,s_0)\to\mc{L}(\pi,s_0)\to 0,\]
it follows that
\[H^{d+2}(\mf{g},K_\infty^\circ;\mc{K}\otimes E^\vee)=0.\]

Now take the long exact sequence associated with
\[0\to\iota(\mc{K})\to\Ind_{\mc{P}(\A)}^{\mc{G}(\A)}(\pi\otimes\Sym(\C)_{s_0})\to\mc{A}_{\pi,s_0}\to 0.\]
In degrees $d+1$ and $d+2$, this long exact sequence contains the snippet
\[H^{d+1}(\mf{g},K_\infty^\circ;\Ind_{\mc{P}(\A)}^{\mc{G}(\A)}(\pi\otimes\Sym(\C)_{s_0})\otimes E^\vee)\to H^{d+1}(\mf{g},K_\infty^\circ;\mc{A}_{\pi,s_0})\to H^{d+2}(\mf{g},K_\infty^\circ;\iota(\mc{K})\otimes E^\vee).\]
We just showed rightmost term is zero, and so is the leftmost by Lemma \ref{lemcohofgloinduced}. Hence the middle term is zero as well, which immediately implies part (c) of Theorem \ref{thmmainthm}.\qed

\subsubsection*{Second proof of (c)}

Part (c) of Theorem \ref{thmmainthm} also follows immediately from the following proposition, of which we give an explicit proof.

\begin{proposition}
\label{propbdrymapissurj}
The boundary map $\partial$ of \eqref{eqnseqwithbdry} is surjective.
\end{proposition}

\begin{proof}
The proof will proceed in several steps by a rather explicit computation involving cochains.

\textit{Step 1}. We set the stage and define archimedean cocycles $c_{\infty,i}^{d,\pm}$. First, fix $\phi_{f,s}\in\Ind_{\mc{P}(\A_f)}^{\mc{G}(\A_f)}(\pi_f,s)$ any flat section with nontrivial image in $\mc{L}(\pi,s_0)_f$ at $s=s_0$. It follows from the Langlands--Shahidi method (Theorem \ref{thmLS} above) that the finite adelic intertwining operator, which we will denote abusively by $M(s,\cdot)$, on $\Ind_{\mc{P}(\A_f)}^{\mc{G}(\A_f)}(\pi_f,s)$ for the Weyl element $w_0$ defined above, is such that $M(s,\phi_{s,f})$ has a simple pole at $s=s_0$. Indeed, for any flat section $\phi_{\infty,s}\in\Ind_{\mc{P}(\R)}^{\mc{G}(\R)}(\pi_\infty,s)$, write
\[\phi_s=\phi_{f,s}\otimes\phi_{\infty,s}\in\Ind_{\mc{P}(\A)}^{\mc{G}(\A)}(\pi,s).\]
Then the associated Eisenstein series $E(\phi,s,g)$ has constant term along $\mc P$ given by
\[E_{\mc{N}}(\phi,s,g)=\phi_{s}(g)(1_{\mc{M}})+(M(s,\phi_{f,s})\otimes M(s,\phi_{\infty,s}))(g)(1_{\mc{M}}).\]
Thus if $\phi_{\infty,s}$ has nontrivial image in $\mc{L}(\pi,s_0)_\infty$ at $s=s_0$, then $E(\phi,s,g)$, and hence $E_{\mc{N}}(\phi,s,g)$, has a (necessarily simple) pole, which must then come from the term $M(s,\phi_{f,s})$ above.

Let $i\in\mc{I}_{\mathrm{coh}}$, and let $\Lambda_i$ be as in the bulleted list of notations above. Let $D_{+,i}$ be the discrete series representation of $\mc{G}(\R)^\circ$ with Harish-Chandra parameter $\Lambda_i$, and let $D_{-,i}$ be that with parameter $\Lambda_i-\alpha_{0,i}$. Theorem \ref{thmstdmodforind} (c) gives us exact sequences
\begin{equation}
\label{eqnsespluss0later}
0\to D_{+,i}\oplus D_{-,i}\to \Ind_{\mc{P}(\R)\cap \mc{G}(\R)^\circ}^{\mc{G}(\R)^\circ}(\pi_{\infty,i},s_0)\to \mc{L}(\pi_{\infty,i},s_0)\to 0,
\end{equation}
and
\begin{equation}
\label{eqnsesminuss0later}
0\to \mc{L}(\pi_{\infty,i},s_0) \to \Ind_{\mc{P}(\R)\cap \mc{G}(\R)^\circ}^{\mc{G}(\R)^\circ}(\pi_{\infty,i},-s_0)\to D_{+,i}\oplus D_{-,i}\to 0.
\end{equation}
By Remark \ref{remfourcochains}, we have two linearly independent cochains
\[c_{\infty,i}^{d,\pm}:\sideset{}{^d}\bigwedge(\mf{g}/\mf{k})\to\Ind_{\mc{P}(\R)\cap \mc{G}(\R)^\circ}^{\mc{G}(\R)^\circ}(\pi_{\infty,i},s_0)\otimes E^\vee\]
respectively factoring through $D_{\pm,i}\otimes E^\vee$. Here, of course, $\mf{k}$ is the complexified Lie algebra of $K_\infty^\circ$. These cochains are in fact cocycles because they come from those representing the nonzero cohomology classes in $H^{d}(\mf{g},K_\infty^\circ;D_{\pm,i}\otimes E^\vee)$.

\textit{Step 2}. Now we define global cocycles $c_{i}^{d,\pm}$ into the induced representation which is the quotient in Grbac's sequence. Let $\iota_i$ be the composition of the maps
\begin{align*}
\iota_i:\Ind_{\mc{P}(\A_f)}^{\mc{G}(\A_f)}(\pi_f,s_0)\otimes (D_{+,i}\oplus D_{-,i})\hookrightarrow &\Ind_{\mc{P}(\A_f)\times(\mc{P}(\R)\cap \mc{G}(\R)^\circ)}^{\mc{G}(\A_f)\times \mc{G}(\R)^\circ}(\pi_f\otimes\pi_{\infty,i},s_0)\\
\hookrightarrow &\Ind_{\mc{P}(\A_f)\times(\mc{P}(\R)\cap \mc{G}(\R)^\circ)}^{\mc{G}(\A_f)\times \mc{G}(\R)^\circ}(\pi,s_0)\\
\cong & \Ind_{\mc{P}(\A)}^{\mc{G}(\A)}(\pi,s_0)\\
\hookrightarrow &\Ind_{\mc{P}(\A)}^{\mc{G}(\A)}(\pi\otimes\Sym(\C)_{s_0}),
\end{align*}
where the isomorphism is from Lemma \ref{lemrestrofindglobal}, and the inclusion at the end is given by Proposition \ref{propinclofinds}. Then $\iota_i$ is an injection of $\mc{G}(\A_f)\times(\mf{g},K_\infty^\circ)$-modules. Let us define the cocycles
\[c_{i}^{d,\pm}:\sideset{}{^d}\bigwedge(\mf{g}/\mf{k})\to\Ind_{\mc{P}(\A)}^{\mc{G}(\A)}(\pi\otimes\Sym(\C)_{s_0})\otimes E^\vee\]
by
\[c_{i}^{d,\pm}(Y)=(\iota_i\otimes\id_{E^\vee})(\phi_{f,s}\otimes c_{\infty,i}^{d,\pm}(Y)),\qquad Y\in\sideset{}{^d}\bigwedge(\mf{g}/\mf{k}).\]
We will compute the images of these cocycles under $\partial$.

\textit{Step 3}. To start this computation, we define a lift $\tilde{c}_{i}^{d,\pm}$ of $c_{i}^{d,\pm}$ to a cochain valued in $\mc{A}_{\pi,s_0}\otimes E^\vee$. After we do this, the image $d\tilde{c}_{i}^{d,\pm}$ of $\tilde{c}_{i}^{d,\pm}$ under the boundary map will be valued in $\mc{L}(\pi,s_0)\otimes E^\vee$, and its class in cohomology will be the image of the class of $c_{i}^{d,\pm}$ under $\partial$ by definition.

So let us define
\[\tilde{c}_{i}^{d,\pm}:\sideset{}{^d}\bigwedge(\mf{g}/\mf{k})\to\mc{A}_{\pi,s_0}\otimes E^\vee\]
as follows. For any vector $e\in E$ and any $F\in\mc{A}_{\pi,s_0}$ and $v\in E^\vee$, we can consider the evaluation of $F\otimes v$ on the vector $e$, given explicitly by $v(e)F\in\mc{A}_{\pi,s_0}$, and this evaluation extends to $\mc{A}_{\pi,s_0}\otimes E^\vee$ by linearity. Thus we can define $\tilde{c}_{i}^{d,\pm}$ by defining its evaluation $\tilde{c}_{i}^{d,\pm}(Y)(e)$ on any $Y\in\bigwedge^d(\mf{g}/\mf{k})$ and any $e\in E$, and we do this as follows. Let
\[\tilde{c}_{i}^{d,\pm}(Y)(e)=E(c_{i}^{d,\pm}(Y)(e),s_0,\cdot),\qquad Y\in\sideset{}{^d}\bigwedge(\mf{g}/\mf{k}),\quad e\in E.\]
We claim that this gives a well defined cochain, i.e., that $\tilde{c}_{i}^{d,\pm}$ is $K_\infty^\circ$-equivariant and regular valued.

To see this, we examine the constant terms of $E(c_{i}^{d,\pm}(Y)(e),s_0,\cdot)$ for $Y\in\bigwedge^d(\mf{g}/\mf{k})$ and $e\in E$. The Eisenstein series $E(c_{i}^{d,\pm}(Y)(e),s_0,\cdot)$ has constant term along $\mc{P}$ given by
\begin{multline}
\label{eqnconstofctildeeval}
E_{\mc{N}}(c_{i}^{d,\pm}(Y)(e),s_0,g)=(\phi_{f,s_0}\otimes c_{\infty,i}^{d,\pm}(Y)(e))(g)(1_{\mc{M}})\\
+\lim_{s\to s_0}(M(s,\phi_{f,s})\otimes M(s,c_{\infty,i}^{d,\pm}(Y)(e)_s))(g)(1_{\mc{M}}), 
\end{multline}
where we have written $c_{\infty,i}^{d,\pm}(Y)(e)_s$ for the flat section in $\Ind_{\mc{P}(\R)\cap \mc{G}(\R)^\circ}^{\mc{G}(\R)^\circ}(\pi_{\infty,i},s)$ specializing to $c_{\infty,i}^{d,\pm}(Y)(e)$ at $s=s_0$. As noted above, the intertwined section $M(s,\phi_{f,s})$ has a simple pole at $s=s_0$, and since $c_{\infty,i}^{d,\pm}(Y)(e)\in D_{\pm,i}$, the other intertwined section $M(s,c_{\infty,i}^{d,\pm}(Y)(e)_s)$ has a zero at $s=s_0$, cancelling this pole. Also, up to conjugacy, there are no other parabolic subgroups $\mc{P}'$ along which $E(c_{i}^{d,\pm}(Y)(e),s_0,g)$ has nonvanishing constant term. Indeed, if any such $\mc{P}'$ puts us in the second case of Theorem \ref{thmLS}, then $\mc{P}'$ is conjugate to the opposite of $\mc{P}$, which itself is already conjugate to $\mc{P}$ by $w_0$; moreover, no such $\mc{P}'$ can put us in the third case of Theorem \ref{thmLS} because we already know $w=1$ is not the unique Weyl element fixing $\mc{M}$, as $w=w_0$ does as well. Thus the nonvanishing constant terms of $E(c_{i}^{d,\pm}(Y)(e),s_0,\cdot)$ are all regular, and so $E(c_{i}^{d,\pm}(Y)(e),s_0,\cdot)$ is itself a regular Eisenstein series. Moreover, since the nonvanishing constant terms clearly have the correct $K_\infty^\circ$-type, the map $\tilde{c}_{i}^{d,\pm}$ is $K_\infty^\circ$-equivariant, proving our claim.

The map $\tilde{c}_{i}^{d,\pm}$ thus defined is therefore a cochain. Moreover, because it is a cochain valued in regular Eisenstein series, we have that
\[\tilde{c}_{i}^{d,\pm}=(\mbf{MW}\otimes\id_{E^\vee})\circ c_{i}^{d,\pm},\]
where $\mbf{MW}$ is Franke's mean value map defined in Section \ref{subseceisenstein}. Since the isomorphism of Grbac's Theorem \ref{thmgrbac} above is induced by $\mbf{MW}$, the surjective map in the short exact sequence \eqref{eqngrbacses} is its inverse, and therefore $\tilde{c}_{i}^{d,\pm}$ is indeed a lift of $c_{i}^{d,\pm}$.

\textit{Step 4}. Now we must compute the coboundaries $d\tilde{c}_{i}^{d,\pm}$. We will do this by computing the constant terms of the Eisenstein series which are values of these coboundaries. We apply the definition from \eqref{eqndefofbdrymap} directly. Let $X_0,\dotsc,X_d\in\mf{g}$ and $e\in E$. Then
\begin{multline*}
(d\tilde{c}_{i}^{d,\pm})(X_0\wedge\dotsb\wedge X_d)(e)=\sum_{j=0}^d (-1)^{j}(X_j \tilde{c}_{i}^{d,\pm}(X_0\wedge\dotsb\wedge\widehat{X}_j\wedge\dotsb\wedge X_d))(e)\\
+\sum_{0\leq j<k\leq d}(-1)^{j+k}\tilde{c}_{i}^{d,\pm}([X_j,X_k]\wedge X_0\wedge\dotsb\wedge\widehat{X}_j\wedge\dotsb\wedge\widehat{X}_k\wedge\dotsb\wedge X_d)(e).
\end{multline*}
For brevity, let us write
\[X=X_0\wedge\dotsb\wedge X_d,\]
and
\[X^{(j)}=X_0\wedge\dotsb\wedge\widehat{X}_j\wedge\dotsb\wedge X_d,\]
and
\[X^{(jk)}=[X_j,X_k]\wedge X_0\wedge\dotsb\wedge\widehat{X}_j\wedge\dotsb\wedge\widehat{X}_k\wedge\dotsb\wedge X_d.\]
Then by \eqref{eqnconstofctildeeval}, the automorphic form $(d\tilde{c}_{i}^{d,\pm})(X)(e)$ has constant term along $\mc{P}$ given by
\begin{multline}
\label{eqnconstofcochain}
((d\tilde{c}_{i}^{d,\pm})(X)(e))_{\mc{N}}(g)=\left(\phi_{f,s_0}\otimes\sum_{j=0}^d (-1)^{j}(X_j c_{\infty,i}^{d,\pm}(X^{(j)}))(e)\right)(g)(1_{\mc{M}})\\
\begin{aligned}
&{} +\left(\phi_{f,s_0}\otimes\sum_{0\leq j<k\leq d}(-1)^{j+k}c_{\infty,i}^{d,\pm}(X^{(jk)})(e)\right)(g)(1_{\mc{M}})\\
&{} +\lim_{s\to s_0}\left(M(s,\phi_{f,s})\otimes \sum_{j=0}^d (-1)^{j}M(s,(X_j c_{\infty,i}^{d,\pm}(X^{(j)}))(e)_s)\right)(g)(1_{\mc{M}})\\
&{} +\lim_{s\to s_0}\left(M(s,\phi_{f,s})\otimes \sum_{0\leq j<k\leq d}(-1)^{j+k}M(s,c_{\infty,i}^{d,\pm}(X^{(jk)})(e)_s)\right)(g)(1_{\mc{M}}),
\end{aligned}
\end{multline}
where, like above, the section $(X_j c_{\infty,i}^{d,\pm}(X^{(j)}))(e)_s$ is the flat section in $\Ind_{\mc{P}(\R)\cap \mc{G}(\R)^\circ}^{\mc{G}(\R)^\circ}(\pi_{\infty,i},s)$ specializing to $(X_j c_{\infty,i}^{d,\pm}(X^{(j)}))(e)$ at $s=s_0$.

Now noting that $\pi_{\infty,i}^{w_0}\cong \pi_{\infty,i}$ by Lemma \ref{lemisopiw0pi}, let us write $Mc_{i}^{d,\pm}$ for the map
\[Mc_{i}^{d,\pm}:\sideset{}{^d}\bigwedge(\mf{g}/\mf{k})\to\Ind_{\mc{P}(\A_f)\times(\mc{P}(\R)\cap \mc{G}(\R)^\circ)}^{\mc{G}(\A_f)\times \mc{G}(\R)^\circ}(\pi_f^{w_0}\otimes\pi_{\infty,i},-s_0)\otimes E^\vee,\]
given by
\[Mc_{i}^{d,\pm}(Y)(e)=\lim_{s\to s_0}M(s,\phi_{f,s})\otimes M(s,c_{\infty,i}^{d,\pm}(Y)(e)_s),\qquad Y\in\sideset{}{^d}\bigwedge(\mf{g}/\mf{k}),\quad e\in E.\]
Since $M(s,\phi_{f,s})$ has a simple pole at $s=s_0$ and (crucially!) by Theorem \ref{thmintertwining}, the section $M(s,c_{\infty,i}^{d,\pm}(Y)(e)_s)$ has a \textit{simple} zero at $s=s_0$, the map $Mc_{i}^{d,\pm}$ is nonvanishing, and therefore defines a nonzero $d$-cochain into
\[\Ind_{\mc{P}(\A_f)\times(\mc{P}(\R)\cap \mc{G}(\R)^\circ)}^{\mc{G}(\A_f)\times \mc{G}(\R)^\circ}(\pi_f^{w_0}\otimes\pi_{\infty,i},-s_0)\otimes E^\vee.\]
The $K_\infty^\circ$-equivariance of the intertwining operator at the archimedean place then implies, by Remark \ref{remfourcochains}, that $Mc_{i}^{d,+}$ and $Mc_{i}^{d,-}$, respectively, fill out $K_\infty^\circ$-types in this representation of highest weights $2\rho_{n,i}$ and $2\rho_{n,i}-2\alpha_{0,i}$. Composing with the surjective map of the short exact sequence \eqref{eqnsesminuss0later} after tensoring with $E^\vee$ then gives nontrivial $d$-cochains, which we denote by $\overline{M}c_{i}^{d,\pm}$, into
\[\Ind_{\mc{P}(\A_f)}^{\mc{G}(\A_f)}(\pi_f^{w_0},-s_0)\otimes D_{\pm,i}\otimes E^\vee.\]
These are necessarily cocycles and define nonzero cohomology classes by Proposition \ref{propcomplexesconc} (a), and thus are given by
\begin{multline}
\label{eqnMbarcidpmYe}
\overline{M}c_{i}^{d,\pm}(Y)(e)\\
\begin{aligned}
&=\lim_{s\to s_0}(s-s_0)M(s,\phi_{f,s})\otimes\left(\frac{1}{s-s_0}M(s,c_{\infty,i}^{d,\pm}(Y)(e)_s)\modulo{\mc{L}(\pi_{\infty,i},s_0)}\right)\\
&=a_\pm\left(\res_{s=s_0}M(s,\phi_{f,s})\right)\otimes c_{\infty,i}^{d,\pm}(Y)(e),
\end{aligned}
\end{multline}
for some nonzero numbers $a_\pm\in\C$.

Now note that
\begin{equation*}
\lim_{s\to s_0}\left(M(s,\phi_{f,s})\otimes \sum_{j=0}^d (-1)^{j}M(s,(X_j c_{\infty,i}^{d,\pm}(X^{(j)}))(e)_s)\right)\\
= \sum_{j=0}^d (-1)^{j}X_jMc_{i}^{d,\pm}(X^{(j)})(e)
\end{equation*}
Thus \eqref{eqnconstofcochain} visibly becomes
\begin{equation*}
((d\tilde{c}_{i}^{d,\pm})(X)(e))_{\mc{N}}(g)=(\phi_{f,s_0}\otimes(dc_{\infty,i}^{d,\pm})(X)(e))(g)(1_{\mc{M}})+((dMc_{i}^{d,\pm})(X)(e))(g)(1_{\mc{M}}).
\end{equation*}
But $dc_{\infty,i}^{d,\pm}=0$ because $c_{\infty,i}^{d,\pm}$ is a cocycle. So this becomes
\begin{equation*}
((d\tilde{c}_{i}^{d,\pm})(X)(e))_{\mc{N}}(g)=((dMc_{i}^{d,\pm})(X)(e))(g)(1_{\mc{M}}).
\end{equation*}

Write $\partial_{\infty,i}$ for the boundary map from degree $d$ to degree $d+1$ in cohomology coming from the short exact sequence \eqref{eqnsesminuss0later} tensored with $E^\vee$, so that $\partial_{\infty,i}$ is a map
\[\partial_{\infty,i}:H^d(\mf{g},K_\infty^\circ;(D_{+,i}\oplus D_{-,i})\otimes E^\vee)\to H^{d+1}(\mf{g},K_\infty^\circ;\mc{L}(\pi_{\infty,i},s_0)\otimes E^\vee).\]
Then $\partial_{\infty,i}$ is nontrivial by Corollary \ref{corbdrynontrivial}. Now since $Mc_{i}^{d,\pm}$ is a cochain lifting $\overline{M}c_{i}^{d,\pm}$, the cochain $dMc_{i}^{d,\pm}$ is, by \eqref{eqnMbarcidpmYe} and the very definition of the connecting homomorphism, equal to
\[a_\pm\left(\res_{s=s_0}M(s,\phi_{f,s})\right)\otimes \partial_{\infty,i}(c_{\infty,i}^{d,\pm}),\]
where we are conflating the class $\partial_{\infty,i}(c_{\infty,i}^{d,\pm})$ with its corresponding cocycle, which we can do by Proposition \ref{propcomplexesconc} (b). Thus we obtain
\begin{equation*}
((\partial\tilde{c}_{i}^{d,\pm})(X)(e))_{\mc{N}}(g)=a_\pm\left(\left(\res_{s=s_0}M(s,\phi_{f,s})\right)\otimes(\partial_{\infty,i}c_{\infty,i}^{d,\pm})(X)(e)\right)(g)(1_{\mc{M}}),
\end{equation*}
and since $\partial_{\infty,i}$ is nontrivial, there is a choice of sign $\epsilon_i\in\{+,-\}$ such that $(\partial\tilde{c}_{i}^{d,\epsilon_i})$ is nontrivial. Since $(\partial\tilde{c}_{i}^{d,\epsilon_i})$ visibly has image contained in
\[\mc{L}(\pi,s_0)_f\otimes\mc{L}(\pi_{\infty,i},s_0)\otimes E^\vee,\]
and $\mc{L}(\pi,s_0)_f$ is irreducible, varying over all $i$ shows that $\partial$ is surjective. This completes the proof of the proposition, and hence of Theorem \ref{thmmainthm}.
\end{proof}

\begin{remark}
Note that we did not need to show in the proof above that $\partial\tilde{c}_{i}^{d,\epsilon_i}$ takes values in residual Eisenstein series, since this follows from the cohomological formalism. However, it is not too difficult to construct the polar Eisenstein series whose residue is $(\partial\tilde{c}_{i}^{d,\epsilon_i})(Y)(e)$ for $Y\in\bigwedge^d(\mf{g}/\mf{k})$ and $e\in E$. Indeed, we have that $(\partial_{\infty,i}c_{\infty,i}^{d,\epsilon_i})(Y)(e)\in\mc{L}(\pi_{\infty,i},s_0)$, and that $\mc{L}(\pi_{\infty,i},s_0)$ is the image of the intertwining operator
\[M(s_0,\cdot):\Ind_{\mc{P}(\R)\cap \mc{G}(\R)^\circ}^{\mc{G}(\R)^\circ}(\pi_{\infty,i},s_0)\to\Ind_{\mc{P}(\R)\cap \mc{G}(\R)^\circ}^{\mc{G}(\R)^\circ}(\pi_{\infty,i},-s_0).\]
Therefore there is a flat section $\phi_{\infty,s}'\in\Ind_{\mc{P}(\R)\cap \mc{G}(\R)^\circ}^{\mc{G}(\R)^\circ}(\pi_{\infty,i},s)$ such that
\[(\partial_{\infty,i}c_{\infty,i}^{d,\epsilon_i})(Y)(e)=M(s_0,\phi_{\infty,s}').\]
One can even upgrade $\phi_{\infty,s}'$ to a cocycle
\[\sideset{}{^{d+1}}\bigwedge(\mf{g}/\mf{k})\to\Ind_{\mc{P}(\R)\cap \mc{G}(\R)^\circ}^{\mc{G}(\R)^\circ}(\pi_{\infty,i},s_0)\otimes E^\vee\]
because the $K_\infty^\circ$-type of highest weight $2\rho_n-\alpha_0$ in $\Ind_{\mc{P}(\R)\cap \mc{G}(\R)^\circ}^{\mc{G}(\R)^\circ}(\pi_{\infty,i},s_0)\otimes E^\vee$ is unique (see Remark \ref{remfourcochains}). In any case, the polar Eisenstein series with residue $(\partial\tilde{c}_{i}^{d,\epsilon_i})(Y)(e)$ at $s=s_0$ is
\[a_{\epsilon_i}E(\phi_f\otimes\phi_\infty',s,\cdot),\]
as can be checked using the proof above and an analysis of constant terms.
\end{remark}

It is worthwhile to note the following corollary, which describes completely the Eisenstein cohomology coming from $\pi$ at $s_0$.

\begin{corollary}
In the context of Theorem \ref{thmmainthm}, we have the following. Let $\mc{K}(\pi_f,s_0)$ be the kernel of the map from $\Ind_{\mc{P}(\A_f)}^{\mc{G}(\A_f)}(\pi_f,s_0)$ to its unique irreducible quotient $\mc{L}(\pi,s_0)_f$. Then the Franke--Schwermer summand $\mc{A}_{\pi,s_0}$ has cohomology
\[H^q(\mf{g},K_\infty^\circ;\mc{A}_{\pi,s_0}\otimes E^\vee)\cong\begin{cases}
\mc{L}(\pi,s_0)_f^{\oplus m}&\textrm{if }q=d-1;\\
\mc{K}(\pi_f,s_0)^{\oplus m}&\textrm{if }q=d;\\
0&\textrm{otherwise.}
\end{cases}\]
\end{corollary}

\begin{proof}
This follows from Proposition \ref{propbdrymapissurj} and the long exact sequence associated with \eqref{eqngrbacses}, which is given by \eqref{eqnisoindegdminusone} and \eqref{eqnseqwithbdry}.
\end{proof}

Our methods have thus, in some sense, revealed the boundary map $\partial$ of \eqref{eqnseqwithbdry} as a cohomological incarnation of the finite adelic intertwining operator.

\subsection{Examples}
\label{subsecexamples}
We now give some examples in which we check the hypotheses of Theorem \ref{thmmainthm} in the symplectic case. These examples should illustrate well the process of checking the setup in Sections \ref{subsecsetup} and \ref{subsecds} in specific cases, and they should also indicate the generality in which Theorem \ref{thmmainthm} holds. We begin by setting up notation for symplectic groups.

\begin{setup}
Let $G=Sp_{2n}$ denote the split simple group over $\Q$ of type $C_n$. it is presented via its image under its faithful standard representation by
\[\sset{g\in GL_{2n}}{\tp{g}Jg=J},\qquad J=\pmat{0 & 1_n\\ -1_n & 0},\]
where $1_n$ denotes the $n\times n$ identity matrix. Then a maximal split torus $\mc{T}_s$ is given by the set of diagonal matrices $t$, which necessarily take the form
\[t=\diag(t_1,t_2,\dotsc,t_n,t_1^{-1},t_2^{-1},\dotsc,t_n^{-1})\]
with $t_1,\dotsc,t_n$ arbitrary invertible. For $t$ of this form, and for $i$ an integer with $1\leq i\leq n$, let $e_i:\mc{T}_s\to\mb{G}_m$ be the character sending $t$ to $t_i$. Then the root system is given by the set
\[\Delta=\Delta(G,\mc{T}_s)=\sset{\pm(e_i+e_j),\,\pm(e_i-e_j),\,\pm 2e_i}{1\leq i<j\leq n}.\]
If a system of positive roots is taken to be those above without a minus sign in front,
\[\Delta^{+,\mr{std}}=\Delta(G,\mc{T}_s)^{+,\mr{std}}=\sset{(e_i+e_j),\,(e_i-e_j),\,2e_i}{1\leq i<j\leq n},\]
then the simple roots are given by
\begin{equation}
\label{eqnspsimpleroots}
e_1-e_2,\, e_2-e_3,\, \dotsc, e_{n-1}-e_n,\, 2e_n.
\end{equation}

Let us assume for simplicity that $n\geq 2$; when $n=1$, we have $S\mc{P}_2\cong SL_2$. The group $Sp_{2n}$ then has two conjugacy classes of maximal parabolic subgroups $\mc{P}=\mc{M}\mc{N}$ such that $\mc{M}(\R)$ has discrete series. In fact, for all integers $r$ with $1\leq r\leq n$, there is a unique conjugacy class of maximal parabolic subgroups whose Levi factors are isomorphic to $GL_r\times Sp_{2n-2r}$ (where we view $Sp_0$ as the trivial group if $r=n$); the standard representative for such a conjugacy class, which we denote by $\mc{P}_r$, is the one whose Levi omits the $r$th simple root from the list \eqref{eqnspsimpleroots} above. The two of these parabolic subgroups $\mc{P}_r$ whose Levi factors have real points with discrete series are $\mc{P}_1$ and $\mc{P}_2$.

The group $Sp_{2n}(\R)$ is connected. Its maximal compact subgroup $K_\infty$ is isomorphic to the group $U(n)$ of Hermitian $n\times n$ invertible matrices with complex entries. The complexification of $K_\infty$ is thus isomorphic to $GL_n(\C)$.

Let $T\subset K_\infty$ be a maximal torus, which is thus a compact Cartan subgroup of $Sp_{2n}(\R)$. Assume a Cartan involution $\theta$ on $Sp_{2n}(\R)$ is fixed giving $K_\infty$. Let $\mf{g}$, $\mf{k}$, and $\mf{t}$ be the complexified Lie algebras of $Sp_{2n}(\R)$, $K_\infty$, and $T$, respectively. Then we may and will identify the set $\Delta(\mf{g},\mf{t})$ of roots of $\mf{t}$ in $\mf{g}$ with the set $\Delta$ above in such a way that the set $\Delta_c$ of compact roots, namely $\Delta(\mf{k},\mf{t})$, is identified with the set
\[\Delta_c=\sset{\pm(e_i-e_j)}{1\leq i<j\leq n}.\]

The Weyl group $W_c$ of $\Delta_c$ may be identified as usual with the symmetric group $S_n$, and it acts in the natural way by permuting the indices of the $e_i$'s. Then the full Weyl group $W$ for $Sp_{2n}$ may be identified with the semidirect product
\[W=S_n\ltimes \{\pm\}^{\oplus n},\]
where the $i$th sign acts via $e_i\mapsto -e_i$, and preserving $e_j$ if $j\ne i$.

Now an integral weight $\lambda$ for $T$ may be identified with an $n$-tuple of integers,
\[\lambda=(\lambda_1,\dotsc,\lambda_n),\qquad \lambda_i\in\Z,\]
where we identify this tuple with the weight $\sum_{i=1}^n\lambda_i e_i$. Such a weight is dominant for the system $\Delta^{+,\mr{std}}$ if and only if $\lambda_1\geq \dotsb\geq \lambda_n\geq 0$, and it is dominant for the induced system of positive roots in $\Delta_c$ if and only if $\lambda_1\geq \dotsb\geq \lambda_n$ (with no sign condition on $\lambda_n$). Moreover, we then have
\begin{equation}
\label{eqnsymprhos}
\rho=(n,n-1,\dotsc,1),\qquad\rho_c=(\tfrac{n-1}{2},\tfrac{n-3}{2},\dotsc,\tfrac{1-n}{2}),\qquad\rho_n=(\tfrac{n+1}{2},\dotsc,\tfrac{n+1}{2}).
\end{equation}
So the Harish-Chandra parameters for $Sp_{2n}$ in the positive chamber for $\Delta^{+,\mr{std}}$ are those integral weights $\lambda$ as above with $\lambda_1>\dotsb >\lambda_n>0$. It follows that a general Harish-Chandra parameter for $Sp_{2n}$ may be viewed as an integral weight $\lambda$ with $\vert\lambda_1\vert>\dotsb >\vert\lambda_n\vert>0$, since the chambers containing such weights are in a natural bijection with the signs $\{\pm\}^{\oplus n}=W_c\backslash W$.
\end{setup}

With this setup we can now give examples involving the parabolic subgroups $\mc{P}_1$ and $\mc{P}_2$ above. We deal with the latter first.

\begin{example}
\label{exgl2sp2nminus4}
Fix an element $\epsilon=(\epsilon_1,\dotsc,\epsilon_n)\in\{\pm\}^{\oplus n}\subset W$. We work in the chamber $\Delta^\epsilon=\epsilon\Delta^{+,\mr{std}}$. Assume there is an integer $i_0$ with $1\leq i_0\leq n-1$ such that $\epsilon_{i_0}\ne\epsilon_{i_0+1}$. Then the set $\Delta^\epsilon$ contains the noncompact simple root $\epsilon(e_{i_0}-e_{i_0+1})=\epsilon_{i_0}(e_{i_0}+e_{i_0+1})$ of $\mf{t}$ in $\mf{g}$, and in the notation of Section \ref{subsecsetup}, we take
\[\alpha_0=\epsilon(e_{i_0}-e_{i_0+1}).\]
Then the Lie algebra $\mf{a}$ constructed from $\alpha_0$ as in \textit{loc. cit.} is, by definition, conjugate by an element of $Sp_{2n}(\C)$ to the Lie algebra of the diagonal torus in the $SL_2$-subgroup of $Sp_{2n}(\C)$ determined by the root $\alpha_0$. The Levi subgroup $M$ of $Sp_{2n}(\R)$, which by definition is the centralizer of $\mf{a}$, therefore has complexified Lie algebra $\mf{m}$ conjugate to the maximal Levi subalgebra containing the roots orthogonal to $\alpha_0$. This Levi subalgebra is the one containing $\epsilon(e_{i_0}+e_{i_0+1})=\epsilon_{i_0}(e_{i_0}-e_{i_0+1})$ and the roots that are combinations of $e_i$'s with $1\leq i\leq n$ and $i\ne i_0,i_0+1$. Conjugating by the element $w_1=(1,i_0)(2,i_0+1)$ in $W$, we see that $M$ is conjugate to the Levi $\mc{M}_2(\R)$ of $\mc{P}_2(\R)$; for clarity, here we have written $(1,i_0)$ and $(2,i_0+1)$ for the corresponding transpositions in $S_n$. Therefore the parabolic subgroup $P$ determined as in Section \ref{subsecsetup} by $\alpha_0$ in positive system $\Delta^\epsilon$ is conjugate to $\mc{P}_2(\R)$.

Now we consider a Harish-Chandra parameter $\Lambda=(\lambda_1,\dotsc,\lambda_n)$ which is dominant for $\Delta^\epsilon$ and which satisfies the condition that
\[1=\frac{2\langle\alpha_0,\Lambda\rangle}{\langle\alpha_0,\alpha_0\rangle}=\epsilon_{i_0}\lambda_{i_0}-\epsilon_{i_0+1}\lambda_{i_0+1},\]
as in condition \eqref{eqnlanranlamrho}. Note that $\epsilon_{i}\lambda_{i}>0$ for all $i$ by the dominance condition, so this may be rewritten as
\[\vert\lambda_{i_0}\vert-\vert\lambda_{i_0+1}\vert=1.\]

A Harish-Chandra parameter for $\mc{M}_2(\R)\cong GL_2(\R)\times Sp_{2n-4}(\R)$ may be seen as a tuple $(\mu_0;\mu_1,\mu_2,\dotsc,\mu_{n-2})$ of integers, where $\vert\mu_0\vert\ne 0$ and $\vert\mu_1\vert>\dotsb >\vert\mu_{n-2}\vert>0$. To restrict $\Lambda$ to a Harish-Chandra parameter for $\mc{M}_2(\R)$, we note that, writing $\mf{m}$ for the complexified Lie algebra of $M$, we have that $\mf{t}'=\mf{t}\cap\mf{m}$ is a compact Cartan subalgebra of $\mf{m}$ and it is given as the orthogonal complement of $H_{\alpha_0}\in\mf{t}$. It follows that the Lie algebra of the $GL_2(\R)$-factor of $M$ has compact Cartan subalgebra spanned by the normalized root vector $H_{\alpha_0^\perp}$ where $\alpha_0^\perp=\epsilon(e_{i_0}+e_{i_0+1})$. So we have
\[\Lambda(H_{\alpha_0^\perp})=\epsilon_{i_0}\lambda_{i_0}+\epsilon_{i_0+1}\lambda_{i_0+1}=2\vert\lambda_{i_0+1}\vert+1.\]
Clearly, on the $Sp_{2n-4}(\R)$-factor of $M$, the parameter $\Lambda$ restricts to the $(n-2)$-tuple
\[(\lambda_1,\lambda_2,\dotsc,\widehat{\lambda}_{i_0},\widehat{\lambda}_{i_0+1},\dotsc,\lambda_n).\]
This determines the restricted Harish-Chandra parameter for $\mc{M}_2(\R)$; it is
\[\mu=(2\vert\lambda_{i_0+1}\vert+1;\lambda_1,\lambda_2,\dotsc,\widehat{\lambda}_{i_0},\widehat{\lambda}_{i_0+1},\dotsc,\lambda_n).\]
As usual, the hats mean we remove that entry.

It is not hard to write down the discrete series representation, which we will call $\pi_\infty'$ here, corresponding to this parameter $\mu$ as in Section \ref{subsecds} (wherein we called it $\pi$ instead). To describe it, we need to understand the components of $\mc{M}_2(\R)$. We note that $\mc{M}_2(\R)\cong GL_2(\R)\times Sp_{2n-4}(\R)$ has two connected components, the identity component splitting as
\[\mc{M}_2(\R)^\circ\cong(\R_{>0}\cdot SL_2(\R))\times Sp_{2n-4}(\R),\]
with the $\R_{>0}$ component in center of the $GL_2$-factor being the identity component $A_{\mc{M}_2}(\R)^\circ$ of the split center. The center of $\mc{M}_2(\R)$ is contained in $\mc{M}_2(\R)^\circ$. Let us write $\mc{M}_2(\R)_0^\circ$ and $\mc{M}_2(\R)_0$ for the semisimple components of $\mc{M}_2(\R)^\circ$ and $\mc{M}_2(\R)$, respectively, in the Langlands decomposition. Then
\[\mc{M}_2(\R)_0^\circ\cong SL_2(\R)\times Sp_{2n-4}(\R),\]
and
\[\mc{M}_2(\R)_0\cong SL_2(\R)^{\pm}\times Sp_{2n-4}(\R),\]
where $SL_2(\R)^{\pm}\subset GL_2(\R)$ is the group of matrices with determinant $\pm 1$.

So then we start with the discrete series representation $\pi_{0,\infty}'$ of $\mc{M}_2(\R)_0^\circ$ with parameter $\mu$ as above. Since the center $Z_{\mc{M}_2(\R)_0}$ is already contained in $\mc{M}_2(\R)_0^\circ$, there is no need to specify a central character $\xi$ as in Section \ref{subsecds}. Then we induce:
\[\pi_\infty'=\Ind_{\mc{M}_2(\R)_0^\circ}^{\mc{M}_2(\R)_0}(\pi_{0,\infty}').\]
This induction has the following effect. Let $\mu_{Sp}$ be the Harish-Chandra parameter for $Sp_{2n-4}(\R)$ given by
\[\mu_{Sp}=(\lambda_1,\lambda_2,\dotsc,\widehat{\lambda}_{i_0},\widehat{\lambda}_{i_0+1},\dotsc,\lambda_n),\]
and $\mu_{SL}$ the Harish-Chandra parameter of $SL_2(\R)$ given by the integer $2\vert\lambda_{i_0+1}\vert+1$. If $\pi_{0,\infty,Sp}'$ and $\pi_{0,\infty,SL}'$ are the corresponding representations, then $\pi_{0,\infty,SL}'$ is the weight $(2\vert\lambda_{i_0+1}\vert+2)$ holomorphic discrete series, and we have
\[\pi_{0,\infty}'=\pi_{0,\infty,SL}'\boxtimes \pi_{0,\infty,Sp}'\]
as a representation of $\mc{M}_2(\R)_0^\circ\cong SL_2(\R)\times Sp_{2n-4}(\R)$. Let $\pi_{\infty,SL^{\pm}}'$ be the usual discrete series representation of $SL_2(\R)^\pm$ of weight $(2\vert\lambda_{i_0+1}\vert+2)$; as a representation of $SL_2(\R)$ it is a direct sum of both holomorphic and antiholomorphic discrete series, but the nonidentity component of $SL_2(\R)^\pm$ switches between them. Then we have
\[\pi_{\infty}'=\pi_{\infty,SL^\pm}'\boxtimes \pi_{0,\infty,Sp}'.\]

We then write $\pi_\infty$ for the representation of $\mc{M}_2(\R)$ given by the representation $\pi_\infty'$ on $\mc{M}_2(\R)_0$ and with trivial character along $A_{\mc{M}_2}(\R)^\circ$, so
\[\pi_\infty=\pi_{\infty}'\otimes 1\qquad\textrm{on}\qquad\mc{M}_2(\R)_0\cdot A_{\mc{M}_2}(\R)^\circ.\]
The cuspidal automorphic representations $\pi$ of $\mc{M}_2(\A)$ with archimedean component one of these representations $\pi_\infty$ are the ones out of which we can build the Eisenstein series which could be considered for our main theorem (and by Proposition \ref{propaltsetup}, there can be no other such cuspidal representations of $\mc{M}_2(\A)$ which have discrete series archimedean component and which could possibly contribute to residual Eisenstein cohomology, at least at the relevant value $s_0$). The condition that such a $\pi$ has such an archimedean component amounts therefore to a mild regularity condition on the weight of the $Sp_{2n-4}(\R)$-component, along with the condition that the weight of the $GL_2(\R)$-component be at least $4$ and even, as well as a minor separation condition on the weights for both components (which is seen in our setup through the position of the index $i_0$).

Now let us fix such a cuspidal automorphic representation $\pi$ of $\mc{M}_2(\A)$ with archimedean component this $\pi_\infty$. We next explain how to use $L$-functions to determine whether or not the Eisenstein series built from inducing $\pi$, along with the parameter $s_0$ from Theorem \ref{thmmainthm}, have poles, and hence whether they can contribute to residual Eisenstein cohomology.

First we compute $s_0$. We note that the positive roots of the standard split torus $\mc{T}_s$ contained in $\mc{N}_2$ are those of the form
\[2e_1,\,\,2e_2,\,\,e_1+e_2,\,\, e_1\pm e_i,\,\,e_2\pm e_i,\,\,\textrm{for }3\leq i\leq n.\]
The sum of these is $2\rho_{\mc{P}_2}|_{\mc{T}_s}=(2n-1)(e_1+e_2)$. We claim the root $e_1+e_2$, which is in $\mc{N}_2$, vanishes on the orthogonal complement of the complexified Lie algebra of $A_{\mc{M}_2}$ in the complexified Lie algebra of the split torus $\mc{T}_s$. Indeed, the torus $A_{\mc{M}_2}$ is the center of the $GL_2$-factor of $\mc{M}_2$. It is therefore image of the cocharacter $e_1^\vee+e_2^\vee$, where we write $e_i^\vee:\mb{G}_m\to\mc{T}_s$ for the cocharacter given by
\[e_i^\vee(t)=\diag(1,\dotsc,t,\dotsc,1,1,\dotsc,t^{-1},\dotsc,1);\]
here, the $t$ is in the $i$th slot, and the $t^{-1}$ is in the $(n+i)$th slot. The cocharacters $e_i^\vee$ form the dual basis to the characters $e_i$ under composition, and so by duality, $e_1+e_2$ vanishes on the orthogonal complement of the complexified Lie algebra of $A_{\mc{M}_2}$ in that of $\mc{T}_s$. Thus we have, by definition of $s_0$, that $\frac{1}{2}(e_1+e_2)|_{A_{\mc{M}_2}}=2s_0\rho_{\mc{P}_{2}}|_{A_{\mc{M}_2}}$, and hence that
\[s_0=\frac{1}{4n-2}.\]

The Langlands--Shahidi method will express the constant terms of the Eisenstein series built from $\pi$ in terms of $L$-functions. In our case, the process will be as follows; see \cite[\S 6.3]{shahidiES} for the general case.

Let $SO_{2n+1}$ denote the usual split special orthogonal group, which is a simple group of type $B_n$; it is Langlands dual to $Sp_{2n}$. Its root system $\Delta^\vee$ is given by the following coroots for $Sp_{2n}$:
\[\Delta^\vee=\sset{\pm(e_i^\vee+e_j^\vee),\,\pm(e_i^\vee-e_j^\vee),\,\pm e_i^\vee}{1\leq i<j\leq n}.\]

Now the parabolic $\mc{P}_2$ in $Sp_{2n}$ is dual to a parabolic $\mc{P}_2^\vee$ in $SO_{2n+1}$ with Levi decomposition $\mc{M}_2^\vee \mc{N}_2^\vee$, where $\mc{M}_2^\vee\cong GL_2\times SO_{2n-3}$. The roots in $\Delta^\vee$ contained in $\mc{N}_2^\vee$ are
\[e_1^\vee,\,\,e_2^\vee,\,\,e_1^\vee+e_2^\vee,\,\, e_1^\vee\pm e_i^\vee,\,\,e_2^\vee\pm e_i^\vee,\,\,\textrm{for }3\leq i\leq n.\]
The weights of the standard $(2n-3)$-dimensional representation of the $SO_{2n-3}$-factor, call it $\std_{SO_{2n-3}}$, are $\pm e_i^\vee$ for $3\leq i\leq n$ along with the $0$ weight. Those for the standard representation $\std_{GL_2}$ of the $GL_2$-factor are $e_1^\vee+e_2^\vee$. So $\Lie(\mc{N}_2^\vee)$ contains all the weights for $\std_{GL_2}\boxtimes \std_{SO_{2n-3}}$, along with the weight $e_1^\vee+e_2^\vee$ which is the determinant on the $GL_2$-factor of $\mc{M}_2^\vee$. Therefore it decomposes as a representation of $GL_2\times SO_{2n-3}$ as
\[\Lie(\mc{N}_2^\vee)\cong (\std_{GL_2}\boxtimes \std_{SO_{2n-3}})\oplus (\det\boxtimes 1).\]
The weights of the first factor above pair with the root $e_1+e_2$ of $Sp_{2n}$ to give $1$, while the second factor pairs with $e_1+e_2$ to give $2$.

For our fixed cuspidal automorphic representation $\pi$ of $\mc{M}_2(\A)$, let us write
\[\pi=\pi_{GL}\boxtimes\pi_{Sp}\]
for cuspidal automorphic representations $\pi_{GL}$ and $\pi_{Sp}$ of $GL_2(\A)$ and $Sp_{2n-4}(\A)$, respectively. Let $S$ be a finite set of places of $\Q$ containing the archimedean place and any finite place where $\pi$ is ramified. Let $\phi_s\in\Ind_{\mc{P}_2(\A)}^{Sp_{2n}(\A)}(\pi,s)$ be a flat section which decomposes over all places $v$ as $\otimes_v\phi_{v,s}$ with
\[\phi_{v,s}\in\Ind_{\mc{P}_2(\Q_v)}^{Sp_{2n}(\Q_v)}(\pi_v,s),\]
where $\pi_v$ is the local component of $\pi$ at $v$. Assume $\phi_{v,s}$ is spherical for all $v\notin S$. Let $w_0=(-)_1(-)_2$ be the Weyl element for $Sp_{2n}$ which flips the signs of $e_1$ and $e_2$. Then $w_0\mc{P}_2 w_0^{-1}$ is opposite to $\mc{P}_2$, and $w_0^{-1}\mc{M}_2w_0=\mc{M}_2$. For brevity, write
\[R_1=\std_{GL_2}\boxtimes \std_{SO_{2n-3}}.\]
Then the Langlands--Shahidi method then gives
\[M(s,\phi)=\frac{L^S((2n-1)s,\pi,R_1)L^S((4n-2)s,\pi_{GL_2},\det)}{L^S((2n-1)s+1,\pi,R_1)L^S((4n-2)s+1,\pi_{GL_2},\det)}\bigotimes_{v\notin S}\phi_{v,-s}\otimes\bigotimes_{v\in S}M(s,\phi_v).\]
Here, we have written $M(s,\cdot)$ for the intertwining operator defined with respect to $w_0$, and for an algebraic representation $R$ of $\mc{M}_2^\vee$, the function $L^S(s,\pi,R)$ denotes the usual Langlands $L$-function with Euler factors only at the places not in $S$. Precisely, if $v\notin S$ corresponds to a rational prime $\ell_v$ and $t_v\in \mc{M}_2^\vee(\C)$ denotes the Satake parameter of $\pi_v$, then the corresponding Euler factor is
\[L_v(s,\pi_v,R)=\det(1-R(t_v)\ell_v^{-s})^{-1}.\]
Then by definition we have
\[L^S(s,\pi,R)=\prod_{v\notin S}L_v(s,\pi_v,R),\]
when $\re(s)$ is sufficiently large. In the case of $R=R_1$, the expression above implies the meromorphic continuation of $L^S(s,\pi,R_1)$, because we know the meromorphic continuation of the Eisenstein series whose constant term contains the intertwined section above as a summand by Theorem \ref{thmLS}, and we know the meromorphic continuation of the Hecke $L$-function $L^S(s,\pi_{GL_2},\det)$ classically.

Now assume that the central character of $\pi_{GL_2}$ is trivial. Then the $L$-function
\[L^S((4n-2)s,\pi_{GL_2},\det)\]
appearing above is the incomplete zeta function $\zeta^S((4n-2)s)$ and therefore has a simple pole at $s=s_0$. The $L$-value
\[L^S((4n-2)s_0+1,\pi_{GL_2},\det)\]
is not divergent because it equals $\zeta^S(2)$. One should expect, at least when $\pi$ has, say, neither CAP nor endoscopic factors, that the other $L$-function appearing in the denominator above, namely
\[L^S((2n-1)s+1,\pi,R_1),\]
should converge as an Euler product at $s=s_0$ for temperedness reasons. As well, the local intertwined sections $M(s,\phi_v)$ could be chosen to be nonvanishing for the same temperedness reasons. Finally, the factor in the numerator,
\[L^S((2n-1)s,\pi,R_1),\]
should have $s=s_0$ as a central critical value, and one should expect that it either may or may not vanish at $s=s_0$. Thus altogether, if $\pi$ is tempered and $L^S(1/2,\pi,R_1)\ne 0$, then the intertwining operator $M(s,\phi)$ has a (simple) pole at $s=s_0$ for some section $\phi_s$, and hence so does the Eisenstein series $E(s,\phi,\cdot)$. Our Theorem \ref{thmmainthm} will then apply to the residual representation $\mc{L}(\pi,s_0)$, showing that its image in automorphic cohomology in degree $\frac{n(n+1)}{2}-1$ contributes one copy of $\mc{L}(\pi,s_0)_f$ (and not more, because $Sp_{2n}(\R)$ is connected), but that its image vanishes in degree $\frac{n(n+1)}{2}+1$, despite it being cohomological itself in that degree.
\end{example}

\begin{example}
Keeping the notation $\epsilon=(\epsilon_1,\dotsc,\epsilon_n)$ and $\Delta^\epsilon$ from the beginning of Example \ref{exgl2sp2nminus4} above (but with no assumption at any fixed index $i_0$ as we had there), let us now quickly describe what happens in the remaining case, namely when the fixed noncompact root $\alpha_0$ is $2\epsilon_n e_n$. We will be somewhat brief, leaving the rest of the details to the interested reader. When $\alpha_0=2\epsilon_n e_n$, the corresponding real Levi $M$ is conjugate to the $\R$-points of $\mc{M}_1\cong GL_1\times Sp_{2n-2}$. The allowable Harish-Chandra parameters $\Lambda=(\lambda_1,\dotsc,\lambda_n)$ are those with
\[\epsilon_1\lambda_1>\dotsb >\epsilon_{n-1}\lambda_{n-1}>\epsilon_n\lambda_n=1.\]
Their restriction to the $Sp_{2n-2}(\R)$-factor of $M$ are given by $(\lambda_1,\dotsc,\lambda_{n-1})$. Thus the additional regularity condition imposed on these parameters for $Sp_{2n-2}(\R)$ is that $\lambda_{n-1}\geq 2$; otherwise these Harish-Chandra parameters are arbitrary.

Now, in this case, the semisimple component $\mc{M}_1(\R)_0$ of $\mc{M}_1(\R)$ in its Langlands decomposition is
\[\mc{M}_1(\R)_0\cong\{\pm 1\}\times Sp_{2n-2}(\R).\]
It has two connected components and the center meets both. The character $\xi$ of Section \ref{subsecds} in this case is determined by its restriction to the $-1$ in the first factor, and is given there by the $n$th component of $\Lambda-\rho_c+\rho_n$. By \eqref{eqnsymprhos}, this is entry given by $n+\epsilon_n$, and so we must require that $\xi|_{\{\pm 1\}}(-1)=(-1)^{n+\epsilon_n}=(-1)^{n-1}$. This means that the discrete series representation $\pi_\infty$ of $\mc{M}_1(\R)$ will have the form
\[\pi_\infty=\sign^{n-1}\boxtimes\pi_{\infty,Sp},\]
where $\sign:GL_1(\R)\to\{\pm 1\}$ is the sign representation of $GL_1(\R)$, and $\pi_{\infty,Sp}$ is a discrete series representation of $Sp_{2n-2}(\R)$ with Harish-Chandra parameter $(\lambda_1,\dotsc,\lambda_{n-1})$ satisfying $\vert\lambda_{n-1}\vert\geq 2$.

Now let $\pi$ be a cuspidal automorphic representation of $\mc{M}_1(\A)$ with archimedean component given by $\pi_\infty$ as above. It is given by a Dirichlet character $\chi$ on the $GL_1(\A)$ factor with parity that of $n-1$, and some cuspidal representation $\pi_{Sp}$ of $Sp_{2n-2}$ with archimedean component $\pi_{\infty,Sp}$. Let $S$ be a finite set of places containing the archimedean one and those for which $\pi$ is ramified. One checks that the root $\alpha$ giving $\widetilde{\alpha}$ as in the notation of Theorem \ref{thmmainthm} is $2e_1$, and that $s_0=1/2n$. Write $\std_{SO_{2n-1}}$ for the standard representation of $SO_{2n-1}(\C)$. Then the Langlands--Shahidi $L$-function appearing in the constant term of the Eisenstein series induced from $\pi$ along $\mc{P}_1(\A)$ is a quotient of twisted $L$-functions,
\[\frac{L^S(2ns,\pi_{Sp}\otimes\chi,\std_{SO_{2n-1}})}{L^S(2ns+1,\pi_{Sp}\otimes\chi,\std_{SO_{2n-1}})}.\]
Here, the local Euler factor of $L^S(s,\pi_{Sp}\otimes\chi,\std_{SO_{2n-1}})$ at $v\notin S$ is given by
\[L_v(2ns,\pi_{Sp}\otimes\chi,\std_{SO_{2n-1}})=\det(1-\chi(\ell_v)\std_{SO_{2n-1}}(t_v)\ell_v^{-s})^{-1},\]
where $t_v$ is the Satake parameter for $\pi$ at $v$ and $\ell_v$ is the rational prime corresponding to $v$.

Now, the point of evaluation $s_0$ thus returns the non-central critical value at $s=1$ for the $L$-function $L^S(s,\pi_{Sp}\otimes\chi,\std_{SO_{2n-1}})$. We expect that very often this $L$-function is entire, and also that it is nonvanishing when $s>1$. In this case the Langlands--Shahidi quotient of $L$-functions above cannot have a pole at $s=s_0$. However, the situation is more interesting when $\pi_{Sp}$ is CAP or endoscopic.

For an example in the endoscopic case, consider when $n=2$. Write $\pi_{SL}=\pi_{Sp}$, which will be an automorphic representation of $SL_2(\A)=Sp_{2}(\A)$. Assume that $\pi_{SL}$ is attached to either a holomorphic or antiholomorphic CM modular form of even weight $k\geq 4$ and trivial nebentypus associated with a Hecke character $\psi$ defined over an imaginary quadratic field $F$. Its Harish-Chandra parameter will be given by $\pm(k-1)$, the sign depending on its holomorphicity or antiholomorphicity. Let $\eta_F$ be the imaginary quadratic character associated with $F$, and let $\chi=\eta_F$. Then $\chi$ is odd, so that the parity condition is satisfied. The standard representation $\std_{SO_{3}}$ is the symmetric square for $SL_2^\vee=PGL_2(\C)$, and so the $L$-function
\[L^S(s,\pi_{SL}\otimes\chi,\std_{SO_{3}})\]
splits into two factors, namely
\[L^S(s,\pi_{SL}\otimes\chi,\std_{SO_{3}})=L^S(s+k-1,\psi^2\chi|_{F})L^S(s,\eta_F\chi),\]
where the first $L$-function on the right hand side is a Hecke $L$-function over $F$ and the second is one over $\Q$; see \cite[Corollary 3.4]{leiCM}, for instance (though our normalizations are different here, hence the shift by $k-1$ from \textit{loc. cit.}). Of course, we have $\chi|_F=\eta_F|_F$ is trivial, as is $\eta_F\chi=\eta_F^2$, so we have
\[L^S(4s,\pi_{SL}\otimes\chi,\std_{SO_{3}})=L^S(4s+k-1,\psi^2)\zeta^S(4s).\]
This has a pole at $s_0=1/4$. Our Theorem \ref{thmmainthm} thus applies to the residual representation $\mc{L}(\pi,1/4)$.

For an example in the CAP case, consider when $n=3$. Now let $\pi_{Sp}$ be attached to a holomorphic or antiholomophic Saito--Kurokawa form, this form itself coming from a modular form $f$ of even weight $k\geq 6$ and trivial nebentypus with $L$-function $L(s,f)$ vanishing to odd order at the central value $k/2$. The odd order vanishing condition here guarantees that such a $\pi_{Sp}$ exists with archimedean components having Harish-Chandra parameter $\pm(\tfrac{k}{2},\tfrac{k-2}{2})$, and so the condition that $k\geq 6$ guarantees that the second component of this weight is at least $2$ in absolute value. Let $\chi$ be the trivial character, which is even now, and hence satisfies the required parity condition. Then one computes that the relevant Langlands--Shahidi $L$-function splits as
\[L^S(6s,\pi_{Sp},\std_{SO_{5}})=L^S(f,6s-1+\tfrac{k-1}{2})L^S(f,6s-1+\tfrac{k+1}{2})\zeta^S(6s).\]
Hence the quotient appearing in the constant term is
\[\frac{L^S(f,6s-1+\tfrac{k-1}{2})\zeta^S(6s)}{L^S(f,6s+\tfrac{k+1}{2})\zeta^S(6s+1)}.\]
The first $L$-function in the numerator does not vanish at $s=s_0=1/6$ by Shahidi's prime number theorem \cite[\S 5]{shahidiPNT} along with the functional equation, as long as we have not removed an Euler factor with a pole at $s=s_0$ from it. One checks that such an Euler factor exists at a finite place $v$ if and only if $f$ is unramified at $v$ and $1$ is a Satake parameter for $f$ at $v$ (with unitary normalizations). In this case, the other Satake parameter at $v$ is also $1$. Conjecturally, this never happens, but we could always omit $v$ from the set $S$ in any case since $f$ is unramified at $v$. Thus this quotient of $L$-function acquires a pole at $s=1/6$ in this case, and so our Theorem \ref{thmmainthm} applies to the residual representation $\mc{L}(\pi,1/6)$.
\end{example}

For one final remark, let $\pi$ be an automorphic representation of $GL_2(\A)$ attached to a holomorphic cuspidal eigenform of even weight at least $4$ and trivial nebentypus. In \cite{mung2cohart}, we considered Eisenstein series for the split exceptional group $G_2$ attached to such $\pi$, induced along the long root parabolic subgroup. The symmetric cube $L$-function $L(s,\pi,\Sym^3)$ appears in the constant term of those Eisenstein series, and in fact the numerator of the quotient of $L$-functions appearing there is
\[L^S(5s,\pi,\Sym^3)\zeta^S(10s).\]
That paper was primarily concerned with the case when the case when $L(s,\pi,\Sym^3)$ vanishes to odd order at $s=1/2$. But in the case when $L(1/2,\pi,\Sym^3)\ne 0$, we made the remark that residual Eisenstein series arise at $s=1/10$, and we wondered in \cite[Remark 3.5.5]{mung2cohart} whether the boundary map in cohomology between middle degree $4$ and degree $5$ coming from Grbac's exact sequence was nontrivial. In the context of the present paper, we have $s_0=1/10$ and the hypotheses of Theorem \ref{thmmainthm} are satisfied for these residual Eisenstein series. So we have answered this question completely affirmatively here.
\printbibliography

@article{blank,
author = {Blank, Brian E.},
year = {1985},
pages = {127-145},
title = {Knapp-Wallach Szegö integrals and generalized principal series representations: The parabolic rank one case},
volume = {60},
journal = {Journal of Functional Analysis},
}

@book {BW,
    AUTHOR = {Borel, A. and Wallach, N.},
     TITLE = {Continuous cohomology, discrete subgroups, and representations
              of reductive groups},
    SERIES = {Mathematical Surveys and Monographs},
    VOLUME = {67},
   EDITION = {Second},
 PUBLISHER = {American Mathematical Society, Providence, RI},
      YEAR = {2000}
}

@article {franke,
    AUTHOR = {Franke, Jens},
     TITLE = {Harmonic analysis in weighted {$L_2$}-spaces},
   JOURNAL = {Ann. Sci. \'{E}cole Norm. Sup. (4)},
  FJOURNAL = {Annales Scientifiques de l'\'{E}cole Normale Sup\'{e}rieure. Quatri\`eme
              S\'{e}rie},
    VOLUME = {31},
      YEAR = {1998},
    NUMBER = {2},
     PAGES = {181--279}
}

@article {FS,
    AUTHOR = {Franke, Jens and Schwermer, Joachim},
     TITLE = {A decomposition of spaces of automorphic forms, and the
              {E}isenstein cohomology of arithmetic groups},
   JOURNAL = {Math. Ann.},
  FJOURNAL = {Mathematische Annalen},
    VOLUME = {311},
      YEAR = {1998},
    NUMBER = {4},
     PAGES = {765--790}
}

@article {gotsgrob,
    AUTHOR = {Gotsbacher, Gerald and Grobner, Harald},
     TITLE = {On the {E}isenstein cohomology of odd orthogonal groups},
   JOURNAL = {Forum Math.},
  FJOURNAL = {Forum Mathematicum},
    VOLUME = {25},
      YEAR = {2013},
    NUMBER = {2},
     PAGES = {283--311}
}

@article {grbac,
    AUTHOR = {Grbac, Neven},
     TITLE = {The {F}ranke filtration of the spaces of automorphic forms
              supported in a maximal proper parabolic subgroup},
   JOURNAL = {Glas. Mat. Ser. III},
  FJOURNAL = {Glasnik Matemati\v{c}ki. Serija III},
    VOLUME = {47(67)},
      YEAR = {2012},
    NUMBER = {2},
     PAGES = {351--372}
}

@article {grbsch,
    AUTHOR = {Grbac, Neven and Schwermer, Joachim},
     TITLE = {On residual cohomology classes attached to relative rank one
              {E}isenstein series for the symplectic group},
   JOURNAL = {Int. Math. Res. Not. IMRN},
  FJOURNAL = {International Mathematics Research Notices. IMRN},
      YEAR = {2011},
    NUMBER = {7},
     PAGES = {1654--1705}
}

@article {grobsp11,
    AUTHOR = {Grobner, Harald},
     TITLE = {The automorphic cohomology and the residual spectrum of
              {H}ermitian groups of rank one},
   JOURNAL = {Internat. J. Math.},
  FJOURNAL = {International Journal of Mathematics},
    VOLUME = {21},
      YEAR = {2010},
    NUMBER = {2},
     PAGES = {255--278}
}

@article {grobqres,
    AUTHOR = {Grobner, Harald},
     TITLE = {Residues of {E}isenstein series and the automorphic cohomology
              of reductive groups},
   JOURNAL = {Compos. Math.},
  FJOURNAL = {Compositio Mathematica},
    VOLUME = {149},
      YEAR = {2013},
    NUMBER = {7},
     PAGES = {1061--1090}
}

@article {GrobnerQuat,
    AUTHOR = {Grobner, Harald},
     TITLE = {Automorphic forms, cohomology and {CAP} representations. {T}he
              case {$GL_2$} over a definite quaternion algebra},
   JOURNAL = {J. Ramanujan Math. Soc.},
  FJOURNAL = {Journal of the Ramanujan Mathematical Society},
    VOLUME = {28},
      YEAR = {2013},
    NUMBER = {1},
     PAGES = {19--48}
}

@book {knappbook,
    AUTHOR = {Knapp, Anthony W.},
     TITLE = {Lie groups beyond an introduction},
    SERIES = {Progress in Mathematics},
    VOLUME = {140},
   EDITION = {Second},
 PUBLISHER = {Birkh\"{a}user Boston, Inc., Boston, MA},
      YEAR = {2002},
     PAGES = {xviii+812}
}

@book {knvo,
    AUTHOR = {Knapp, Anthony W. and Vogan, Jr., David A.},
     TITLE = {Cohomological induction and unitary representations},
    SERIES = {Princeton Mathematical Series},
    VOLUME = {45},
 PUBLISHER = {Princeton University Press, Princeton, NJ},
      YEAR = {1995}
}

@book {LanglandsES,
    AUTHOR = {Langlands, Robert P.},
     TITLE = {On the functional equations satisfied by {E}isenstein series},
    SERIES = {Lecture Notes in Mathematics, Vol. 544},
 PUBLISHER = {Springer-Verlag, Berlin-New York},
      YEAR = {1976}
}

@article {leiCM,
    AUTHOR = {Lei, Antonio},
     TITLE = {Iwasawa theory for the symmetric square of a {CM} modular form
              at inert primes},
   JOURNAL = {Glasg. Math. J.},
  FJOURNAL = {Glasgow Mathematical Journal},
    VOLUME = {54},
      YEAR = {2012},
    NUMBER = {2},
     PAGES = {241--259}
}

@article {LS,
    AUTHOR = {Li, Jian-Shu and Schwermer, Joachim},
     TITLE = {On the {E}isenstein cohomology of arithmetic groups},
   JOURNAL = {Duke Math. J.},
  FJOURNAL = {Duke Mathematical Journal},
    VOLUME = {123},
      YEAR = {2004},
    NUMBER = {1},
     PAGES = {141--169}
}

@book {MW,
    AUTHOR = {M{\oe}glin, C. and Waldspurger, J.-L.},
     TITLE = {Spectral decomposition and {E}isenstein series},
    SERIES = {Cambridge Tracts in Mathematics},
    VOLUME = {113},
      NOTE = {Une paraphrase de l'\'{E}criture [A paraphrase of Scripture]},
 PUBLISHER = {Cambridge University Press, Cambridge},
      YEAR = {1995}
}

@misc {mung2cohart,
    AUTHOR = {Mundy, Sam},
     TITLE = {Automorphic cohomology, Arthur's conjectures and applications to $G_2$},
      NOTE = {To appear, Amer. J. Math.}
}

@article {RSson1,
    AUTHOR = {Rohlfs, J\"{u}rgen and Speh, Birgit},
     TITLE = {Representations with cohomology in the discrete spectrum of
              subgroups of {${\rm SO}(n,1)({\bf Z})$} and {L}efschetz
              numbers},
   JOURNAL = {Ann. Sci. \'{E}cole Norm. Sup. (4)},
  FJOURNAL = {Annales Scientifiques de l'\'{E}cole Normale Sup\'{e}rieure.
              Quatri\`eme S\'{e}rie},
    VOLUME = {20},
      YEAR = {1987},
    NUMBER = {1},
     PAGES = {89--136}
}

@incollection {rohlfsspeh,
    AUTHOR = {Rohlfs, J\"{u}rgen and Speh, Birgit},
     TITLE = {Pseudo {E}isenstein forms and the cohomology of arithmetic
              groups {III}: residual cohomology classes},
 BOOKTITLE = {On certain {$L$}-functions},
    SERIES = {Clay Math. Proc.},
    VOLUME = {13},
     PAGES = {501--523},
 PUBLISHER = {Amer. Math. Soc., Providence, RI},
      YEAR = {2011}
}

@article {shahidiPNT,
    AUTHOR = {Shahidi, Freydoon},
     TITLE = {On certain {$L$}-functions},
   JOURNAL = {Amer. J. Math.},
  FJOURNAL = {American Journal of Mathematics},
    VOLUME = {103},
      YEAR = {1981},
    NUMBER = {2},
     PAGES = {297--355}
}

@book {shahidiES,
    AUTHOR = {Shahidi, Freydoon},
     TITLE = {Eisenstein series and automorphic {$L$}-functions},
    SERIES = {American Mathematical Society Colloquium Publications},
    VOLUME = {58},
 PUBLISHER = {American Mathematical Society, Providence, RI},
      YEAR = {2010},
}

@phdthesis{shang,
  author       = {Shang, Zhengyuan}, 
  title        = {Special {$L$}-values and motivic extensions via {E}isenstein series on {$GL_2(F)$}},
  school       = {Princeton University},
  year         = {2026},
  url          = {https://www.proquest.com/dissertations-theses/special-em-l-values-motivic-extensions-via/docview/3347410275/se-2}
}

@book {voganbook,
    AUTHOR = {Vogan, Jr., David A.},
     TITLE = {Representations of real reductive {L}ie groups},
    SERIES = {Progress in Mathematics},
    VOLUME = {15},
 PUBLISHER = {Birkh\"{a}user, Boston, Mass.},
      YEAR = {1981}
}

@article {VZ,
    AUTHOR = {Vogan, Jr., David A. and Zuckerman, Gregg J.},
     TITLE = {Unitary representations with nonzero cohomology},
   JOURNAL = {Compositio Math.},
  FJOURNAL = {Compositio Mathematica},
    VOLUME = {53},
      YEAR = {1984},
    NUMBER = {1},
     PAGES = {51--90}
}
\end{document}